\documentclass[11pt]{article}

\usepackage{microtype}
\usepackage{graphicx}
\usepackage{subfigure}
\usepackage{booktabs} 
\usepackage{natbib}
\usepackage{epsfig,epsf,fancybox}
\usepackage{amsmath}
\usepackage{mathrsfs}
\usepackage{amssymb}
\usepackage{color}
\usepackage{multirow}
\usepackage{paralist}
\usepackage{verbatim}
\usepackage{algorithm}
\usepackage{algorithmic}
\usepackage{galois}
\usepackage{boxedminipage}
\usepackage{accents}
\usepackage{stmaryrd}
\usepackage[bottom]{footmisc}
\usepackage{hhline}

\usepackage{natbib}

\usepackage{url}
\usepackage[colorlinks,linkcolor=magenta,citecolor=blue, pagebackref=true,backref=true]{hyperref}
\renewcommand*{\backrefalt}[4]{%
    \ifcase #1 \footnotesize{(Not cited.)}%
    \or        \footnotesize{(Cited on page~#2.)}%
    \else      \footnotesize{(Cited on pages~#2.)}%
    \fi}

\newtheorem{theorem}{Theorem}[section]
\newtheorem{corollary}[theorem]{Corollary}
\newtheorem{lemma}[theorem]{Lemma}

\newtheorem{assumption}[theorem]{Assumption}

\usepackage{hyperref}

\usepackage[framemethod=default]{mdframed} 

\def\beq{\begin{equation} }
\def\eeq{\end{equation} }
\newcommand{\reals}{\mathbb{R}}
\newcommand{\Xtrap}[1]{^{\textnormal{md}}_{#1}}
\newcommand{\avg}[1]{^{\textnormal{ag}}_{#1}}
\newcommand{\strcvx}{\mu_{\operatorname{Str}}}
\newcommand{\stepsize}{\eta}
\def\oholder{\boldsymbol{\omega}}
\newcommand{\smoothFG}{L_{\operatorname{Str}}}
\newcommand{\smoothBLin}{L_{\operatorname{Bil}}}
\def\ratio{\mathcal{R}}
\def\Cholder{\mathscr{C}}
\def\epoch{s}
\def\Sholder{\mathscr{S}}
\def\Tholder{\mathscr{T}}

\newcommand{\Exs}{\ensuremath{{\mathbb{E}}}}
\def\stdFG{\sigma_{\operatorname{Str}}}
\def\stdBLin{\sigma_{\operatorname{Bil}}}
\def\quantile{r}
\def\distancemetric{\mathcal{S}}
\def\thetaholder{\boldsymbol{\theta}}
\def\foneholder{\boldsymbol{\varphi}_1}
\def\ftwoholder{\boldsymbol{\varphi}_2}
\def\boneholder{\boldsymbol{\delta}_1}
\def\btwoholder{\boldsymbol{\delta}_2}
\def\arbholder{\mathbf{z}}
\newcommand{\Vquantity}{V}
\newcommand{\noiseFG}{\boldsymbol{\Delta}_{\operatorname{Str}}}
\newcommand{\noiseBLin}{\boldsymbol{\Delta}_{\operatorname{Bil}}}

\def\calT{\mathcal{T}}

\usepackage{enumitem}
\usepackage[dvipsnames]{xcolor}
\def\blue#1{\textcolor{black}{#1}}

\def\red#1{}\def\pb{}

\begin{document}


\begin{center}

{\bf{\LARGE{Optimal Extragradient-Based Stochastic Bilinearly-Coupled Saddle-Point Optimization}}}

\vspace*{.2in}
{\large{
\begin{tabular}{cccc}
Simon S. Du$^\star$
\quad
Gauthier Gidel$^\dagger$
\quad
Michael I.~Jordan$^{\diamond, \ddagger}$
\quad
Chris Junchi Li$^\diamond$
\\
\end{tabular}
}}

\vspace*{.2in}

\begin{tabular}{c}
Paul G.~Allen School of Computer Science and Engineering, University of Washington$^\star$
\\
DIRO, Université de Montréal and Mila$^\dagger$
\\
Department of Electrical Engineering and Computer Sciences, UC Berkeley$^\diamond$
\\
Department of Statistics, UC Berkeley$^\ddagger$
\end{tabular}

\vspace*{.2in}

\today

\vspace*{.2in}

\let\thefootnote\relax\footnotetext{Authorship in alphabetical orders.}

\begin{abstract} 
We consider the smooth convex-concave bilinearly-coupled saddle-point problem, $\min_{\mathbf{x}}\max_{\mathbf{y}}~F(\mathbf{x}) + H(\mathbf{x},\mathbf{y}) - G(\mathbf{y})$, where one has access to stochastic first-order oracles for $F$, $G$ as well as the bilinear coupling function $H$. Building upon standard stochastic extragradient analysis for variational inequalities, we present a stochastic \emph{accelerated gradient-extragradient (AG-EG)} descent-ascent algorithm that combines extragradient and Nesterov's acceleration in general stochastic settings.  This algorithm leverages scheduled restarting to admit a fine-grained nonasymptotic convergence rate that matches known lower bounds by both \citet{ibrahim2020linear} and \citet{zhang2021lower} in their corresponding settings, plus an additional statistical error term for bounded stochastic noise that is optimal up to a constant prefactor. This is the first result that achieves such a relatively mature characterization of optimality in saddle-point optimization.
\end{abstract}
\end{center}


\pb\section{Introduction}
In this work, we focus on a widely studied stochastic convex-concave minimax optimization problem with bilinear coupling, also known as convex-concave bilinearly-coupled saddle-point problem:
\beq\label{problemopt_stochastic}
\min_{\mathbf{x}\in\reals^n}\max_{\mathbf{y}\in\reals^m}~
\mathscr{F}(\mathbf{x},\mathbf{y})
	=
\Exs_\xi\left[f(\mathbf{x}; \xi)\right]
+
\Exs_\zeta\left[h(\mathbf{x}, \mathbf{y}; \zeta)\right]
-
\Exs_\xi\left[g(\mathbf{y}; \xi)\right]
	\equiv
F(\mathbf{x})
+
H(\mathbf{x},\mathbf{y})
-
G(\mathbf{y}),
\eeq
where $H(\mathbf{x},\mathbf{y})
\equiv
\mathbf{x}^\top \mathbf{B}\mathbf{y}
-
\mathbf{x}^\top \mathbf{u}_{\mathbf{x}} 
+
\mathbf{u}_{\mathbf{y}}^\top \mathbf{y}
$ is the bilinear coupling function with the coupling matrix $\mathbf{B}$ of dimension $n\times m$, and where $\xi$ and $\zeta$ are drawn from distributions $\mathcal{D}_\xi$ and $\mathcal{D}_\zeta$, respectively.
We aim to solve \eqref{problemopt_stochastic} when either both $F(\mathbf{x})$ and $G(\mathbf{y})$ are smooth and strongly convex, or both are zero.
In addition to a wide range of applications in economics, problems of form~\eqref{problemopt_stochastic} are becoming increasingly important in machine learning.
For instance~\eqref{problemopt_stochastic} appears in reinforcement learning, differentiable games, regularized empirical risk minimization, and robust optimization formulations. It also can be seen as a local approximation of the objective of nonconvex-nonconcave minimax games---e.g., a GAN---around a local Nash equilibrium~\citep{mescheder2017numerics,nagarajan2017gradient}.
Our exposition begins with an overview of some of these applications.

\paragraph{Reinforcement learning.}
Reinforcement learning problems can be formalized as Markov Decision Processes (MDPs) where, at each step $t=1,\dots,n$, the learner receives a four-element tuple, $\{s_t, a_t, r_t, s_{t + 1}\}$, where $(s_t, a_t)$ is the current state-action pair, $r_t$ is the reward received upon choosing $a_t$, and $s_{t+1}$ is the next state drawn from a transition distribution.
For example, policy evaluation with a linear function approximator can be formalized in terms of the minimization of the \emph{mean squared projected Bellman-Error} (MSPBE)~\citep{du2017stochastic} based on a set of tuples:
\begin{align}
\min_{\boldsymbol{\theta}}~
\frac12\left\|\mathbf{A}\boldsymbol{\theta} - \mathbf{b}\right\|^2_{\mathbf{C}^{-1}}
+
\frac{\rho}{2}\left\|\boldsymbol{\theta}\right\|^2
,
\label{eq:MSPBE}
\end{align}
\looseness=-1
where $
\mathbf{A}=\frac{1}{n}\sum_{t=1}^n \boldsymbol{\phi}(s_t)(\boldsymbol{\phi}(s_t) - \gamma \boldsymbol{\phi}(s_{t+1}))^\top,\, 
\mathbf{b}	=\frac{1}{n}\sum_{t=1}^n r_t \boldsymbol{\phi}(s_t)
$, and $
\mathbf{C}	=	\frac{1}{n}\sum_{t=1}^n \boldsymbol{\phi}(s_t)\boldsymbol{\phi}(s_t)^\top
$ for a given feature mapping $\phi$.
To reduce the computational cost incurred by calculating the inverse of matrix $\mathbf{C}$,~\citet{du2017stochastic} propose an alternative min-max form of \eqref{eq:MSPBE}:
\begin{align*}
\min_{\boldsymbol{\theta}}\max_{\mathbf{w}}~
\frac{\rho}{2}\|\boldsymbol{\theta}\|^2 - \mathbf{w}^\top \mathbf{A}\boldsymbol{\theta} - \frac{1}{2}\|\mathbf{w}\|_{\mathbf{C}}^2 + \mathbf{w}^\top \mathbf{b}
,
\end{align*}
which falls under the umbrella of problem~\eqref{problemopt_stochastic} whenever $\mathbf{C}$ is positive definite.

\paragraph{Quadratic games.}
Another class of examples arises in the setting of bilinear games, where the minimax objective is:
\beq\label{quad_game}
\mathscr{F}(\mathbf{x},\mathbf{y})
=
\frac12 \mathbf{x}^\top \mathbf{M}_F \mathbf{x} 
+ 
\mathbf{x}^\top \mathbf{B}\mathbf{y}
- 
\frac12 \mathbf{y}^\top \mathbf{M}_G \mathbf{y} - \mathbf{x}^\top \mathbf{v}_{\mathbf{x}}	+	\mathbf{v}_{\mathbf{y}}^\top \mathbf{y}
,
\eeq
where $\mathbf{M}_F, \mathbf{M}_G$ are real-valued matrices of dimensions $n\times n$ and $m\times m$.
This has the form~\eqref{problemopt_stochastic} with $
F(\mathbf{x})	=	\frac12 \mathbf{x}^\top \mathbf{M}_F \mathbf{y}	-	\mathbf{x}^\top \mathbf{v}_{\mathbf{x}}
$, $
G(\mathbf{y})	=	\frac12 \mathbf{y}^\top \mathbf{M}_G \mathbf{y}	-	\mathbf{v}_{\mathbf{y}}^\top \mathbf{y}
$ and $
H(\mathbf{x},\mathbf{y})
\equiv
\mathbf{x}^\top \mathbf{B}\mathbf{y}
$.
A particular case we will be considering in \S\ref{sec_AG_EG_BLin} is the case of bilinear games, i.e.~where there are no quadratic terms.
We provide a detailed analysis of the nonasymptotic convergence in this setting in \S\ref{sec_AG_EG_BLin} and show that the upper bound on the convergence rate given by our algorithm matches the lower bound of~\citet[Theorem 3]{ibrahim2020linear}.

\looseness=-1
\paragraph{Regularized empirical risk minimization.}
The problem of the minimization of the regularized empirical risk for convex losses and linear predictors is a core problem in classical supervised learning:
$$
\min_{\mathbf{x}\in \reals^d}~\mathcal{L}(\mathbf{A}\mathbf{x}) + F(\mathbf{x})
\equiv
\frac{1}{n} \sum_{i=1}^n \mathcal{L}_i(\mathbf{a}_i^\top \mathbf{x}) + F(\mathbf{x})
,
$$
where $\mathbf{A} = \left[\mathbf{a}_1,\ldots, \mathbf{a}_n\right]^\top\in\reals^{n\times d}$ consists of feature vectors $\{\mathbf{a}_i\}$, $\mathcal{L}_i(\mathbf{y})$ is a univariate convex loss for the $i$th data point, and $F(\mathbf{x})$ is a convex regularizer.
A standard construction turns this empirical risk minimization problem into a saddle-point problem as follows:
$$
\min_{\mathbf{x} \in \reals^d}\max_{\mathbf{y} \in \reals^m}~
F(\mathbf{x}) 
+ 
\mathbf{x}^\top\mathbf{A} \mathbf{y} 
- 
\underbrace{
\mathcal{L}^\star(\mathbf{y})
}_{\text{Legendre dual function of $\mathcal{L}(\mathbf{y})$}}
\equiv
F(\mathbf{x}) 
+ 
\frac{1}{n} \sum_{i=1}^n \mathbf{x}_i \mathbf{y}^\top \mathbf{a}_i 
- 
\frac{1}{n} \sum_{i=1}^n \mathcal{L}^\star(\mathbf{y}_i)
.
$$
See \citet{zhang2017stochastic,wang2017exploiting,xiao2019dscovr} for in-depth studies of solving this problem under such a dual form of representation.

\pb\subsection{Main contributions}
Despite the range of real-world applications of the bilinearly-coupled saddle-point problem in~\eqref{problemopt_stochastic}, there is a limited nonasymptotic theoretical analysis of the problem.
Notable exceptions include \citet{zhang2021complexity} and \citet{ibrahim2020linear}, who develop lower bounds in the strongly-convex-strongly-concave and bilinear settings.
\emph{Matching these lower bounds in a single algorithm in the general stochastic setting has been an open problem.}
In particular, standard acceleration techniques do not achieve the optimal nonasymptotic convergence rate for the bilinear minimax game~\citep{gidel2019negative}.

We tackle this problem in a new way, proposing a stochastic \emph{accelerated gradient-extragradient (AG-EG)} descent-ascent algorithm for solving \eqref{problemopt_stochastic}, bringing together Nesterov’s acceleration method \citep{nesterov1983method}---applied to the individual $F(\mathbf{x})$ and $G(\mathbf{y})$ terms---and the extragradient method \citep{korpelevich1976extragradient}---which extrapolates the bilinear coupling term.
This combination allows us to arrive at a general algorithmic convergence result that yields optimality in nonasymptotic convergence rates for the strongly-convex-strongly-concave and bilinear settings.
This general result subsumes many special cases of interest:

\begin{itemize}[leftmargin=5mm]
\item
For the function class of bilinear games where $\nabla f(\mathbf{x}; \xi) = \mathbf{0}$ and $\nabla g(\mathbf{y}; \xi) = \mathbf{0}$ a.s., Algorithm \ref{algo_AG_EG}, equipped with scheduled restarting achieves an $O\bigg(
\sqrt{
\tfrac{\lambda_{\max}(\mathbf{B}^\top \mathbf{B})}{\lambda_{\min}(\mathbf{B}\mathbf{B}^\top)}
} \log\left(
\tfrac{
\sqrt[4]{\lambda_{\min}(\mathbf{B}\mathbf{B}^\top)\lambda_{\max}(\mathbf{B}^\top \mathbf{B})}
}{\stdBLin}
\right)
+
\tfrac{\stdBLin^2}{\lambda_{\min}(\mathbf{B}\mathbf{B}^\top)\varepsilon^2}
\bigg)$ iteration complexity, where $\stdBLin$ is the variance of the stochastic gradient (bilinear coupling term).
When there is no randomness, this complexity result reduces to $
\mathcal{O}\left(
\sqrt{\frac{\lambda_{\max}(\mathbf{B}^\top \mathbf{B})}{\lambda_{\min}(\mathbf{B}\mathbf{B}^\top)}} \log\left(\frac{1}{\varepsilon}\right)
\right)
$ for the bilinear problem, matching the lower bound of \citet{ibrahim2020linear}.%
\footnote{As will be discussed in Assumption \ref{assu_coupling} of \S\ref{sec_settings_assumptions}, we can assume our coupling matrix $\mathbf{B}$ is tall in the sense that $n\ge m$ without loss of generality. For the function class of bilinear games, we assume that $n=m$ where $\mathbf{B}$ is a nonsingular square matrix, so the complexity to make sense. See \S\ref{sec_AG_EG_BLin} for more on this.}
In other words, our algorithm admits a sharp dependency on $\lambda_{\min}(\mathbf{B}\mathbf{B}^\top)$ and matches the \citet{ibrahim2020linear} lower bound [\S\ref{sec_AG_EG_BLin}, Corollary \ref{theo_S_AG_EG_Bilinear_restart}].

\item
For the function class of strongly-convex-strongly-concave objectives, the same stochastic AG-EG descent-ascent Algorithm \ref{algo_AG_EG}, when equipped with scheduled restarting, achieves an iteration complexity of $O\left(
\left(
\sqrt{\tfrac{L_F}{\mu_F}\lor \tfrac{L_G}{\mu_G}} 
+ 
\sqrt{\tfrac{\lambda_{\max}(\mathbf{B}^\top \mathbf{B})}{\mu_F\mu_G}}
\right)\log\left(\frac{1}{\varepsilon}\right)
+
\tfrac{\sigma^2}{\mu_F^2\varepsilon^2}
\right)$, where $F: \reals^n \rightarrow \reals$ is $L_F$-smooth and $\mu_F$-strongly convex, $G: \reals^m \rightarrow \reals$ is $L_G$-smooth and $\mu_G$-strongly convex, and $\sigma$ is a uniformly weighted variance of the stochastic gradient.
When the problem is nonrandom, this complexity upper bound matches the \citet{zhang2021complexity} lower bound [\S\ref{sec_AG_EG}, Corollary \ref{theo_S_AG_EG_restart}].

\item
We also present a direct approach for the function class of strongly-convex-strongly-concave objectives, where the lower bound in iteration complexity due to \citet{zhang2021complexity} is matched as $\left(
\sqrt{\tfrac{L_F}{\mu_F} \lor \tfrac{L_G}{\mu_G}} 
+ 
\sqrt{\tfrac{\lambda_{\max}(\mathbf{B}^\top \mathbf{B})}{\mu_F\mu_G}}
+
O\left(
\tfrac{\sigma^2}{\mu_F^2\varepsilon^2}
\right)\right) \log\left((\tfrac{L_F}{\mu_F} \lor \tfrac{L_G}{\mu_G})\tfrac{1}{\varepsilon}\right)
$, an iteration complexity that admits a near-unity sharp coefficient [\S\ref{sec_AG_EG_Str}, Theorem \ref{theo_AG_EG_Str}].
\end{itemize}

Throughout our analysis, we frequently make use of a \emph{scheduled-restarting} approach and a \emph{scaling-reduction} argument that allows us to reduce problems to cases that are relatively easier to analyze.  This general strategy may be of independent interest.

\pb\subsection{Related work}\label{sec_comp}
Here we compare our results with related work on saddle-point (minimax) optimization in machine learning and optimization literature.

\paragraph{Bilinear game case, nonstochastic setting.}
In the bilinear game case where $L_F = \mu_F = L_G = \mu_G = 0$, a lower bound has been established by \cite{ibrahim2020linear}: $
\Omega\bigg(
\sqrt{\frac{\lambda_{\max}(\mathbf{B}^\top \mathbf{B})}{\lambda_{\min}(\mathbf{B}\mathbf{B}^\top)}} \log\left(\frac{1}{\varepsilon}\right)
\bigg)
$.
The study of bilinear example has been initiated by \citet{daskalakis2018training} for understanding saddle-point optimization. They proposed the gradient descent-ascent (OGDA) algorithm and achieved sublinear convergence.
Subsequently, the classical methods of ExtraGradient (EG) and Optimistic Gradient Descent Ascent (OGDA) algorithms were proven to have linear convergence rate for strongly monotone and Lipschitz operator with $O\bigg(
\frac{\lambda_{\max}(\mathbf{B}^\top \mathbf{B})}{\lambda_{\min}(\mathbf{B}\mathbf{B}^\top)}\log(\frac{1}{\varepsilon})
\bigg)$ iteration complexity~\citep{gidel2019negative,mokhtari2020unified}. 
\citet{azizian2020tight} proved in another study that by considering first order methods using a fixed number of composed gradient evaluations and only the last iteration (this class of methods is called 1-SCLI and excludes momentum and restarting), the $O\bigg(
\frac{\lambda_{\max}(\mathbf{B}^\top \mathbf{B})}{\lambda_{\min}(\mathbf{B}\mathbf{B}^\top)}\log(\frac{1}{\varepsilon})
\bigg)$ iteration complexity for EG is optimal.
In the absence of strong monotonicity assumptions, \citet{loizou2020stochastic} generated the first set of global non-asymptotic last-iterate convergence guarantees for a stochastic game over a non-compact domain using a Hamiltonian viewpoint.
In particular, the proposed stochastic Hamiltonian gradient method ensures convergence in the finite-sum stochastic bilinear game as well.
In very recent work, when restricted to the bilinear minimax optimization, \citet{kovalev2021accelerated} derive an iteration complexity that is essentially $
\mathcal{O}\bigg(\frac{\lambda_{\max}(\mathbf{B}^\top \mathbf{B})}{\lambda_{\min}(\mathbf{B}\mathbf{B}^\top)} \log(\frac{1}{\varepsilon})\bigg)
$.  This is comparable to the rates in \citet{daskalakis2018training,liang2019interaction,gidel2019negative,mokhtari2020unified,mishchenko2020revisiting}.
For matching the $\mathcal{O}\bigg(
\sqrt{\frac{\lambda_{\max}(\mathbf{B}^\top \mathbf{B})}{\lambda_{\min}(\mathbf{B}\mathbf{B}^\top)}} \log(\frac{1}{\varepsilon})
\bigg)$ lower bound provided by \citet{ibrahim2020linear}, the work of~\citet{azizian2020accelerating} considered EG with momentum. They used a perturbed spectral analysis encompassing Polyak momentum.
Nonetheless,~\citet{azizian2020accelerating} only provide accelerated rates in the regime where the condition number is large. 
\citet{li2021convergence} was the first to show that a version of stochastic extragradient method converges at an accelerated convergence rates for bilinear games with unbounded domain and unbounded stochastic noise using restarted iteration averaging, and when focusing on the nonstochastic setting, matches the lower bound of~\citep{ibrahim2020linear}.

\paragraph{Smooth strongly convex-concave case, nonstochastic setting.}
A lower bound for smooth strongly convex-concave minimax optimization has been recently established by \cite{zhang2021lower}.
This lower bound is of the form $
\Omega\bigg(\bigg(
\sqrt{\frac{L_F}{\mu_F}\lor\frac{L_G}{\mu_G}}
+
\sqrt{\frac{\lambda_{\max}(\mathbf{B}^\top \mathbf{B})}{\mu_F\mu_G}}
\bigg) \log\left(\frac{1}{\varepsilon}\right)
\bigg)
$.
As for upper bounds, earlier extragradient-based methods \citet{tseng1995linear} and accelerated dual extrapolation algorithm \cite{nesterov2011solving} achieve, when translated to our bilinearly-coupled problem, an iteration complexity of $
\tilde{\mathcal{O}}\bigg(
\frac{L_F}{\mu_F}\lor\frac{L_G}{\mu_G}
+
\sqrt{\frac{\lambda_{\max}(\mathbf{B}^\top \mathbf{B})}{\mu_F\mu_G}}
\bigg)
$.
The same complexity has also been matched by \citet{gidel2019variational}, \citet{mokhtari2020unified}, \citet{cohen2021relative} from a relative Lipschitz viewpoint.%
\footnote{\citet{mokhtari2020unified} report a $
\tilde{\mathcal{O}}\bigg(
\frac{L_F\lor L_G + \sqrt{\lambda_{\max}(\mathbf{B}^\top \mathbf{B})}}{\mu_F\land\mu_G}
\bigg)
$ complexity, but the mentioned complexity can be obtained via a scaling-reduction argument: consider $\mu_F = \mu_G$ case first, then consider the general case by rescaling the $y$ variable by a factor of $\sqrt{\frac{\mu_G}{\mu_F}}$.}
Improving upon this result, \citet{lin2020near} achieve a complexity of
$
\tilde{\mathcal{O}}\bigg(
\sqrt{
\frac{L_FL_G}{\mu_F\mu_G}
}
+
\sqrt{
\frac{\lambda_{\max}(\mathbf{B}^\top \mathbf{B})}{\mu_F \mu_G}
}
\bigg)
$ using proper acceleration methods, when restricted to the bilinearly-coupled problem.
\cite{wang2020improved} achieves%
\footnote{Note the cross term here cannot be absorbed into the summation of the remaining terms.}
$
\tilde{\mathcal{O}}\bigg(
\sqrt{\frac{L_F}{\mu_F}\lor\frac{L_G}{\mu_G}}
+
\sqrt{
\frac{
\sqrt{\lambda_{\max}(\mathbf{B}^\top \mathbf{B}) L_FL_G}	+	\lambda_{\max}(\mathbf{B}^\top \mathbf{B}) 
}{\mu_F \mu_G}
}
\bigg)
$
and a Hermitian-skew-based analysis nearly matches \citet{zhang2021lower} for the quadratic minimax game case. For the same problem, \cite{xie2021dippa} achieves a complexity of
$
\tilde{\mathcal{O}}\bigg(
\sqrt[4]{\frac{L_FL_G}{\mu_F\mu_G}\left(\frac{L_F}{\mu_F}\lor\frac{L_G}{\mu_G}\right)}
+
\sqrt{\frac{\lambda_{\max}(\mathbf{B}^\top \mathbf{B})}{\mu_F \mu_G}}
\bigg)
$.
These works improve upon \citet{lin2020near} in a fine-grained fashion.
In early 2022, three concurrent works \citet{kovalev2021accelerated,thekumparampil2022lifted,jin2022sharper} studies the nonstochastic problem and independently match the lower bound by \citet{zhang2021lower}.
The main novelty of this work is that both lower bounds \citet{ibrahim2020linear} and \citet{zhang2021lower} are achieved in one single algorithm, plus an optimal statistical error term up to a constant prefactor in the stochastic setting.

\begin{table}[!tb]
\centering
\begin{tabular}{lc}
\toprule
\textbf{References}
&
\textbf{Iteration Complexity}
\\ \hhline{==}
\citet{mokhtari2020unified,cohen2021relative}
&
$
\frac{L_F}{\mu_F}\lor\frac{L_G}{\mu_G}
+
\sqrt{\frac{\lambda_{\max}(\mathbf{B}^\top \mathbf{B})}{\mu_F\mu_G}}
$
\\
\citet{lin2020near}
&
$
\sqrt{
\frac{L_FL_G}{\mu_F\mu_G}
}
+
\sqrt{
\frac{\lambda_{\max}(\mathbf{B}^\top \mathbf{B})}{\mu_F \mu_G}
}
$
\\
\citet{wang2020improved}
&
$
\sqrt{\frac{L_F}{\mu_F}\lor\frac{L_G}{\mu_G}}
+
\sqrt{
\frac{
\sqrt{\lambda_{\max}(\mathbf{B}^\top \mathbf{B}) L_FL_G}	+	\lambda_{\max}(\mathbf{B}^\top \mathbf{B}) 
}{\mu_F \mu_G}
}
$
\\
\citet{xie2021dippa}
&
$
\sqrt[4]{\frac{L_FL_G}{\mu_F\mu_G}\left(\frac{L_F}{\mu_F}\lor\frac{L_G}{\mu_G}\right)}
+
\sqrt{\frac{\lambda_{\max}(\mathbf{B}^\top \mathbf{B})}{\mu_F \mu_G}}
$
\\
\begin{tabular}{@{}l@{}}
\citet{kovalev2021accelerated} and concurrently
\\
\citet{thekumparampil2022lifted,jin2022sharper}
\end{tabular}
&
$
\sqrt{\frac{L_F}{\mu_F}\lor\frac{L_G}{\mu_G}}
+
\sqrt{\frac{\lambda_{\max}(\mathbf{B}^\top \mathbf{B})}{\mu_F\mu_G}}
$
\\
AG-EG (this work), Theorems \ref{theo_AG_EG} \& \ref{theo_AG_EG_Str}
&
$
\sqrt{\frac{L_F}{\mu_F}\lor\frac{L_G}{\mu_G}}
+
\sqrt{\frac{\lambda_{\max}(\mathbf{B}^\top \mathbf{B})}{\mu_F\mu_G}}
$
\\ \midrule
\citet{zhang2021lower} (Lower bound)
&
$
\Omega\left(\left(
\sqrt{\frac{L_F}{\mu_F}\lor\frac{L_G}{\mu_G}}
+
\sqrt{\frac{\lambda_{\max}(\mathbf{B}^\top \mathbf{B})}{\mu_F\mu_G}}
\right) \log\left(\frac{1}{\varepsilon}\right)
\right)
$
\\ \hhline{==}
\citet{gidel2019negative} among other work
&
$
\tfrac{\lambda_{\max}(\mathbf{B}^\top \mathbf{B})}{\lambda_{\min}(\mathbf{B}\mathbf{B}^\top)}
$
\\
\citet{azizian2020accelerating,li2021convergence}
&
$
\sqrt{
\tfrac{\lambda_{\max}(\mathbf{B}^\top \mathbf{B})}{\lambda_{\min}(\mathbf{B}\mathbf{B}^\top)}
}
$
\\
AG-EG (this work), Corollary \ref{theo_S_AG_EG_Bilinear_restart}
&
$
\sqrt{
\tfrac{\lambda_{\max}(\mathbf{B}^\top \mathbf{B})}{\lambda_{\min}(\mathbf{B}\mathbf{B}^\top)}
}
$
\\ \midrule
\citet{ibrahim2020linear} (Lower bound)
&
$
\Omega\left(
\sqrt{\frac{\lambda_{\max}(\mathbf{B}^\top \mathbf{B})}{\lambda_{\min}(\mathbf{B}\mathbf{B}^\top)}} \log\left(\frac{1}{\varepsilon}\right)
\right)
$
\\ \bottomrule
\end{tabular}
\caption{Table of comparison with related work for both strongly case and bilinear case, concentrating on the nonstochastic setting. For upper bounds, a polylogarithmic prefactor is ignored.}%
\label{tab_comp}
\end{table}

\paragraph{Stochastic setting.}
Stochastic minimax optimization has been studied intensively as a special case of the variational inequalities.
It is widely accepted in classical literature on stochastic variational inequality~\citep{nemirovski2009robust,juditsky2011solving} that the set of parameters and the variance of the stochastic estimate of the vector field are bounded.
\citet{chen2017accelerated} extended the analysis of \cite{juditsky2011solving} that accelerates the convergence rates for a class of variational inequalities. \citet{iusem2017extragradient} proposed an analysis of stochastic extragradient using large batches to reduce the variance. \citet{mertikopoulos2018optimistic} showed almost sure convergence of SEG to a strictly coherent solution (a.k.a.~star-strict monotone VIP). In a similar vein,~\citet{ryu2019ode} showed that SGDA with anchoring almost surely converge to strictlyconvex-concave saddle points.
\citet{fallah2020optimal} developed a multistage variant of stochastic gradient descent ascent and stochastic optimistic gradient descent ascent with constant learning rate decay schedule. We improve upon their rates since their iteration complexity depends on a significantly larger condition number than our method and is infinite in absence of strong convex-concavity. 
They achieved the optimal dependency on the noise variance but suboptimal dependency on the condition number.
\citet{hsieh2020explore} developed a double stepsize extragradient method and proved the last-iterate convergence rates under an error bound condition similar to star-strong monotonicity. 
\citet{kotsalis2020simple} proposed a simple and optimal scheme for a class of generalized strongly monotone (stochastic) variational inequalities.
Due to the unconstrained nature of stochastic bilinear models, these two assumptions do not hold in this case because the noise increases with the value of the parameters.
In recent work,~\citet{mishchenko2020revisiting} has shown that stochastic extragradients can be computed under a different step size, which removes the bounded domain assumption, while still requiring the bounded noise assumption. The work also discussed the advantages of using the same mini-batch for the two gradients in stochastic extragradients.
In another vein,~\citet{jelassi2020extra} focuses on stochastic extragradient in games with a large number of players. In that case they propose an extragradient algorithm that randomly update a small subset of the players at each iterations.

\paragraph{Organization.}
The rest of this work is organized as follows.
\S\ref{sec_settings_assumptions} presents the basic settings and assumptions.
\S\ref{sec_AG_EG_BLin} gives the optimality of convergence for our proposed AcceleratedGradient-Extragradient (AG-EG) descent-ascent algorithm, for the class of bilinear games, and \S\ref{sec_AG_EG} presents the optimality of AG-EG for the class of strongly-convex-strongly-concave objectives.
\S\ref{sec_AG_EG_Str} provides an alternative direct approach for the same strongly-convex-strongly-concave function class.
\S\ref{sec_discuss} discusses future directions.
In the Appendix, \S\ref{sec_proof} details the proofs of our main convergence results, and \S\ref{sec_proofaux} supplements the proofs with auxiliary lemmas.

\paragraph{Notations.}
Let $\lambda_{\max}(\mathbf{M})$ (resp.~$\lambda_{\min}(\mathbf{M})$ be the largest (resp.~smallest) eigenvalue of a real symmetric matrix $\mathbf{M}$. 
Let $a\lor b \equiv \max(a,b)$ (resp.~$a\land b \equiv \min(a,b)$) denote the maximum (resp.~minimum) value of two reals $a,b$.
For two nonnegative real sequences $(a_n)$ and $(b_n)$, we write $a_n = O(b_n)$ or $a_n\lesssim b_n$ (resp.~$a_n = \Omega(b_n)$ or $a_n\gtrsim b_n$) to denote $a_n\le Cb_n$ (resp.~$a_n\ge Cb_n$) for all $n\ge 1$ for a positive, numerical constant $C$, and let $a_n\asymp b_n$ if both $a_n\lesssim b_n$ and $a_n\gtrsim b_n$ hold.
We also let $a_n = \tilde{O}\left(b_{n}\right)$ denote $a_n\le C b_n$ where $C$ hides a polylogarithmic factor in problem-dependent constants, and let $[\mathbf{x}; \mathbf{y}]\in \reals^{n+m}$ concatenate two vectors $\mathbf{x}\in \reals^n$ and $\mathbf{y}\in \reals^m$.
Finally for two real symmetric matrices $\mathbf{A}$ and $\mathbf{B}$, we denote $\mathbf{A}\preceq \mathbf{B}$ (resp.~$\mathbf{A}\succeq \mathbf{B}$) when $\mathbf{v}^\top (\mathbf{A}-\mathbf{B}) \mathbf{v}\le 0$ (resp.~$\mathbf{v}^\top (\mathbf{A} - \mathbf{B}) \mathbf{v}\ge 0$) holds for all vectors $\mathbf{v}$.

\pb\section{Setting and assumptions}\label{sec_settings_assumptions}
In this section, we formally introduce our framework and assumptions.
Our development is inspired by the work of \citet{chen2017accelerated} on a (stochastic) \emph{Accelerated MirrorProx (AMP)} algorithm.
This work is developed in the general setting of monotone variational inequalities with a $O(1/\sqrt{T})$ convergence rate bound where the prefactor depends on domain size and hence does not accommodate unbounded domains.
As a result when translated directly into minimax optimization, this result does not match the lower bound in \citep{zhang2021lower}.
To achieve the lower bound, we present an alternative approach in Algorithm \ref{algo_AG_EG}, the stochastic \emph{accelerated gradient-extragradient (AG-EG)} descent-ascent algorithm.
Our algorithm applies Nesterov’s acceleration method \citep{nesterov1983method} to the individual $F(\mathbf{x})$ and $G(\mathbf{y})$ terms and applies the extragradient method \citep{korpelevich1976extragradient} to the bilinear coupling part.
As we show, a particular combination---with the incorporation of scheduled restarting---succeeds at matching the lower bound provided in \citet{ibrahim2020linear} and \citet{zhang2021lower} in their corresponding settings.

\begin{algorithm}[!t]
\caption{Stochastic AcceleratedGradient-ExtraGradient (AG-EG) Descent-Ascent Algorithm, with Scheduled Restarting}
\begin{algorithmic}[1]
\REQUIRE 
Initialization $\mathbf{x}_0^{[0]}, \mathbf{y}_0^{[0]}$, total number of epoches $\Sholder\ge 1$, total number of per-epoch iterates $(\Tholder_{\epoch}: \epoch=1,\dots,\Sholder)$, step sizes $(\alpha_t, \stepsize_t: t=1,2,\dots)$, ratio of strong-convexity parameters $\ratio = \frac{\mu_G}{\mu_F}$
\FOR{$\epoch=1, 2,\dots,\Sholder$}
\STATE
Set $
\mathbf{x}\avg{-\frac12}\leftarrow \mathbf{x}_0^{[\epoch-1]}
$, $
\mathbf{y}\avg{-\frac12}\leftarrow \mathbf{y}_0^{[\epoch-1]}
$, $
\mathbf{x}_0\leftarrow \mathbf{x}_0^{[\epoch-1]}
$, $
\mathbf{y}_0\leftarrow \mathbf{y}_0^{[\epoch-1]}
$, $
\mathbf{x}\Xtrap{0}\leftarrow \mathbf{x}_0^{[\epoch-1]}
$, $
\mathbf{y}\Xtrap{0}\leftarrow \mathbf{y}_0^{[\epoch-1]}
$
\FOR{$t=1, 2,\dots,\Tholder_{\epoch}$}
\STATE
Draw samples $\xi_{t-\frac12}\sim \mathcal{D}_\xi$ from oracle, and also $\zeta_{t-\frac12}, \zeta_t\sim \mathcal{D}_\zeta$ independently from oracle
\STATE\label{line_a}
$
\mathbf{x}_{t-\frac12}	\leftarrow	\mathbf{x}_{t-1} - \stepsize_t\left(
\nabla f(\mathbf{x}\Xtrap{t-1}; \xi_{t-\frac12})
+
\nabla_{\mathbf{x}} h(\mathbf{x}_{t-1}, \mathbf{y}_{t-1}; \zeta_{t-\frac12})
\right)
$
\STATE\label{line_b}
$
\mathbf{y}_{t-\frac12}	\leftarrow	\mathbf{y}_{t-1} - \tfrac{\stepsize_t}{\ratio}\left(
-\nabla_{\mathbf{y}} h(\mathbf{x}_{t-1}, \mathbf{y}_{t-1}; \zeta_{t-\frac12})
+
\nabla g(\mathbf{y}\Xtrap{t-1}; \xi_{t-\frac12})
\right)
$
\STATE\label{line_c}
$
\mathbf{x}\avg{t-\frac12}	\leftarrow	(1-\alpha_t)\mathbf{x}\avg{t-\frac32} + \alpha_t\mathbf{x}_{t-\frac12}
$
\STATE\label{line_d}
$
\mathbf{y}\avg{t-\frac12}	\leftarrow	(1-\alpha_t)\mathbf{y}\avg{t-\frac32} + \alpha_t\mathbf{y}_{t-\frac12}
$
\STATE\label{line_e}
$
\mathbf{x}_t			\leftarrow	\mathbf{x}_{t-1} - \stepsize_t\left(
\nabla f(\mathbf{x}\Xtrap{t-1}; \xi_{t-\frac12})
+
\nabla_{\mathbf{x}} h(\mathbf{x}_{t-\frac12}, \mathbf{y}_{t-\frac12}; \zeta_t)
\right)
$
\STATE\label{line_f}
$
\mathbf{y}_t			\leftarrow	\mathbf{y}_{t-1} - \tfrac{\stepsize_t}{\ratio}\left(
-\nabla_{\mathbf{y}} h(\mathbf{x}_{t-\frac12}, \mathbf{y}_{t-\frac12}; \zeta_t)
+
\nabla g(\mathbf{y}\Xtrap{t-1}; \xi_{t-\frac12})
\right)
$
\STATE\label{line_g}
$
\mathbf{x}\Xtrap{t}		\leftarrow	(1-\alpha_{t+1})\mathbf{x}\avg{t-\frac12} + \alpha_{t+1}\mathbf{x}_t
$
\STATE\label{line_h}
$
\mathbf{y}\Xtrap{t}		\leftarrow	(1-\alpha_{t+1})\mathbf{y}\avg{t-\frac12} + \alpha_{t+1}\mathbf{y}_t
$
\ENDFOR
\STATE\label{lineoutput}
Set $
\mathbf{x}_0^{[\epoch]}\leftarrow \mathbf{x}\avg{\Tholder_{\epoch}-\frac12}
$, $
\mathbf{y}_0^{[\epoch]}\leftarrow \mathbf{y}\avg{\Tholder_{\epoch}-\frac12}
$
\hfill \text{//Warm-start using the output of the previous epoch}
\ENDFOR
\STATE {\bfseries Output:}
$[\mathbf{x}_0^{[\Sholder]}; \mathbf{y}_0^{[\Sholder]}]$
\end{algorithmic}
\label{algo_AG_EG}
\end{algorithm}

For simplicity, we consider unconstrained domains $\mathbf{x}\in \reals^n$ and $\mathbf{y}\in \reals^m$.
For the constrained case with convex domains one can introduce a projection step and proceed analogously with the analysis; we omit this generalization for simplicity.
We first state the smoothness and convexity assumptions that we impose on the $F(\mathbf{x})$ and $G(\mathbf{y})$ terms.
\begin{assumption}[Convexity and smoothness]\label{assu_cvxsmth}
We assume that $F(\mathbf{x})$ is $L_F$-smooth and $\mu_F$-strongly convex, and $G(\mathbf{y})$ is $L_G$-smooth and $\mu_G$-strongly convex.
That is, for any $\mathbf{x},\mathbf{x}'\in \reals^n$,
$$
\tfrac{\mu_F}{2}\|\mathbf{x} - \mathbf{x}'\|^2
\le
F(\mathbf{x}) - F(\mathbf{x}') - \nabla F(\mathbf{x}')^\top (\mathbf{x} - \mathbf{x}')
\le
\tfrac{L_F}{2}\|\mathbf{x} - \mathbf{x}'\|^2,
$$
and for any $\mathbf{y},\mathbf{y}'\in \reals^m$,
$$
\tfrac{\mu_G}{2}\|\mathbf{y} - \mathbf{y}'\|^2
\le
G(\mathbf{y}) - G(\mathbf{y}') - \nabla G(\mathbf{y}')^\top (\mathbf{y} - \mathbf{y}')
\le
\tfrac{L_G}{2}\|\mathbf{y} - \mathbf{y}'\|^2
.
$$
\end{assumption}

We assume that the coupling matrix $\mathbf{B}$ is a tall matrix, which can otherwise be satisfied by considering the symmetrized problem $\min_\mathbf{y}\max_\mathbf{x} -f(\mathbf{x},\mathbf{y})$ (an equivalence guaranteed by the strong convexity of the functions and Sion’s minimax theorem~\citep{sion1958general}.) 
\begin{assumption}[Coupling matrix]\label{assu_coupling}
We assume without loss of generality that $\mathbf{B}$ is tall, i.e., $n\ge m$.
\end{assumption}
Assumption \ref{assu_coupling}, which is introduced for the purpose of notational consistency, guarantees that $
\lambda_{\max}(\mathbf{B}^\top \mathbf{B}) = \lambda_{\max}(\mathbf{B}\mathbf{B}^\top)
$ but $
\lambda_{\min}(\mathbf{B}^\top \mathbf{B}) \ge \lambda_{\min}(\mathbf{B}\mathbf{B}^\top)
$, where the latter is strictly zero when $\mathbf{B}$ is nonsquare.

It is straightforward to show that~\eqref{problemopt_stochastic} admits a unique saddle point (or Nash equilibrium) in the strongly-convex-strongly-concave case [Assumption~\ref{assu_cvxsmth}]; i.e., there exists a unique pair $(\oholder_\mathbf{x}^\star, \oholder_\mathbf{y}^\star)$ such that
\begin{equation}\label{assu_saddle}
\mathscr{F}(\oholder_{\mathbf{x}}^\star, \mathbf{y})
\le
\mathscr{F}(\oholder_{\mathbf{x}}^\star, \oholder_{\mathbf{y}}^\star)
\le
\mathscr{F}(\mathbf{x}, \oholder_{\mathbf{y}}^\star)
,\qquad
\text{for all $\mathbf{x}\in\reals^n$ and $\mathbf{y}\in\reals^d$}
\,.
\tag{SP}
\end{equation}
For the bilinear game case where $L_F = \mu_F = 0$, $L_G = \mu_G = 0$, this is satisfied for square matrices $\mathbf{B}$ with least singular value being strictly positive.

Third, we impose assumptions on the noise variance bound.
We first introduce the following rescaling parameters:
\beq\label{params}
\smoothFG = L_F \lor \left( \tfrac{\mu_F}{\mu_G} L_G \right)
,\qquad
\smoothBLin = \sqrt{\lambda_{\max}(\mathbf{B}^\top \mathbf{B})\cdot \tfrac{\mu_F}{\mu_G}}
,\qquad
\strcvx = \mu_F
,\qquad
\ratio = \tfrac{\mu_G}{\mu_F}
.
\eeq

\begin{assumption}[Unbiased gradients and variance bounds]\label{assu_noise}
We assume that $\mathbf{x}\in \reals^n, \mathbf{y}\in \reals^m$, $\xi\sim \mathcal{D}_\xi$ and $\zeta\sim \mathcal{D}_\zeta$ are drawn from distributions such that the following conditions hold: $
\Exs_\xi[\nabla f(\mathbf{x}; \xi)] = \nabla F(\mathbf{x})
$, $
\Exs_\xi[\nabla g(\mathbf{y}; \xi)] = \nabla G(\mathbf{y})
$, $
\Exs_\zeta[\nabla_{\mathbf{x}} h(\mathbf{x}, \mathbf{y}; \zeta)]	=	\nabla_{\mathbf{x}} H(\mathbf{x}, \mathbf{y})
$ and $
\Exs_\zeta[\nabla_{\mathbf{y}} h(\mathbf{x}, \mathbf{y}; \zeta)]	=	\nabla_{\mathbf{y}} H(\mathbf{x}, \mathbf{y})
$, with
\begin{align}\label{noisevarone}
\begin{aligned}
\Exs_\xi\left[
\|\nabla f(\mathbf{x}; \xi)	-	\nabla F(\mathbf{x})\|^2
+
\tfrac{1}{\ratio}\|\nabla g(\mathbf{y}; \xi)	-	\nabla G(\mathbf{y})\|^2
\right]
\le
\stdFG^2
,
\end{aligned}\end{align}
and
\begin{align}\label{noisevartwo}
\begin{aligned}
\Exs_\zeta\left[
\|\nabla_{\mathbf{x}} h(\mathbf{x}, \mathbf{y}; \zeta)	-	\nabla_{\mathbf{x}} H(\mathbf{x}, \mathbf{y})\|^2
+
\tfrac{1}{\ratio}\|-\nabla_{\mathbf{y}} h(\mathbf{x}, \mathbf{y}; \zeta)	+	\nabla_{\mathbf{y}} H(\mathbf{x}, \mathbf{y})\|^2
\right]
\le
\stdBLin^2
.
\end{aligned}\end{align}
\end{assumption}

For all results in this work, we suppose that Assumptions \ref{assu_cvxsmth}, \ref{assu_coupling} and \ref{assu_noise} hold with appropriate parameter settings.
Given a desired accuracy $\varepsilon > 0$, our goal is to find an $\varepsilon$-saddle point $(\mathbf{x},\mathbf{y})$, where $\|\mathbf{x} - \oholder_{\mathbf{x}}^\star\|^2 + \ratio\|\mathbf{y} - \oholder_{\mathbf{y}}^\star\|^2\le \varepsilon^2$---for the purposes of our analysis we adopt this slightly different metric that is equivalent to Euclidean norm.
The resulting iteration complexities in the Euclidean norm are obtained by replacing the $\varepsilon$-desired accuracy for that metric by $\varepsilon/(\sqrt{\ratio\lor\ratio^{-1}})$.%
\footnote{The metric conversion from duality gap to Euclidean distance or weighted Euclidean distance to saddle in our case is straightforward: they are equivalent up to a multiplicative factor. In other words, although our algorithmic convergence is characterized by the metric of weighted Euclidean distance, it completely matches both lower bounds by \citet{zhang2021lower} and \citet{ibrahim2020linear} in the nonstochastic setting.}

\pb\section{Optimality for bilinear games}\label{sec_AG_EG_BLin}
We first consider the particular case of bilinear games, where we show that Algorithm \ref{algo_AG_EG}, with proper averaging and scheduled restarting, achieves an optimal statistical rate up to a constant prefactor and with a bias term that matches the lower bound of \citet[Theorem 3]{ibrahim2020linear} for bilinear games.
\emph{We assume in this section that $n=m$ where $\mathbf{B}$ is a nonsingular square matrix}, $\nabla f(\mathbf{x}; \xi) = \mathbf{0}$ and $\nabla g(\mathbf{y}; \xi) = \mathbf{0}$ a.s., so \eqref{problemopt_stochastic} reduces to
\beq\label{BilinearF_stochastic}
\min_{\mathbf{x}}\max_{\mathbf{y}}~
\mathscr{F}(\mathbf{x},\mathbf{y})
=
\Exs_\zeta\left[h(\mathbf{x}, \mathbf{y}; \zeta)\right]
=
H(\mathbf{x},\mathbf{y})
=
\mathbf{x}^\top \mathbf{B}\mathbf{y}
- 
\mathbf{x}^\top \mathbf{u}_{\mathbf{x}} 
+ 
\mathbf{u}_{\mathbf{y}}^\top \mathbf{y}
,
\eeq
and Algorithm \ref{algo_AG_EG} reduces to the independent-sample extragradient descent-ascent algorithm for \eqref{BilinearF_stochastic}.
The saddle point $[\oholder_{\mathbf{x}}^\star; \oholder_{\mathbf{y}}^\star]$ in this case is the unique solution to the linear equation
$$\begin{bmatrix}
\mathbf{0}			&	\mathbf{B}
\\
-\mathbf{B}^\top	&	\mathbf{0}
\end{bmatrix}
\begin{bmatrix}
\oholder_{\mathbf{x}}^\star	\\	\oholder_{\mathbf{y}}^\star
\end{bmatrix}
=
\begin{bmatrix}
\mathbf{u}_{\mathbf{x}}	\\	\mathbf{u}_{\mathbf{y}}
\end{bmatrix}
,
\quad\text{which reduces to $
\begin{bmatrix}
-(\mathbf{B}^\top)^{-1}\mathbf{u}_{\mathbf{y}}
\\
\mathbf{B}^{-1}\mathbf{u}_{\mathbf{x}}
\end{bmatrix}
$
}.
$$
In earlier work, \citet[Proposition 7]{azizian2020accelerating} achieve an upper bound that matches the lower bound of \citet{ibrahim2020linear}.
Our algorithm is in the independent-sample setting with bounded noise variance, that is, it consumes two independent samples, one for each in the extrapolation and update steps.
This is different from the version of \citet{li2021convergence} that consumes a shared sample in both steps.
We allow $[\mathbf{x}_0; \mathbf{y}_0]$ to be randomly initialized, which reduces to a point mass in the case of nonrandom initialization.
Due to the special stepsize selection in the averaging, our analysis of stochastic bilinear game yields the following:

\begin{theorem}[Convergence of stochastic AG-EG, bilinear case]\label{theo_S_AG_EG_Bilinear}
Setting parameters as in \eqref{params} with $\ratio$ being arbitrary, $
\smoothBLin	= \sqrt{\lambda_{\max}(\mathbf{B}^\top \mathbf{B})\cdot \tfrac{1}{\ratio}}
$, $
\smoothFG	= \strcvx = 0
$, and also choosing the stepsizes $\alpha_t = \frac{2}{t+1}$ and $\stepsize_t
\equiv
\tfrac{1}{\smoothBLin}
=
\sqrt{\tfrac{\ratio}{\lambda_{\max}(\mathbf{B}^\top \mathbf{B})}}
$, we have
\begin{equation}\label{AG_EG_Bilinear}\begin{aligned}
\lefteqn{
\Exs\left[
\|\mathbf{x}\avg{\Tholder-\frac12} - \oholder_{\mathbf{x}}^\star\|^2 
+ 
\ratio\|\mathbf{y}\avg{\Tholder-\frac12} - \oholder_{\mathbf{y}}^\star\|^2
\right]
}
\\&\le
\tfrac{\ratio}{\lambda_{\min}(\mathbf{B}\mathbf{B}^\top)}\left(
\tfrac{4\sqrt{\lambda_{\max}(\mathbf{B}^\top \mathbf{B})\cdot\frac{1}{\ratio}}}{\Tholder}
\sqrt{\Exs\left[
\|\mathbf{x}_0 - \oholder_{\mathbf{x}}^\star\|^2 
+ 
\ratio\|\mathbf{y}_0 - \oholder_{\mathbf{y}}^\star\|^2
\right]}
+
\tfrac{7\stdBLin}{\sqrt{\Tholder}}
\right)^2
.
\end{aligned}\end{equation}
\end{theorem}

The proof of Theorem \ref{theo_S_AG_EG_Bilinear} is provided in \S\ref{sec_proof,theo_S_AG_EG_Bilinear}.
For regularity purposes we set the ratio of strong-convexity parameters (both being zero) as $\ratio = 1$ in the rest of this section.%
\footnote{Since it is of an $\frac{0}{0}$-indefinite form, the result also holds for an arbitrary choice of $\ratio\in (0,\infty)$, providing flexibility on the parameter choices.}
Note that our choice of the stepsize is maximally feasible and independent of the noise.
Let us now consider a scheduled restarting version of the algorithm, with a constant epoch length $
\asymp 
\sqrt{
\tfrac{\lambda_{\max}(\mathbf{B}^\top \mathbf{B})}{\lambda_{\min}(\mathbf{B}\mathbf{B}^\top)}
}
$ steps using with the same constant stepsize, until the iteration reaches the stationary noise level in the sense that $
\Exs\left[
\|\mathbf{x}_0 - \oholder_{\mathbf{x}}^\star\|^2 
+ 
\ratio\|\mathbf{y}_0 - \oholder_{\mathbf{y}}^\star\|^2
\right]
\asymp
\tfrac{\stdBLin^2}{\sqrt{\lambda_{\min}(\mathbf{B}\mathbf{B}^\top)\lambda_{\max}(\mathbf{B}^\top \mathbf{B})}}
$.
The convergence rate for this restarting variant is linear in bias term plus an optimal statistical error term, as follows:

\begin{corollary}[Convergence of stochastic AG-EG with scheduled restarting, bilinear case]\label{theo_S_AG_EG_Bilinear_restart}
Equipped with scheduled restarting, the iteration complexity is bounded by
$$
O \left(
\sqrt{
\tfrac{\lambda_{\max}(\mathbf{B}^\top \mathbf{B})}{\lambda_{\min}(\mathbf{B}\mathbf{B}^\top)}
} \log\left(
\tfrac{
\sqrt[4]{\lambda_{\min}(\mathbf{B}\mathbf{B}^\top)\lambda_{\max}(\mathbf{B}^\top \mathbf{B})}
}{\stdBLin}
\right)
+
\tfrac{\stdBLin^2}{\lambda_{\min}(\mathbf{B}\mathbf{B}^\top)\varepsilon^2}
\right)
.
$$
\end{corollary}

In the setting where there is no stochasticity, setting $
\stdBLin	\asymp	\varepsilon\sqrt[4]{\lambda_{\min}(\mathbf{B}\mathbf{B}^\top)\lambda_{\max}(\mathbf{B}^\top \mathbf{B})}
$ the complexity bound in Corollary \ref{theo_S_AG_EG_Bilinear_restart} reduces to $
O \left(
\sqrt{
\tfrac{\lambda_{\max}(\mathbf{B}^\top \mathbf{B})}{\lambda_{\min}(\mathbf{B}\mathbf{B}^\top)}
} \log\left(\tfrac{1}{\varepsilon}\right)
\right)
$ and hence matches the lower bound of \citet{ibrahim2020linear}.
The $
\frac{\stdBLin^2}{\lambda_{\min}(\mathbf{B}\mathbf{B}^\top)\varepsilon^2}
$ term corresponds to the optimal statistical rate for the current problem.

\pb\section{Optimality for strongly-convex-strongly-concave objectives}\label{sec_AG_EG}
In this section, we proceed to solve \eqref{problemopt_stochastic} using Algorithm~\ref{algo_AG_EG} in the general strongly-convex-strongly-concave setting.
Unless otherwise specified, we assume throughout the rest of this paper that $H(\mathbf{x},\mathbf{y})$ is of bilinear form $
\mathbf{x}^\top \mathbf{B}\mathbf{y}
-
\mathbf{x}^\top \mathbf{u}_{\mathbf{x}} 
+
\mathbf{u}_{\mathbf{y}}^\top \mathbf{y}$, without assuming $n=m$ or the nonsingularity of $\mathbf{B}$.
Recall that Algorithm~\ref{algo_AG_EG} conducts acceleration on the strongly parts $F(\mathbf{x})$ and $G(\mathbf{y})$ and extrapolates on bilinear part $H(\mathbf{x},\mathbf{y})$.
We continue to allow $[\mathbf{x}_0; \mathbf{y}_0]$ be randomly initialized and denote
\begin{align}\label{eq_stepsize}
\bar{\stepsize}_t(\tilde{\sigma};\Tholder,\Cholder,\quantile,\beta)
\equiv
\frac{t}{\tfrac{2}{\quantile}\smoothFG\lor \frac{\tilde{\sigma} [\Tholder(\Tholder+1)^2]^{1/2}}{\Cholder\sqrt{\Exs\left[
\|\mathbf{x}_0 - \oholder_{\mathbf{x}}^\star\|^2
+
\ratio\|\mathbf{y}_0 - \oholder_{\mathbf{y}}^\star\|^2
\right]}} + \sqrt{\tfrac{1+\beta}{\quantile}}\smoothBLin t}
,
\end{align}
where $\Cholder\in (0,\infty)$ is an input parameter that allows flexibility in our stepsize selection.
We state our general result as follows:

\begin{theorem}[Convergence of stochastic AG-EG]\label{theo_S_AG_EG}
Let the epoch length $\Tholder\ge 1$ be known in advance, fix  $\quantile\in (0,1)$ and $\beta\in (0,\infty)$ arbitrarily, set the rescaling parameters $\smoothFG$, $\smoothBLin$, $\strcvx$, $\ratio$ as in \eqref{params}, set $
\sigma	\equiv
\frac{1}{\sqrt{3}}\sqrt{
\tfrac{1}{1-\quantile}\stdFG^2
+
(2+\tfrac{1}{\beta})\stdBLin^2
}
$
and choose the stepsizes $\alpha_t = \frac{2}{t+1}$ and $\stepsize_t
=
\bar{\stepsize}_t(\sigma;\Tholder,\Cholder,\quantile,\beta)
$ to be defined as in \eqref{eq_stepsize} with $\Cholder\in (0,\infty)$ being an input parameter.
We have that the output of single-epoch ($\Sholder=1$) Algorithm~\ref{algo_AG_EG} $
[\mathbf{x}\avg{\Tholder-\frac12}; \mathbf{y}\avg{\Tholder-\frac12}]
$ satisfies
\beq\label{S_AG_EG}\begin{aligned}
\lefteqn{
\Exs\left[
\|\mathbf{x}\avg{\Tholder-\frac12} - \oholder_{\mathbf{x}}^\star\|^2
+
\ratio\|\mathbf{y}\avg{\Tholder-\frac12} - \oholder_{\mathbf{y}}^\star\|^2
\right]
}
\\&\le
\tfrac{2}{\strcvx(\Tholder+1)}\left(
\frac{\tfrac{2}{\quantile}\smoothFG}{\Tholder} + \mathcal{A}(\sigma;\Tholder,\Cholder,\quantile,\beta) \sqrt{\tfrac{1+\beta}{\quantile}} \smoothBLin
\right)\Exs\left[
\|\mathbf{x}_0 - \oholder_{\mathbf{x}}^\star\|^2
+
\ratio\|\mathbf{y}_0 - \oholder_{\mathbf{y}}^\star\|^2
\right]
\\&\quad\,
+
\tfrac{2(\frac{1}{\Cholder}+\Cholder)\sigma}{\strcvx\Tholder^{1/2}} \sqrt{\Exs\left[
\|\mathbf{x}_0 - \oholder_{\mathbf{x}}^\star\|^2
+
\ratio\|\mathbf{y}_0 - \oholder_{\mathbf{y}}^\star\|^2
\right]
}
,
\end{aligned}\eeq
where the prefactor
\begin{align}\label{prefactorA}
\mathcal{A}(\tilde{\sigma};\Tholder,\Cholder,\quantile,\beta)	\equiv
1 + \frac{\Cholder\tilde{\sigma}[\Tholder(\Tholder+1)^2]^{1/2}}{\frac{1}{
\stepsize_1(\tilde{\sigma};\Tholder,\Cholder,\quantile,\beta)
}\sqrt{\Exs\left[
\|\mathbf{x}_0 - \oholder_{\mathbf{x}}^\star\|^2
+
\ratio\|\mathbf{y}_0 - \oholder_{\mathbf{y}}^\star\|^2
\right]
}}
,
\end{align}
 lies in $[1,1+\Cholder^2]$ and reduces to 1 when $\tilde{\sigma} = 0$.
\end{theorem}

The proof of Theorem \ref{theo_S_AG_EG} is provided in \S\ref{sec_proof,theo_S_AG_EG}.
In the case that there is no stochasticity, by taking $\quantile\to 1^-$, $\beta\to 0^+$ in our analysis we obtain the following result:

\begin{theorem}[Convergence of AG-EG]\label{theo_AG_EG}
Setting the rescaling parameters $\smoothFG$, $\smoothBLin$, $\strcvx$, $\ratio$ as in \eqref{params}, we have that by choosing $\stepsize_t = \tfrac{t}{2\smoothFG + \smoothBLin t}$ the output of Algorithm~\ref{algo_AG_EG} with $\Sholder=1$ satisfies
\begin{equation}\label{AG_EG}
\|\mathbf{x}\avg{\Tholder-\frac12} - \oholder_{\mathbf{x}}^\star\|^2
+
\ratio\|\mathbf{y}\avg{\Tholder-\frac12} - \oholder_{\mathbf{y}}^\star\|^2
\le
\tfrac{2}{\strcvx (\Tholder+1)}\Big(\tfrac{2\smoothFG}{\Tholder} + \smoothBLin\Big)\left[ 
\|\mathbf{x}_0 - \oholder_{\mathbf{x}}^\star\|^2
+ 
\ratio\|\mathbf{y}_0 - \oholder_{\mathbf{y}}^\star\|^2 
\right]
,
\end{equation}
where $\mathbf{x}\avg{\Tholder-\frac12}$ and $\mathbf{y}\avg{\Tholder-\frac12}$ are defined in~Algorithm~\ref{algo_AG_EG}.
\end{theorem}

We make a few remarks on Theorems \ref{theo_S_AG_EG} and \ref{theo_AG_EG} as follows:

\begin{enumerate}[label=(\roman*),leftmargin=5mm]
\item
When $\smoothBLin$ is set as zero the problem is decoupled, and our algorithm for a single variate reduces to the standard three-line formulation of stochastic Nesterov's accelerated gradient descent, where the choice of $\alpha_t = \frac{2}{t+1}$ is essential to achieve desirable convergence behavior \citep{nesterov1983method}.
The step-size choice $\stepsize_t
=
\bar{\stepsize}_t(\sigma;\Tholder,\Cholder,\quantile,\beta)
$ as in \eqref{eq_stepsize} is directly generalized from the optimal choice in stochastic Nesterov’s method by incorporating the bilinear coupling term in its dominator;
we refer interested readers to \cite[Chap.~4]{LAN[First]} for a careful treatment.
Our hyperparameter dependency is in a fine-grained fashion; often, the convergence rate coefficients are not a concern, and the coarse choices of $\quantile = \frac12$ and $\beta=1$ should suffice.
In words, how $\quantile$ deviating from 1 and $\beta$ deviating from 0 should be a trade-off between the noise variance and the convergence rate coefficients.

\item
Compared with Theorem \ref{theo_S_AG_EG_Bilinear}, the nonasymptotic convergence rate in Theorem \ref{theo_S_AG_EG} is slowed down from $O(\frac{1}{\mathscr{T}})$ to $O(\frac{1}{\sqrt{\mathscr{T}}})$ in squared metric due to the nonlinear nature of our system.
As we will see immediately afterward, with the use of scheduled restarting the dependency on initialization will be exponential. 
Also we note that although in different settings, for the nonrandom Theorem \ref{theo_AG_EG} the stepsize choice is consistent with the choice in the bilinear game Theorem \ref{theo_S_AG_EG_Bilinear}.

\item
The choice of $\Cholder$ reflects the trade-offs between terms in our convergence rate bounds.
In the nonrandom setting the algorithm does not require any knowledge or estimate of the initial distance to a saddle to achieve the desirable rate.

\item
Suppose that randomness exists, the leading-order term on the right hand of \eqref{S_AG_EG} admits a hyperbolic dependency on $\Cholder$.
In the case that we are given the full knowledge of the initial distance to a saddle we can optimally choose $\Cholder = 1$ on the right-hand side of \eqref{S_AG_EG}.
In the alternative case where only an upper estimate $\Gamma_0$ of $\sqrt{\Exs\left[
\|\mathbf{x}_0 - \oholder_{\mathbf{x}}^\star\|^2
+
\ratio\|\mathbf{y}_0 - \oholder_{\mathbf{y}}^\star\|^2
\right]}$ is provided, we simply set $
\Cholder	=	\frac{\Gamma_0}{ 
\sqrt{\Exs\left[
\|\mathbf{x}_0 - \oholder_{\mathbf{x}}^\star\|^2
+
\ratio\|\mathbf{y}_0 - \oholder_{\mathbf{y}}^\star\|^2
\right]
}}\ge 1$ which yields the following bound to \eqref{S_AG_EG}
\beq\label{recrecursion}
\begin{aligned}
\lefteqn{
\Exs\left[
\|\mathbf{x}\avg{\Tholder-\frac12} - \oholder_{\mathbf{x}}^\star\|^2
+
\ratio\|\mathbf{y}\avg{\Tholder-\frac12} - \oholder_{\mathbf{y}}^\star\|^2
\right]
}
\\&\le
\tfrac{2}{\strcvx(\Tholder+1)}\left(
\frac{\tfrac{2}{\quantile}\smoothFG}{\Tholder}\Exs\left[
\|\mathbf{x}_0 - \oholder_{\mathbf{x}}^\star\|^2
+
\ratio\|\mathbf{y}_0 - \oholder_{\mathbf{y}}^\star\|^2
\right]
+
2\sqrt{\tfrac{1+\beta}{\quantile}} \smoothBLin\Gamma_0^2
\right)
+
\tfrac{4\sigma}{\strcvx\Tholder^{1/2}} \Gamma_0
\\&\le
\tfrac{2}{\strcvx(\Tholder+1)}\left(
\frac{\tfrac{2}{\quantile}\smoothFG}{\Tholder} + 2\sqrt{\tfrac{1+\beta}{\quantile}} \smoothBLin
\right)\Gamma_0^2
+
\tfrac{4\sigma}{\strcvx\Tholder^{1/2}} \Gamma_0
.
\end{aligned}\eeq
Note we used $\mathcal{A}(\tilde{\sigma};\Tholder,\Cholder,\quantile,\beta) \le 2\Cholder^2$.
Our upcoming scheduled-restarting analysis is heavily based on this bound.
\end{enumerate}

To prepare for our multi-epoch result with the help of scheduled restarting, analogous to Corollary \ref{theo_S_AG_EG_Bilinear_restart} we perform an induction based on \eqref{recrecursion}: suppose $
\Exs\left[
\|\mathbf{x}_0^{[\epoch-1]} - \oholder_{\mathbf{x}}^\star\|^2
+
\ratio\|\mathbf{y}_0^{[\epoch-1]} - \oholder_{\mathbf{y}}^\star\|^2
\right]
\le
\Gamma_0^2 e^{1-\epoch}
$ hold, and we obtain (by taking $\quantile = \frac12$ and $\beta = 1$ for simplicity)
$$\begin{aligned}
\Exs\left[
\|\mathbf{x}_0^{[\epoch]} - \oholder_{\mathbf{x}}^\star\|^2
+
\ratio\|\mathbf{y}_0^{[\epoch]} - \oholder_{\mathbf{y}}^\star\|^2
\right]
\lesssim
\tfrac{\smoothFG}{\strcvx \Tholder_{\epoch}^2} \Gamma_0^2 e^{1-\epoch}
+
\tfrac{\smoothBLin}{\strcvx \Tholder_{\epoch}} \Gamma_0^2 e^{1-\epoch}
+
\tfrac{\sigma}{\strcvx \Tholder_{\epoch}^{1/2}} \Gamma_0 e^{\frac{1-\epoch}{2}}
.
\end{aligned}$$
Setting the above display as $\le
\Gamma_0^2 e^{-\epoch}
$, and setting the length of epoch $s$ as $
\Tholder_{\epoch}
\asymp
\sqrt{\tfrac{\smoothFG}{\strcvx}}+ \tfrac{\smoothBLin}{\strcvx}
+ 
\tfrac{\sigma^2}{\strcvx^2 \Gamma_0^2 e^{1-\epoch}}
$ we arrive at a total complexity of
$$
\lesssim
\sum_{s=1}^{LOG}\left[
\sqrt{\tfrac{\smoothFG}{\strcvx}}+ \tfrac{\smoothBLin}{\strcvx}
+
\tfrac{\sigma^2}{\strcvx^2 \Gamma_0^2 e^{1-\epoch}}
\right]
=
\left(\sqrt{\tfrac{\smoothFG}{\strcvx}}+ \tfrac{\smoothBLin}{\strcvx}\right) LOG
+
\tfrac{\sigma^2}{\strcvx^2 \Gamma_0^2}\cdot \tfrac{e^{LOG}-1}{e-1}
,
$$
where $
LOG \equiv \left\lceil\log \tfrac{\Gamma_0^2}{\varepsilon^2}	\right\rceil
$, so it is bounded by a constant multiple of
$$
\left(\sqrt{\tfrac{\smoothFG}{\strcvx}}+ \tfrac{\smoothBLin}{\strcvx}\right) \left\lceil\log \tfrac{\Gamma_0^2}{\varepsilon^2}	\right\rceil
+
\tfrac{\sigma^2}{\strcvx^2 \Gamma_0^2} e^{\left\lceil\log \tfrac{\Gamma_0^2}{\varepsilon^2}\right\rceil}
\asymp
\left(\sqrt{\tfrac{\smoothFG}{\strcvx}}+ \tfrac{\smoothBLin}{\strcvx}\right) \left\lceil\log \tfrac{\Gamma_0^2}{\varepsilon^2}	\right\rceil
+
\tfrac{\sigma^2}{\strcvx^2\varepsilon^2}
.
$$
This yields the following multi-epoch iteration complexity bound result:

\begin{corollary}[Convergence of stochastic AG-EG with scheduled restarting]\label{theo_S_AG_EG_restart}
When a scheduled restarting argument is employed on top of Algorithm~\ref{algo_AG_EG}, with an epoch length $
\Tholder_s	\asymp	\sqrt{\tfrac{\smoothFG}{\strcvx}} + \tfrac{\smoothBLin}{\strcvx} + 
\tfrac{\sigma^2}{\strcvx^2 \Gamma_0^2 e^{1-\epoch}}
$ we obtain the iteration complexity of  
$$\begin{aligned}
O\left(
\left(\sqrt{\tfrac{\smoothFG}{\strcvx}} + \tfrac{\smoothBLin}{\strcvx}\right)\log\left(\tfrac{1}{\varepsilon}\right) 
+
\tfrac{\sigma^2}{\strcvx^2\varepsilon^2}
\right)
&=
O\left(
\left(
\sqrt{\tfrac{L_F}{\mu_F} \lor \tfrac{L_G}{\mu_G}} 
+ 
\sqrt{\tfrac{\lambda_{\max}(\mathbf{B}^\top \mathbf{B})}{\mu_F\mu_G}}
\right) \log\left(\tfrac{1}{\varepsilon}\right)
+
\tfrac{\sigma^2}{\mu_F^2\varepsilon^2}
\right)
.
\end{aligned}$$
\end{corollary}
In the nonrandom setting, the iteration complexity upper bound in Theorem \ref{theo_S_AG_EG_restart} matches the lower bound of \citet{zhang2021lower} $
\Omega\left(\left(
\sqrt{\frac{L_F}{\mu_F} \lor \frac{L_G}{\mu_G}} 
+ 
\sqrt{\frac{\lambda_{\max}(\mathbf{B}^\top \mathbf{B})}{\mu_F\mu_G}}
\right) \log\left(\frac{1}{\varepsilon}\right)\right)
$, and we achieve the optimal statistical rate $
\frac{\sigma^2}{\mu_F^2\varepsilon^2}
$ up to a constant prefactor.
Note that the hard instance constructed by \citet{zhang2021lower} has the form of a quadratic minimax game, and hence it is in a special case of a bilinearly-coupled saddle-point problem, and the same lower bound holds for problem \eqref{problemopt_stochastic} in this  case.

\pb\section{A direct approach for strongly-convex-strongly-concave objectives}\label{sec_AG_EG_Str}
For solving \eqref{problemopt_stochastic} we turn to our (AMP-inspired) stochastic AG-EG algorithm that targets strongly-convex-strongly-concave problems.
For $F(\mathbf{x})$ being $\mu_F$-strongly-convex and $G(\mathbf{y})$ being $\mu_G$-strongly-convex, by letting the algorithm be initialized at a fixed $[\mathbf{x}_0; \mathbf{y}_0]$ we group the objective in \eqref{problemopt_stochastic} as
\beq\label{problemopt_group}
\begin{aligned}
\mathscr{F}(\mathbf{x},\mathbf{y})
&=
\left(
F(\mathbf{x}) - \tfrac{\mu_F}{2}\|\mathbf{x}-\mathbf{x}_0\|^2 
\right)
+ 
\left(
\tfrac{\mu_F}{2}\|\mathbf{x}-\mathbf{x}_0\|^2 + H(\mathbf{x},\mathbf{y}) - \tfrac{\mu_G}{2}\|\mathbf{y}-\mathbf{y}_0\|^2
\right)
\\&\hspace{3in}
- 
\left(
G(\mathbf{y}) - \tfrac{\mu_G}{2}\|\mathbf{y}-\mathbf{y}_0\|^2 
\right)
,
\end{aligned}\eeq
where $
\frac{\mu_F}{2}\|\mathbf{x}-\mathbf{x}_0\|^2 + H(\mathbf{x},\mathbf{y}) - \frac{\mu_G}{2}\|\mathbf{y}-\mathbf{y}_0\|^2
$ is a $\mu_F$-strongly-convex-$\mu_G$-strongly-concave isotropic quadratic function.
Applying the updates in Lines~\ref{line_a} to \ref{line_h} in Algorithm \ref{algo_AG_EG} to the new grouping yields Algorithm \ref{algo_AG_EG_Str}, which resembles the algorithmic design of \citet{thekumparampil2022lifted}, except we are employing an extragradient method instead of the Chambolle-Pock-style primal-dual method as an approximation of proximal point methods.
We also redefine in this section our rescaling parameters as
\beq\label{params_group}
\smoothFG = L_F \lor \left( \tfrac{\mu_F}{\mu_G} L_G \right) - \mu_F
,\quad
\smoothBLin = \sqrt{\lambda_{\max}(\mathbf{B}^\top \mathbf{B})\cdot \tfrac{\mu_F}{\mu_G} + \mu_F^2}
,\quad
\mu_\star = \mu_F
,\quad
\ratio = \tfrac{\mu_G}{\mu_F}
.
\eeq
Our new result is as follows:

\begin{algorithm}[!tb]
\caption{Stochastic AcceleratedGradient-ExtraGradient (AG-EG) Descent-Ascent Algorithm, Direct Approach}
\begin{algorithmic}[1]
\REQUIRE 
Initialization $\mathbf{x}_0, \mathbf{y}_0$, total number of iterates $\Tholder$, step sizes $(\alpha_t, \stepsize_t: t=1,2,\dots)$, ratio of strong-convexity parameters $\ratio = \frac{\mu_G}{\mu_F}$
\STATE
Set $
\mathbf{x}\avg{-\frac12}\leftarrow \mathbf{x}_0
$, $
\mathbf{y}\avg{-\frac12}\leftarrow \mathbf{y}_0
$, $
\mathbf{x}\Xtrap{0}\leftarrow \mathbf{x}_0
$, $
\mathbf{y}\Xtrap{0}\leftarrow \mathbf{y}_0
$
\FOR{$t=1, 2,\dots,\Tholder$}
\STATE
Draw samples $\xi_{t-\frac12}\sim \mathcal{D}_\xi$ from oracle, and also $\zeta_{t-\frac12}, \zeta_t\sim \mathcal{D}_\zeta$ independently from oracle
\STATE\label{line_a_pri}
$
\mathbf{x}_{t-\frac12}	\leftarrow	\mathbf{x}_{t-1} - \stepsize_t\left(
\nabla f(\mathbf{x}\Xtrap{t-1}; \xi_{t-\frac12})
+
\nabla_{\mathbf{x}} h(\mathbf{x}_{t-1}, \mathbf{y}_{t-1}; \zeta_{t-\frac12}) - \mu_F(\mathbf{x}\Xtrap{t-1}-\mathbf{x}_{t-1})
\right)
$
\STATE\label{line_b_pri}
$
\mathbf{y}_{t-\frac12}	\leftarrow	\mathbf{y}_{t-1} - \tfrac{\stepsize_t}{\ratio}\left(
-\nabla_{\mathbf{y}} h(\mathbf{x}_{t-1}, \mathbf{y}_{t-1}; \zeta_{t-\frac12})
+
\nabla g(\mathbf{y}\Xtrap{t-1}; \xi_{t-\frac12}) - \mu_G(\mathbf{y}\Xtrap{t-1}-\mathbf{y}_{t-1})
\right)
$
\STATE\label{line_c_pri}
$
\mathbf{x}\avg{t-\frac12}	\leftarrow	(1-\alpha_t)\mathbf{x}\avg{t-\frac32} + \alpha_t\mathbf{x}_{t-\frac12}
$
\STATE\label{line_d_pri}
$
\mathbf{y}\avg{t-\frac12}	\leftarrow	(1-\alpha_t)\mathbf{y}\avg{t-\frac32} + \alpha_t\mathbf{y}_{t-\frac12}
$
\STATE\label{line_e_pri}
$
\mathbf{x}_t			\leftarrow	\mathbf{x}_{t-1} - \stepsize_t\left(
\nabla f(\mathbf{x}\Xtrap{t-1}; \xi_{t-\frac12})
+
\nabla_{\mathbf{x}} h(\mathbf{x}_{t-\frac12}, \mathbf{y}_{t-\frac12}; \zeta_t) - \mu_F(\mathbf{x}\Xtrap{t-1}-\mathbf{x}_{t-\frac12})
\right)
$
\STATE\label{line_f_pri}
$
\mathbf{y}_t			\leftarrow	\mathbf{y}_{t-1} - \tfrac{\stepsize_t}{\ratio}\left(
-\nabla_{\mathbf{y}} h(\mathbf{x}_{t-\frac12}, \mathbf{y}_{t-\frac12}; \zeta_t)
+
\nabla g(\mathbf{y}\Xtrap{t-1}; \xi_{t-\frac12}) - \mu_G(\mathbf{y}\Xtrap{t-1}-\mathbf{y}_{t-\frac12})
\right)
$
\STATE\label{line_g_pri}
$
\mathbf{x}\Xtrap{t}		\leftarrow	(1-\alpha_{t+1})\mathbf{x}\avg{t-\frac12} + \alpha_{t+1}\mathbf{x}_t
$
\STATE\label{line_h_pri}
$
\mathbf{y}\Xtrap{t}		\leftarrow	(1-\alpha_{t+1})\mathbf{y}\avg{t-\frac12} + \alpha_{t+1}\mathbf{y}_t
$
\ENDFOR
\STATE\label{lineoutput_pri}
{\bfseries Output:}
$[\mathbf{x}_{\Tholder}; \mathbf{y}_{\Tholder}]$
\end{algorithmic}
\label{algo_AG_EG_Str}
\end{algorithm}

\begin{theorem}[Convergence of stochastic AG-EG, direct approach]\label{theo_S_AG_EG_Str}
For solving problem \eqref{problemopt_group}, assume for each $\mathbf{x}\in \reals^n, \mathbf{y}\in \reals^m$ and $\xi\sim \mathcal{D}_\xi, \zeta\sim \mathcal{D}_\zeta$ that \eqref{noisevarone} and \eqref{noisevartwo} are satisfied.
Fix arbitrarily $\quantile\in (0,1)$, $\beta\in (0,\infty)$, set the rescaling parameters $\smoothFG$, $\smoothBLin$, $\strcvx$, $\ratio$ as in \eqref{params_group}, choose the stepsizes $\alpha_t\in \left(0, \bar{\alpha}(\quantile,\beta)\right]$ with
\beq\label{baralpha}
\bar{\alpha}(\quantile,\beta) 
\equiv 
\frac{\quantile}{1+\sqrt{1 + \quantile\left(\tfrac{\smoothFG}{\mu_\star} + \tfrac{(1+\beta)\smoothBLin^2}{\mu_\star^2}\right)}}
,
\eeq
as well as $\stepsize_t = \frac{\alpha_t}{\mu_\star}$.
Then the iterates of Algorithm \ref{algo_AG_EG_Str} satisfies for all $t=1,\dots,\Tholder$
\begin{equation}\label{S_AG_EG_Str}
\begin{aligned}
\lefteqn{
\Exs\left[
\|\mathbf{x}_t - \oholder_{\mathbf{x}}^\star\|^2
+
\ratio\|\mathbf{y}_t - \oholder_{\mathbf{y}}^\star\|^2
\right]
}
\\&\le
\left[
\|\mathbf{x}_0 - \oholder_{\mathbf{x}}^\star\|^2
+
\ratio\|\mathbf{y}_0 - \oholder_{\mathbf{y}}^\star\|^2
\right]
\left(\tfrac{\smoothFG}{\mu_\star}+1\right)\prod_{\tau=1}^t (1-\alpha_\tau)
+
\tfrac{3\sigma^2}{\mu_\star^2}\sum_{\tau=1}^t \alpha_\tau^2\prod_{\tau’=\tau+1}^t (1-\alpha_{\tau’})
,
\end{aligned}\end{equation}
where we inherit the notation $\sigma	=
\frac{1}{\sqrt{3}}\sqrt{
\tfrac{1}{1-\quantile}\stdFG^2	+	(2+\tfrac{1}{\beta})\stdBLin^2
}$ from Theorem \ref{theo_S_AG_EG}.
\end{theorem}

The proof of Theorem \ref{theo_S_AG_EG_Str} is provided in \S\ref{sec_proof,theo_S_AG_EG_Str}.
We highlight that our result applies to the output in Line \ref{lineoutput_pri} as $\mathbf{x}_{\Tholder}$, $\mathbf{y}_{\Tholder}$ instead of $\mathbf{x}\avg{\Tholder-\frac12}$, $\mathbf{y}\avg{\Tholder-\frac12}$ as in Line \ref{lineoutput} of Algorithm \ref{algo_AG_EG}.
Additionally, let the total number of iterates $\Tholder\ge 1$ be known in advance, and consider a constant stepsize $\alpha_t\equiv \alpha$. Optimizing the error bound over $\alpha$ gives
$$
\alpha
=
\tfrac{1}{\Tholder}\left(1 + \log\left(
\left[
\|\mathbf{x}_0 - \oholder_{\mathbf{x}}^\star\|^2
+
\ratio\|\mathbf{y}_0 - \oholder_{\mathbf{y}}^\star\|^2
\right]
\left(\tfrac{\smoothFG}{\mu_\star}+1\right)
\cdot
\tfrac{\mu_\star^2\Tholder}{3\sigma^2}
\right)\right)
\,\land
\bar{\alpha}(\quantile,\beta)
,
$$
and hence \eqref{S_AG_EG_Str} leads to the following:
$$\begin{aligned}
&
\Exs\left[
\|\mathbf{x}_\Tholder - \oholder_{\mathbf{x}}^\star\|^2
+
\ratio\|\mathbf{y}_\Tholder - \oholder_{\mathbf{y}}^\star\|^2
\right]
\le
\left[
\|\mathbf{x}_0 - \oholder_{\mathbf{x}}^\star\|^2
+
\ratio\|\mathbf{y}_0 - \oholder_{\mathbf{y}}^\star\|^2
\right]
\left(\tfrac{\smoothFG}{\mu_\star}+1\right) e^{-\bar{\alpha}(\quantile,\beta)\Tholder}
\\&\hspace{1.5in}\,
+
\tfrac{3\sigma^2}{\mu_\star^2\Tholder}\left(1 + \log\left(
\left[
\|\mathbf{x}_0 - \oholder_{\mathbf{x}}^\star\|^2
+
\ratio\|\mathbf{y}_0 - \oholder_{\mathbf{y}}^\star\|^2
\right]
\left(\tfrac{\smoothFG}{\mu_\star}+1\right)
\cdot
\tfrac{\mu_\star^2\Tholder}{3\sigma^2}
\right)\right)
.
\end{aligned}$$
Prescribing the desired accuracy $\varepsilon > 0$, Theorem \ref{theo_S_AG_EG_Str} shows that the iteration complexity to output an iterate $\mathbf{x}_\Tholder \in \reals^n$, $\mathbf{y}_\Tholder \in \reals^m$ that satisfies $\Exs[
\|\mathbf{x}_\Tholder - \oholder_{\mathbf{x}}^\star\|^2
+
\ratio\|\mathbf{y}_\Tholder - \oholder_{\mathbf{y}}^\star\|^2
]	\le	\varepsilon^2$ is upper bounded by a constant multiple of%
\footnote{Throughout this work, we focus on the iteration complexity whereas the required number of queries to stochastic gradient oracle is three times the iteration complexity (one query to $[\nabla f(\mathbf{x}; \xi); \nabla g(\mathbf{y}; \xi)]$ and two queries to $\nabla h(\mathbf{x}, \mathbf{y}; \zeta)$).}
$$
\left(
\sqrt{\tfrac{\smoothFG}{\mu_\star}} 
+ 
\tfrac{\smoothBLin}{\mu_\star}
+
\tfrac{\sigma^2}{\mu_\star^2\varepsilon^2}
\right) \log\Big((\tfrac{\smoothFG}{\mu_\star}+1)\tfrac{1}{\varepsilon}\Big)
=
\Big(
\sqrt{\tfrac{L_F}{\mu_F} \lor \tfrac{L_G}{\mu_G}} 
+ 
\sqrt{\tfrac{\lambda_{\max}(\mathbf{B}^\top \mathbf{B})}{\mu_F\mu_G}}
+
\tfrac{\sigma^2}{\mu_F^2\varepsilon^2}
\Big) \log\Big((\tfrac{L_F}{\mu_F} \lor \tfrac{L_G}{\mu_G})\tfrac{1}{\varepsilon}\Big)
.
$$
Compared to the stochastic AG-EG with restarting in Theorem \ref{theo_S_AG_EG_restart}, we see that there is a multiplicative $\tfrac{L_F}{\mu_F} \lor \tfrac{L_G}{\mu_G}$ term inside the logarithmic factor.
We believe that the extra logarithmic factor on the optimal statistical rate $\tfrac{\sigma^2}{\mu_\star^2\varepsilon^2}$ is removable using a proper diminishing stepsize strategy, a possibility that we reserve for future study.

Analogous to Theorem \ref{theo_AG_EG} in the case of no stochasticity, setting $\quantile\to 1^-$, $\beta\to 0^+$ gives us the following convergence rate which matches the \citet{zhang2021lower} lower bound:

\begin{theorem}[Convergence of AG-EG, direct approach]\label{theo_AG_EG_Str}
Suppose we are in the setting of Theorem \ref{theo_S_AG_EG_Str} with no stochasticity.
We have by choosing $\alpha_t \equiv \bar{\alpha}(1,0)$ defined as in \eqref{baralpha} as well as $\stepsize_t \equiv \frac{\bar{\alpha}(1,0)}{\mu_\star}$, the output of Algorithm~\ref{algo_AG_EG_Str} satisfies
\begin{equation}\label{AG_EG_Str}
\begin{aligned}
&
\|\mathbf{x}_\Tholder - \oholder_{\mathbf{x}}^\star\|^2
+
\ratio\|\mathbf{y}_\Tholder - \oholder_{\mathbf{y}}^\star\|^2
\\&\le
\left[
\|\mathbf{x}_0 - \oholder_{\mathbf{x}}^\star\|^2
+
\ratio\|\mathbf{y}_0 - \oholder_{\mathbf{y}}^\star\|^2
\right]
\left(\tfrac{\smoothFG}{\mu_\star}+1\right) \exp\left(
- \frac{\Tholder}{1+\sqrt{1 + \tfrac{\smoothFG}{\mu_\star} + \tfrac{\smoothBLin^2}{\mu_\star^2}}}
\right)
.
\end{aligned}\end{equation}
\end{theorem}

We end this section by remarking that in the nonrandom Theorem \ref{theo_AG_EG_Str}, this convergence rate upper bound indicates a near-unity coefficient on its condition-number exponent, yielding an iteration complexity that is asymptotically
$$
\left(
1+\sqrt{
1+\tfrac{\smoothFG}{\mu_\star}
+
\tfrac{\smoothBLin^2}{\mu_\star^2}
}
\right) \log\left((\tfrac{\smoothFG}{\mu_\star}+1)\tfrac{1}{\varepsilon}\right)
\sim
\left(\sqrt{
\tfrac{L_F}{\mu_F} \lor \tfrac{L_G}{\mu_G}
+
\tfrac{\lambda_{\max}(\mathbf{B}^\top \mathbf{B})}{\mu_F\mu_G}
}\right) \log\left((\tfrac{L_F}{\mu_F} \lor \tfrac{L_G}{\mu_G})\tfrac{1}{\varepsilon}\right)
,
$$
which is sharper in its prefactor than the restarting iteration complexity result in Corollary \ref{theo_S_AG_EG_restart} in \S\ref{sec_AG_EG}.
Nevertheless in the bilinear game case without scheduled restarting,%
\footnote{With some effort one can generalize the argument of \citet{chen2017accelerated} to the case of a strongly monotone operator, yet a projection step is still necessary, without which a scheduled restarting argument leads to an extra multiplicative logarithmic factor in condition number in its iteration complexity.}
our direct approach in Algorithm~\ref{algo_AG_EG_Str} reduces to a last-iterate independent-sample stochastic extragradient algorithm whose bias term does \emph{not} match the \citet{ibrahim2020linear} lower bound, and it suffers from non-convergence behavior for the bounded stochastic case \citep{hsieh2020explore}.

\pb\section{Discussion}\label{sec_discuss}
We have presented a stochastic extragradient-based acceleration algorithm, AG-EG, for solving the bilinearly-coupled saddle-point problem \eqref{problemopt_stochastic} that simultaneously matches lower bounds due to \citet{zhang2021lower} and \citet{ibrahim2020linear} for strongly-convex-strongly-concave and bilinear games, respectively.
To the best of our knowledge, this is the first time that both lower bounds have been met by a single algorithm.
There are some remaining issues to be addressed, however, including the case of one-sided non-strong convexity, the setting of unbounded noise variance, and the characterization of the full parameter regime dependency on $\lambda_{\min}(\mathbf{B}\mathbf{B}^\top)$.  These are left as important directions for future research.

\section*{Acknowledgements}
This work is supported in part by NSF Award’s IIS-2110170 and DMS-2134106 to SSD, by Canada CIFAR AI Chair to GG, by the Mathematical Data Science program of the Office of Naval Research under grant number N00014-18-1-2764 and also the Vannevar Bush Faculty Fellowship program under grant number N00014-21-1-2941 and NSF grant IIS-1901252 to MIJ.

\bibliographystyle{plainnat}
\bibliography{SAILreference}

\pb\section{Proofs of main results}\label{sec_proof}
In this section we present the proofs of our main results.
\S\ref{sec_scalingreduction} illustrates the scaling reduction argument.
\S\ref{sec_aux_lemma} provides auxiliary lemmas.
With a slight adjustment of their presentation order \S\ref{sec_proof,theo_S_AG_EG} proves Theorem \ref{theo_S_AG_EG}, \S\ref{sec_proof,theo_S_AG_EG_Str} proves Theorem \ref{theo_S_AG_EG_Str} and finally \S\ref{sec_proof,theo_S_AG_EG_Bilinear} proves Theorem \ref{theo_S_AG_EG_Bilinear}.

\pb\subsection{Scaling reduction argument}\label{sec_scalingreduction}
Here we illustrate the scaling reduction argument that reduces our analysis of our AG-EG Algorithm \ref{algo_AG_EG} to the one with equal strong-convexity parameters of $F$ and $G$ using a reparametrized objective function; the same argument applies to Algorithm \ref{algo_AG_EG_Str} and we omit the details.
The idea is in fact analogous to mirror descent-ascent with respect to a Bregman divergence, and our goal here is to detail this argument for our analysis.

In lieu to \eqref{problemopt_stochastic} we consider
$$
\min_{\hat{\mathbf{x}}} \max_{\hat{\mathbf{y}}}~
\hat{\mathscr{F}}(\hat{\mathbf{x}}, \hat{\mathbf{y}})
=
F(\hat{\mathbf{x}}) + \hat{H}(\hat{\mathbf{x}},\hat{\mathbf{y}}) - \hat{G}(\hat{\mathbf{y}})
,
$$
where we have $\hat{\mathscr{F}}(\hat{\mathbf{x}}, \hat{\mathbf{y}})	=	\mathscr{F}(\mathbf{x}, \mathbf{y})$ with the symbolic reparametrization $
\hat{\mathbf{x}} = \mathbf{x}
$, $
\hat{\mathbf{y}} = \sqrt{\frac{\mu_G}{\mu_F}} \mathbf{y}
$, $
\hat{H}(\hat{\mathbf{x}},\hat{\mathbf{y}}) = H(\mathbf{x},\mathbf{y})
$, $
\hat{h}(\hat{\mathbf{x}},\hat{\mathbf{y}};\zeta) = h(\mathbf{x},\mathbf{y};\zeta)
$, $
\hat{G}(\hat{\mathbf{y}}) = G(\mathbf{y})
$, $
\hat{g}(\hat{\mathbf{y}};\xi) = g(\mathbf{y};\xi)
$ and also their derivatives
$$
\nabla_{\hat{\mathbf{y}}}\hat{H}(\hat{\mathbf{x}},\hat{\mathbf{y}}) = \sqrt{\frac{\mu_F}{\mu_G}} \nabla_{\mathbf{y}}H(\mathbf{x},\mathbf{y})
,\qquad
\nabla_{\hat{\mathbf{y}}}\hat{h}(\hat{\mathbf{x}},\hat{\mathbf{y}};\zeta) = \sqrt{\frac{\mu_F}{\mu_G}} \nabla_{\mathbf{y}}h(\mathbf{x},\mathbf{y};\zeta)
,
$$
and
$$
\nabla\hat{G}(\hat{\mathbf{y}}) = \sqrt{\frac{\mu_F}{\mu_G}} \nabla G(\mathbf{y})
,\qquad
\nabla\hat{g}(\hat{\mathbf{y}};\xi) = \sqrt{\frac{\mu_F}{\mu_G}} \nabla g(\mathbf{y};\xi)
.
$$
It is straightforward to verify $\hat{\mathscr{F}}(\hat{\mathbf{x}}, \hat{\mathbf{y}})$ is arguably $\strcvx$-strongly-convex-$\strcvx$-strongly-concave.
The essence of our update rules is captured by 8 lines corresponding to Lines~\ref{line_a}--\ref{line_h} in Algorithm \ref{algo_AG_EG}, which becomes:
\begin{subequations}
\begin{align}
\hat{\mathbf{x}}_{t-\frac12}	&=	\hat{\mathbf{x}}_{t-1} - \stepsize_t \left( 
\nabla f(\hat{\mathbf{x}}\Xtrap{t-1}; \xi_{t-\frac12}) + \nabla_{\hat{\mathbf{x}}} h(\hat{\mathbf{x}}_{t-1}, \hat{\mathbf{y}}_{t-1}; \zeta_{t-\frac12})
\right) 
,\label{eq_line_a}
\\
\hat{\mathbf{y}}_{t-\frac12}	&=	\hat{\mathbf{y}}_{t-1} - \stepsize_t \left( 
-\nabla_{\hat{\mathbf{y}}} h(\hat{\mathbf{x}}_{t-1}, \hat{\mathbf{y}}_{t-1}; \zeta_{t-\frac12}) + \nabla g(\hat{\mathbf{y}}\Xtrap{t-1}; \xi_{t-\frac12})
\right) 
,\label{eq_line_b}
\\
\hat{\mathbf{x}}\avg{t-\frac12}		&=	(1-\alpha_t)\hat{\mathbf{x}}\avg{t-\frac32} + \alpha_t\hat{\mathbf{x}}_{t-\frac12}
,\label{eq_line_c}
\\
\hat{\mathbf{y}}\avg{t-\frac12}		&=	(1-\alpha_t)\hat{\mathbf{y}}\avg{t-\frac32} + \alpha_t\hat{\mathbf{y}}_{t-\frac12}
,\label{eq_line_d}
\\
\hat{\mathbf{x}}_t			&=	\hat{\mathbf{x}}_{t-1} - \stepsize_t \left( 
\nabla f(\hat{\mathbf{x}}\Xtrap{t-1}; \xi_{t-\frac12}) + \nabla_{\hat{\mathbf{x}}} h(\hat{\mathbf{x}}_{t-\frac12}, \hat{\mathbf{y}}_{t-\frac12}; \zeta_t)
\right) 
,\label{eq_line_e}
\\
\hat{\mathbf{y}}_t			&=	\hat{\mathbf{y}}_{t-1} - \stepsize_t \left( 
-\nabla_{\hat{\mathbf{y}}} h(\hat{\mathbf{x}}_{t-\frac12}, \hat{\mathbf{y}}_{t-\frac12}; \zeta_t) + \nabla g(\hat{\mathbf{y}}\Xtrap{t-1}; \xi_{t-\frac12})
\right) 
,\label{eq_line_f}
\\
\hat{\mathbf{x}}\Xtrap{t}		&=	(1-\alpha_{t+1})\hat{\mathbf{x}}\avg{t-\frac12} + \alpha_{t+1}\hat{\mathbf{x}}_t
,\label{eq_line_g}
\\
\hat{\mathbf{y}}\Xtrap{t}		&=	(1-\alpha_{t+1})\hat{\mathbf{y}}\avg{t-\frac12} + \alpha_{t+1}\hat{\mathbf{y}}_t
.\label{eq_line_h}
\end{align}
\end{subequations}
It is obvious to translate Eqs.~\eqref{eq_line_c}, \eqref{eq_line_d}, \eqref{eq_line_g}, \eqref{eq_line_h} into Lines~\ref{line_c}, \ref{line_d}, \ref{line_g}, \ref{line_h}, separately.
The rest translations are also straightforward, represented by Eqs.~\eqref{eq_line_a} into Line~\ref{line_a}
$$\begin{aligned}
&
\hat{\mathbf{x}}_{t-\frac12}	=	\hat{\mathbf{x}}_{t-1} - \stepsize_t \left( 
\nabla f(\hat{\mathbf{x}}\Xtrap{t-1}; \xi_{t-\frac12}) + \nabla_{\hat{\mathbf{x}}} h(\hat{\mathbf{x}}_{t-1}, \hat{\mathbf{y}}_{t-1}; \zeta_{t-\frac12})
\right) 
\\&\Leftrightarrow\,
\mathbf{x}_{t-\frac12}		=	\mathbf{x}_{t-1} - \stepsize_t \left( 
\nabla f(\mathbf{x}\Xtrap{t-1}; \xi_{t-\frac12}) + \nabla_{\mathbf{x}} h(\mathbf{x}_{t-1}, \mathbf{y}_{t-1}; \zeta_{t-\frac12})
\right) 
,
\end{aligned}$$
as well as Eqs.~\eqref{eq_line_f} into Line~\ref{line_f}
$$\begin{aligned}
&
\hat{\mathbf{y}}_t	=	\hat{\mathbf{y}}_{t-1} - \stepsize_t \left( 
- \nabla_{\hat{\mathbf{y}}} h(\hat{\mathbf{x}}_{t-\frac12}, \hat{\mathbf{y}}_{t-\frac12}; \zeta_t)
+ 
\nabla g(\hat{\mathbf{y}}\Xtrap{t-1}; \xi_{t-\frac12})
\right)
\\&\Leftrightarrow\,
\mathbf{y}_t	=	\mathbf{y}_{t-1} - \stepsize_t \cdot \tfrac{\mu_F}{\mu_G} \left( 
- \nabla_{\mathbf{y}} h(\mathbf{x}_{t-\frac12}, \mathbf{y}_{t-\frac12}; \zeta_t)
+ 
\nabla g(\hat{\mathbf{y}}\Xtrap{t-1}; \xi_{t-\frac12})
\right) 
.
\end{aligned}$$
It is also straightforward to justify that Assumptions \ref{assu_cvxsmth}, \ref{assu_coupling} and \ref{assu_noise} are rediscovered by reverting the scaling reduction from $\hat{\mathscr{F}}(\hat{\mathbf{x}}, \hat{\mathbf{y}})$ to $\mathscr{F}(\mathbf{x},\mathbf{y})$.
Therefore, it suffices to analyze Algorithm \ref{algo_AG_EG} for $\hat{\mathscr{F}}(\hat{\mathbf{x}}, \hat{\mathbf{y}})$ and due to this scaling reduction, we only need to prove all results for the case of $\ratio = 1$.
To keep the notations simple, till the rest of this work we slightly abuse the notations and remove the hats in all symbols.

\pb\subsection{Auxiliary lemmas}\label{sec_aux_lemma}
We first state the following basic lemma to handle the inner-product induced terms for extragradient analysis:
\begin{lemma}\label{lemm_PRecursion}
Given $\thetaholder, \foneholder, \ftwoholder \in \reals^d$ and also $\boneholder, \btwoholder$ that satisfies
\begin{align}\label{eqnProx1}
\foneholder	=	\thetaholder - \boneholder
,\qquad
\ftwoholder	=	\thetaholder - \btwoholder
,
\end{align}
then for any $\arbholder\in \reals^d$ we have
\beq\label{eqnPRecursion}
\langle \btwoholder,	\foneholder - \arbholder \rangle
\le
\frac{1}{2}\|\btwoholder - \boneholder\|^2 
+ 
\frac{1}{2} \left[
\|\thetaholder - \arbholder\|^2 - \|\ftwoholder - \arbholder\|^2 - \|\thetaholder - \foneholder\|^2
\right]
.
\eeq
\end{lemma}

Proof of Lemma \ref{lemm_PRecursion} is provided in \S\ref{sec_proof,lemm_PRecursion}.
Lemma \ref{lemm_PRecursion} is standard and commonly adopted in extragradient-based analysis; see Lemma 2 of \citep{chen2017accelerated} for one with similar flavor.

En route to our proofs of Theorems \ref{theo_S_AG_EG} and \ref{theo_S_AG_EG_Str} we first introduce some notations.
Let $\tilde{\mathbf{x}}\in\reals^n, \tilde{\mathbf{y}}\in\reals^m$ and let the \emph{pointwise primal-dual gap function} be
\beq\label{Vquantity_def}\begin{aligned}
\Vquantity(\mathbf{x}, \mathbf{y} \mid \tilde{\mathbf{x}}, \tilde{\mathbf{y}})
=
F(\mathbf{x}) - F(\tilde{\mathbf{x}})
+
G(\mathbf{y}) - G(\tilde{\mathbf{y}})
+
\langle \nabla_{\mathbf{x}} H(\tilde{\mathbf{x}}, \tilde{\mathbf{y}}),	\mathbf{x} - \tilde{\mathbf{x}} \rangle
+
\langle -\nabla_{\mathbf{y}} H(\tilde{\mathbf{x}}, \tilde{\mathbf{y}}),	\mathbf{y} - \tilde{\mathbf{y}} \rangle
,
\end{aligned}\eeq
and they can be separated $
\Vquantity(\mathbf{x}, \mathbf{y} \mid \tilde{\mathbf{x}}, \tilde{\mathbf{y}})
=
\Vquantity_F(\mathbf{x} \mid \tilde{\mathbf{x}}, \tilde{\mathbf{y}})
+
\Vquantity_G(\mathbf{y} \mid \tilde{\mathbf{x}}, \tilde{\mathbf{y}})
$ defined as
\beq\label{Vquantity_def2}\begin{aligned}
\Vquantity_F(\mathbf{x} \mid \tilde{\mathbf{x}}, \tilde{\mathbf{y}})
&=
F(\mathbf{x}) - F(\tilde{\mathbf{x}})
+
\left\langle \nabla_{\mathbf{x}} H(\tilde{\mathbf{x}}, \tilde{\mathbf{y}}),		\mathbf{x} - \tilde{\mathbf{x}} \right\rangle
,
\\
\Vquantity_G(\mathbf{y} \mid \tilde{\mathbf{x}}, \tilde{\mathbf{y}})
&=
G(\mathbf{y}) - G(\tilde{\mathbf{y}})
+
\left\langle -\nabla_{\mathbf{y}} H(\tilde{\mathbf{x}}, \tilde{\mathbf{y}}),	\mathbf{y} - \tilde{\mathbf{y}} \right\rangle
.
\end{aligned}\eeq
We prove that either of these two quantities is lower-bounded by a positive quadratic:

\begin{lemma}\label{lemm_QuanBdd}
We have both $F(\mathbf{x})$ and $G(\mathbf{y})$ are $\smoothFG$-smooth and $\strcvx$-strongly convex.
Furthermore, for any $\mathbf{x}\in \reals^n$ we have
\beq\label{QuanBddx}
\Vquantity_F(\mathbf{x} \mid \oholder_{\mathbf{x}}^\star, \oholder_{\mathbf{y}}^\star)
=
F(\mathbf{x}) - F(\oholder_{\mathbf{x}}^\star)
+ 
\left\langle \nabla_{\mathbf{x}} H(\oholder_{\mathbf{x}}^\star, \oholder_{\mathbf{y}}^\star),	\mathbf{x} - \oholder_{\mathbf{x}}^\star \right\rangle 
\ge
\frac{\strcvx}{2} \left\|\mathbf{x} - \oholder_{\mathbf{x}}^\star\right\|^2
,
\eeq
and for any $\mathbf{y}\in \reals^m$
\beq\label{QuanBddy}
\Vquantity_G(\mathbf{y} \mid \oholder_{\mathbf{x}}^\star, \oholder_{\mathbf{y}}^\star)
=
G(\mathbf{y}) - G(\oholder_{\mathbf{y}}^\star)
-
\left\langle \nabla_{\mathbf{y}} H(\oholder_{\mathbf{x}}^\star, \oholder_{\mathbf{y}}^\star), \mathbf{y} - \oholder_{\mathbf{y}}^\star \right\rangle
\ge
\frac{\strcvx}{2} \left\|\mathbf{y} - \oholder_{\mathbf{y}}^\star\right\|^2
,
\eeq
where these two $V$-quantities are defined as in \eqref{Vquantity_def2}.
\end{lemma}

Proof of Lemma \ref{lemm_QuanBdd} is provided in \S\ref{sec_proof,lemm_QuanBdd}.
Our final auxiliary lemma on the key properties on stepsizes spells as follows:

\begin{lemma}\label{lemm_ssproperties}
Set $
\square
\equiv
\frac{\tilde{\sigma} [\Tholder(\Tholder+1)^2]^{1/2}}{\Cholder\sqrt{\Exs[
\|\mathbf{x}_0 - \oholder_{\mathbf{x}}^\star\|^2
+
\ratio\|\mathbf{y}_0 - \oholder_{\mathbf{y}}^\star\|^2
]}}
$.
Our stepsize choice \eqref{eq_stepsize} satisfies
(i)
$\stepsize_t\le \frac{t}{\square}$;
(ii)
$\left(\frac{t}{\stepsize_t}: t\ge 1\right)$ is a nonnegative, nondecreasing arithmetic sequence with common difference $\sqrt{\tfrac{1+\beta}{\quantile}}\smoothBLin$;
(iii)
$\smoothBLin\stepsize_t\le 1$, and
(iv)
the stepsize condition
\begin{align}\label{stepsize_cond}
\quantile - \frac{2\smoothFG}{t+1} \stepsize_t - (1+\beta)\smoothBLin^2 \stepsize_t^2
\ge
0
.
\end{align}
\end{lemma}

Proof of Lemma \ref{lemm_ssproperties} is provided in \S\ref{sec_proof,lemm_ssproperties}.

\pb\subsection{Proof of Theorem \ref{theo_S_AG_EG}}\label{sec_proof,theo_S_AG_EG}
Throughout the proof we assume $\smoothFG + \smoothBLin > 0$ without loss of generality, since otherwise the result holds trivially.
Due to the scaling reduction argument in \S\ref{sec_scalingreduction}, we assume without loss of generality that $\ratio = 1$.

We introduce some notations.
Denote the (squared) metric by $
\distancemetric(\mathbf{x}, \mathbf{y}; \tilde{\mathbf{x}}, \tilde{\mathbf{y}})
\equiv
\left\|\mathbf{x} - \tilde{\mathbf{x}}\right\|^2
+
\left\|\mathbf{y} - \tilde{\mathbf{y}}\right\|^2
$, and denote the incurred stochastic noise terms
\beq\label{noisedef}
\begin{aligned}
\noiseFG^{t-\frac12}
&\equiv
\begin{bmatrix}
\nabla f(\mathbf{x}\Xtrap{t-1}; \xi_{t-\frac12})	-	\nabla F(\mathbf{x}\Xtrap{t-1})
\\
\nabla g(\mathbf{y}\Xtrap{t-1}; \xi_{t-\frac12})	-	\nabla G(\mathbf{y}\Xtrap{t-1})
\end{bmatrix}
,
\qquad
\noiseBLin^{t-\frac12}
\equiv
\begin{bmatrix}
\nabla_{\mathbf{x}} h(\mathbf{x}_{t-1}, \mathbf{y}_{t-1}; \zeta_{t-\frac12})	-	\nabla_{\mathbf{x}} H(\mathbf{x}_{t-1}, \mathbf{y}_{t-1})
\\
-\nabla_{\mathbf{y}} h(\mathbf{x}_{t-1}, \mathbf{y}_{t-1}; \zeta_{t-\frac12})	+	\nabla_{\mathbf{y}} H(\mathbf{x}_{t-1}, \mathbf{y}_{t-1})
\end{bmatrix}
,
\\
&\noiseBLin^t
\equiv
\begin{bmatrix}
\nabla_{\mathbf{x}} h(\mathbf{x}_{t-\frac12}, \mathbf{y}_{t-\frac12}; \zeta_t)	-	\nabla_{\mathbf{x}} H(\mathbf{x}_{t-\frac12}, \mathbf{y}_{t-\frac12})
\\
-\nabla_{\mathbf{y}} h(\mathbf{x}_{t-\frac12}, \mathbf{y}_{t-\frac12}; \zeta_t)	+	\nabla_{\mathbf{y}} H(\mathbf{x}_{t-\frac12}, \mathbf{y}_{t-\frac12})
\end{bmatrix}
.
\end{aligned}\eeq
For our martingale analysis we adopt the filtrations $
\mathcal{F}^\xi_t \equiv \sigma\left(\xi_s: s = \frac12, \frac32, \dots, s\le t \right)
$ and $
\mathcal{F}^\zeta_t \equiv \sigma\left(\zeta_s: s = \frac12, 1, \frac32, \dots, s\le t \right)
$, and also $\mathcal{F}_t \equiv \sigma( \mathcal{F}^\xi_t \cup \mathcal{F}^\zeta_t )$ be the $\sigma$-algebra generated by the union of $\mathcal{F}^\xi_t$ and $\mathcal{F}^\zeta_t$.

We are ready for the proof which proceeds as the following steps:

\paragraph{Step 1.}
We prove the following lemma:
\begin{lemma}\label{lemm_aggregated}
For arbitrary $\tilde{\mathbf{x}}\in \reals^n, \tilde{\mathbf{y}}\in \reals^m$ and $\alpha_t\in (0,1]$ the iterates of Algorithm~\ref{algo_AG_EG} ($\Sholder=1$) satisfy for $t=1,\dots,\Tholder$, almost surely
\begin{equation}\label{aggregated}
\begin{aligned}
\lefteqn{
\Vquantity(\mathbf{x}\avg{t-\frac12}, \mathbf{y}\avg{t-\frac12} \mid \tilde{\mathbf{x}}, \tilde{\mathbf{y}})
-
(1-\alpha_t)\Vquantity(\mathbf{x}\avg{t-\frac32}, \mathbf{y}\avg{t-\frac32} \mid \tilde{\mathbf{x}}, \tilde{\mathbf{y}})
}
\\&\le
\alpha_t\langle\nabla F(\mathbf{x}\Xtrap{t-1}) + \nabla_{\mathbf{x}} H(\mathbf{x}_{t-\frac12},\mathbf{y}_{t-\frac12}),	\mathbf{x}_{t-\frac12} - \tilde{\mathbf{x}} \rangle
+
\alpha_t\langle -\nabla_{\mathbf{y}} H(\mathbf{x}_{t-\frac12},\mathbf{y}_{t-\frac12}) + \nabla G(\mathbf{y}\Xtrap{t-1}),	\mathbf{y}_{t-\frac12} - \tilde{\mathbf{y}} \rangle
\\&\quad\,
+
\tfrac{\alpha_t^2\smoothFG}{2}\distancemetric(\mathbf{x}_{t-\frac12}, \mathbf{y}_{t-\frac12}; \mathbf{x}_{t-1}, \mathbf{y}_{t-1})
.
\end{aligned}\end{equation}
\end{lemma}
Note the proof only relies on the interpolation updates in our algorithm as in Lines \ref{line_c}, \ref{line_d}, \ref{line_g} and \ref{line_h}, and hence this result holds in a per-trajectory (almost-sure) fashion.

\pb\begin{proof}[Proof of Lemma \ref{lemm_aggregated}]
From the convexity and $\smoothFG$-smoothness of $F$, we know that for arbitrary $\tilde{\mathbf{x}}, \tilde{\mathbf{y}}$
\begin{align*}
\lefteqn{
F(\mathbf{x}\avg{t-\frac12}) - F(\tilde{\mathbf{x}})
=
F(\mathbf{x}\avg{t-\frac12}) - F(\mathbf{x}\Xtrap{t-1}) - \left(F(\tilde{\mathbf{x}}) - F(\mathbf{x}\Xtrap{t-1})\right)
}
\\&\le
\langle\nabla F(\mathbf{x}\Xtrap{t-1}), \mathbf{x}\avg{t-\frac12} - \mathbf{x}\Xtrap{t-1}\rangle
+
\tfrac{\smoothFG}{2} \left\| \mathbf{x}\avg{t-\frac12} - \mathbf{x}\Xtrap{t-1} \right\|^2
-
\langle 
\nabla F(\mathbf{x}\Xtrap{t-1}), \tilde{\mathbf{x}} - \mathbf{x}\Xtrap{t-1}
\rangle
.
\end{align*}
Taking $\tilde{\mathbf{x}} = \mathbf{x}\avg{t-\frac32} $ in the above inequality, we have
\begin{align*}
\lefteqn{
F(\mathbf{x}\avg{t-\frac12}) - F(\mathbf{x}\avg{t-\frac32})
=
F(\mathbf{x}\avg{t-\frac12}) - F(\mathbf{x}\Xtrap{t-1}) - \left(F(\mathbf{x}\avg{t-\frac32}) - F(\mathbf{x}\Xtrap{t-1})\right)
}
\\&\le
\langle\nabla F(\mathbf{x}\Xtrap{t-1}), \mathbf{x}\avg{t-\frac12} - \mathbf{x}\Xtrap{t-1}\rangle
+
\tfrac{\smoothFG}{2} \left\| \mathbf{x}\avg{t-\frac12} - \mathbf{x}\Xtrap{t-1} \right\|^2
-
\langle 
\nabla F(\mathbf{x}\Xtrap{t-1}), \mathbf{x}\avg{t-\frac32} - \mathbf{x}\Xtrap{t-1}
\rangle
.
\end{align*}
Multiplying the first display by $\alpha_t$ and the second display by $(1-\alpha_t)$ and adding them up, we have
\begin{align}\label{eq_acc_1}
\begin{aligned}
\lefteqn{
F(\mathbf{x}\avg{t-\frac12}) - (1-\alpha_t) F(\mathbf{x}\avg{t-\frac32}) - \alpha_t F(\tilde{\mathbf{x}})
}
\\&\le
\langle\nabla F(\mathbf{x}\Xtrap{t-1}), \mathbf{x}\avg{t-\frac12} - \mathbf{x}\Xtrap{t-1}\rangle
+
\tfrac{\smoothFG}{2} \left\| \mathbf{x}\avg{t-\frac12} - \mathbf{x}\Xtrap{t-1} \right\|^2
-
\langle 
\nabla F(\mathbf{x}\Xtrap{t-1}), (1-\alpha_t)\mathbf{x}\avg{t-\frac32} + \alpha_t\tilde{\mathbf{x}} - \mathbf{x}\Xtrap{t-1}
\rangle
\\&\le
\langle\nabla F(\mathbf{x}\Xtrap{t-1}), \alpha_t (\mathbf{x}_{t-\frac12} - \mathbf{x}_{t-1}) \rangle
+
\tfrac{\smoothFG}{2}\|\alpha_t (\mathbf{x}_{t-\frac12} - \mathbf{x}_{t-1})\|^2
-
\langle 
\nabla F(\mathbf{x}\Xtrap{t-1}), \alpha_t (\tilde{\mathbf{x}} - \mathbf{x}_{t-1})
\rangle
\\&=
\alpha_t\langle\nabla F(\mathbf{x}\Xtrap{t-1}), \mathbf{x}_{t-\frac12} - \tilde{\mathbf{x}} \rangle
+
\tfrac{\alpha_t^2\smoothFG}{2}\|\mathbf{x}_{t-\frac12} - \mathbf{x}_{t-1}\|^2
,
\end{aligned}\end{align}
where we applied the fact from our update rules that $
\mathbf{x}\avg{t-\frac12} - \mathbf{x}\Xtrap{t-1}
= 
\alpha_t(\mathbf{x}_{t-\frac12} - \mathbf{x}_{t-1})
$.
Following an analogous argument for $G$ we obtain
\begin{align}\label{eq_acc_2}
G(\mathbf{y}\avg{t-\frac12}) - (1-\alpha_t) G(\mathbf{y}\avg{t-\frac32}) - \alpha_t G(\tilde{\mathbf{y}})
\le
\alpha_t\langle\nabla G(\mathbf{y}\Xtrap{t-1}), \mathbf{y}_{t-\frac12} - \tilde{\mathbf{y}} \rangle
+
\tfrac{\alpha_t^2\smoothFG}{2}\|\mathbf{y}_{t-\frac12} - \mathbf{y}_{t-1}\|^2
.
\end{align}
On the other hand, due to Lines \ref{line_c} and \ref{line_d} we have
$$\begin{aligned}
&
\langle \nabla_{\mathbf{x}} H(\tilde{\mathbf{x}},\tilde{\mathbf{y}}), \mathbf{x}\avg{t-\frac12} - \tilde{\mathbf{x}} \rangle - (1-\alpha_t)\langle \nabla_{\mathbf{x}} H(\tilde{\mathbf{x}},\tilde{\mathbf{y}}), \mathbf{x}\avg{t-\frac32} - \tilde{\mathbf{x}} \rangle
\\&=
\langle \nabla_{\mathbf{x}} H(\tilde{\mathbf{x}},\tilde{\mathbf{y}}),	\mathbf{x}\avg{t-\frac12} - \tilde{\mathbf{x}} - (1-\alpha_t)(\mathbf{x}\avg{t-\frac32} - \mathbf{x}) \rangle
=
\alpha_t\langle \nabla_{\mathbf{x}} H(\tilde{\mathbf{x}},\tilde{\mathbf{y}}),	\mathbf{x}_{t-\frac12} - \tilde{\mathbf{x}} \rangle
,
\end{aligned}$$
and analogously
$$\begin{aligned}
&
\langle -\nabla_{\mathbf{y}} H(\tilde{\mathbf{x}},\tilde{\mathbf{y}}), \mathbf{y}\avg{t-\frac12} - \tilde{\mathbf{y}} \rangle - (1-\alpha_t)\langle -\nabla_{\mathbf{y}} H(\tilde{\mathbf{x}},\tilde{\mathbf{y}}), \mathbf{y}\avg{t-\frac32} - \tilde{\mathbf{y}} \rangle
\\&=
\langle -\nabla_{\mathbf{y}} H(\tilde{\mathbf{x}},\tilde{\mathbf{y}}),	\mathbf{y}\avg{t-\frac12} - \tilde{\mathbf{y}} - (1-\alpha_t)(\mathbf{y}\avg{t-\frac32} - \tilde{\mathbf{y}}) \rangle
=
\alpha_t\langle -\nabla_{\mathbf{y}} H(\tilde{\mathbf{x}},\tilde{\mathbf{y}}),	\mathbf{y}_{t-\frac12} - \tilde{\mathbf{y}} \rangle
.
\end{aligned}$$
Due to our assumption on $H$ we have
\pb$$\begin{aligned}
\lefteqn{
\langle \nabla_{\mathbf{x}} H(\tilde{\mathbf{x}},\tilde{\mathbf{y}}), \mathbf{x}_{t-\frac12} - \tilde{\mathbf{x}} \rangle
+
\langle -\nabla_{\mathbf{y}} H(\tilde{\mathbf{x}},\tilde{\mathbf{y}}), \mathbf{y}_{t-\frac12} - \tilde{\mathbf{y}} \rangle
}
\\&\le
\langle \nabla_{\mathbf{x}} H(\mathbf{x}_{t-\frac12},\mathbf{y}_{t-\frac12}), \mathbf{x}_{t-\frac12} - \tilde{\mathbf{x}} \rangle
+
\langle -\nabla_{\mathbf{y}} H(\mathbf{x}_{t-\frac12},\mathbf{y}_{t-\frac12}), \mathbf{y}_{t-\frac12} - \tilde{\mathbf{y}} \rangle
.
\end{aligned}$$
Combining the above three displays together yields 
\begin{equation}\label{eq_acc_3}
\begin{aligned}
&
\langle \nabla_{\mathbf{x}} H(\tilde{\mathbf{x}},\tilde{\mathbf{y}}), \mathbf{x}\avg{t-\frac12} - \tilde{\mathbf{x}} \rangle 
- (1-\alpha_t)\langle \nabla_{\mathbf{x}} H(\tilde{\mathbf{x}},\tilde{\mathbf{y}}), \mathbf{x}\avg{t-\frac32} - \tilde{\mathbf{x}} \rangle
\\&\hspace{1in}
+
\langle -\nabla_{\mathbf{y}} H(\tilde{\mathbf{x}},\tilde{\mathbf{y}}), \mathbf{y}\avg{t-\frac12} - \tilde{\mathbf{y}} \rangle - (1-\alpha_t)\langle -\nabla_{\mathbf{y}} H(\tilde{\mathbf{x}},\tilde{\mathbf{y}}), \mathbf{y}\avg{t-\frac32} - \tilde{\mathbf{y}} \rangle 
\\&\le
\alpha_t \left[ 
\langle \nabla_{\mathbf{x}} H(\mathbf{x}_{t-\frac12},\mathbf{y}_{t-\frac12}), \mathbf{x}_{t-\frac12} - \tilde{\mathbf{x}} \rangle
+
\langle -\nabla_{\mathbf{y}} H(\mathbf{x}_{t-\frac12},\mathbf{y}_{t-\frac12}), \mathbf{y}_{t-\frac12} - \tilde{\mathbf{y}} \rangle
\right]
.
\end{aligned}\end{equation}
Now, summing up Eqs.~\eqref{eq_acc_1},~\eqref{eq_acc_2} and~\eqref{eq_acc_3} and noting the definition of $\Vquantity$ in \eqref{Vquantity_def}, we have
$$\begin{aligned}
&
\Vquantity(\mathbf{x}\avg{t-\frac12}, \mathbf{y}\avg{t-\frac12} \mid \tilde{\mathbf{x}}, \tilde{\mathbf{y}})
-
(1-\alpha_t)\Vquantity(\mathbf{x}\avg{t-\frac32}, \mathbf{y}\avg{t-\frac32} \mid \tilde{\mathbf{x}}, \tilde{\mathbf{y}})
\\&=
F(\mathbf{x}\avg{t-\frac12}) - (1-\alpha_t) F(\mathbf{x}\avg{t-\frac32}) - \alpha_t F(\tilde{\mathbf{x}})
+
G(\mathbf{y}\avg{t-\frac12}) - (1-\alpha_t) G(\mathbf{y}\avg{t-\frac32}) - \alpha_t G(\tilde{\mathbf{y}})
\\&\quad\,
+
\langle \nabla_{\mathbf{x}} H(\tilde{\mathbf{x}},\tilde{\mathbf{y}}), \mathbf{x}\avg{t-\frac12} - \tilde{\mathbf{x}} \rangle - (1-\alpha_t)\langle \nabla_{\mathbf{x}} H(\tilde{\mathbf{x}},\tilde{\mathbf{y}}), \mathbf{x}\avg{t-\frac32} - \tilde{\mathbf{x}} \rangle
\\&\hspace{1in}
+
\langle -\nabla_{\mathbf{y}} H(\tilde{\mathbf{x}},\tilde{\mathbf{y}}), \mathbf{y}\avg{t-\frac12} - \tilde{\mathbf{y}} \rangle - (1-\alpha_t)\langle -\nabla_{\mathbf{y}} H(\tilde{\mathbf{x}},\tilde{\mathbf{y}}), \mathbf{y}\avg{t-\frac32} - \tilde{\mathbf{y}} \rangle
\\&\le
\alpha_t\left[
\langle
\nabla F(\mathbf{x}\Xtrap{t-1})			+	\nabla_{\mathbf{x}} H(\mathbf{x}_{t-\frac12},\mathbf{y}_{t-\frac12})
,
\mathbf{x}_{t-\frac12}				-	\tilde{\mathbf{x}}
\rangle
+
\langle
-\nabla_{\mathbf{y}} H(\mathbf{x}_{t-\frac12},\mathbf{y}_{t-\frac12})	+	\nabla G(\mathbf{y}\Xtrap{t-1}) 
,
\mathbf{y}_{t-\frac12}				-	\tilde{\mathbf{y}}
\rangle
\right]
\\&\quad\,
+ 
\tfrac{\alpha_t^2\smoothFG}{2}\left[
\|\mathbf{x}_{t-\frac12} - \mathbf{x}_{t-1}\|^2
+
\|\mathbf{y}_{t-\frac12} - \mathbf{y}_{t-1}\|^2
\right]
,
\end{aligned}$$
and hence conclude \eqref{aggregated} and Lemma \ref{lemm_aggregated}.
\end{proof}

\paragraph{Step 2.}
We target to prove, for our choice of $\stepsize_t$ that satisfies, for a given $\quantile\in (0, 1)$, \eqref{stepsize_cond} of Lemma \ref{lemm_ssproperties}(iv) that
\blue{$\quantile - \frac{2\smoothFG}{t+1} \stepsize_t - (1+\beta)\smoothBLin^2 \stepsize_t^2	\ge	0$}
we have for any $\tilde{\mathbf{x}}\in \reals^n$, $\tilde{\mathbf{y}}\in \reals^m$ and $\calT = 1,\dots,\Tholder$ that
\begin{align}\label{recursion_bdd_key}
\begin{aligned}
\lefteqn{
\calT(\calT + 1) \Exs[\Vquantity(\mathbf{x}\avg{\calT-\frac12}, \mathbf{y}\avg{\calT-\frac12} \mid \tilde{\mathbf{x}}, \tilde{\mathbf{y}})]
\blue{
+
\frac{\calT}{\stepsize_{\calT}} \Exs[ \distancemetric(\mathbf{x}_{\calT}, \mathbf{y}_{\calT}; \tilde{\mathbf{x}}, \tilde{\mathbf{y}})]
}
}
\\&\le
\frac{1}{\stepsize_1} \Exs[\distancemetric(\mathbf{x}_0, \mathbf{y}_0; \tilde{\mathbf{x}}, \tilde{\mathbf{y}})]
+
\sqrt{\tfrac{1+\beta}{\quantile}}\smoothBLin\sum_{t=2}^{\calT} \Exs[ \distancemetric(\mathbf{x}_{t-1}, \mathbf{y}_{t-1}; \tilde{\mathbf{x}}, \tilde{\mathbf{y}})]
+
\frac{\calT(\calT+\frac12)(\calT+1)}{[\Tholder(\Tholder+1)^2]^{1/2}}\cdot\Cholder\sigma \sqrt{\Exs[\distancemetric(\mathbf{x}_0, \mathbf{y}_0; \tilde{\mathbf{x}}, \tilde{\mathbf{y}})]}
.
\end{aligned}\end{align}
To bound the inner-product terms in \eqref{aggregated}, by setting $
\foneholder	=	\mathbf{x}_{t-\frac12}
$, $
\thetaholder	=	\mathbf{x}_{t-1}
$, $
\ftwoholder	=	\mathbf{x}_t
$, $
\boneholder	=	\stepsize_t\bigg(\nabla f(\mathbf{x}\Xtrap{t-1}; \xi_{t-\frac12})	+	\nabla_{\mathbf{x}} h(\mathbf{x}_{t-1}, \mathbf{y}_{t-1}; \zeta_{t-\frac12})\bigg)
$, $
\btwoholder	=	\stepsize_t\bigg(\nabla f(\mathbf{x}\Xtrap{t-1}; \xi_{t-\frac12})	+	\nabla_{\mathbf{x}} h(\mathbf{x}_{t-\frac12}, \mathbf{y}_{t-\frac12}; \zeta_t)\bigg)
$ as in Lemma~\ref{lemm_PRecursion} (with $\arbholder = \tilde{\mathbf{x}}$), we have
$$\begin{aligned}
\lefteqn{
\stepsize_t\langle
\nabla f(\mathbf{x}\Xtrap{t-1}; \xi_{t-\frac12})
+
\nabla_{\mathbf{x}} h(\mathbf{x}_{t-\frac12}, \mathbf{y}_{t-\frac12}; \zeta_t)
,
\mathbf{x}_{t-\frac12} - \tilde{\mathbf{x}}
\rangle
}
\\&\le
\frac12\left(
\|\mathbf{x}_{t-1} - \tilde{\mathbf{x}}\|^2
-
\|\mathbf{x}_t - \tilde{\mathbf{x}}\|^2
-
\|\mathbf{x}_{t-\frac12} - \mathbf{x}_{t-1}\|^2
\right)
+
\frac{\stepsize_t^2}{2}\|
\nabla_{\mathbf{x}} h(\mathbf{x}_{t-\frac12}, \mathbf{y}_{t-\frac12}; \zeta_t)
-
\nabla_{\mathbf{x}} h(\mathbf{x}_{t-1}, \mathbf{y}_{t-1}; \zeta_{t-\frac12})
\|^2
,
\end{aligned}$$
where Young’s inequality combined with the martingale structure yields (also noting \eqref{noisedef})
$$\begin{aligned}
\lefteqn{
\Exs\|
\nabla_{\mathbf{x}} h(\mathbf{x}_{t-\frac12}, \mathbf{y}_{t-\frac12}; \zeta_t)
-
\nabla_{\mathbf{x}} h(\mathbf{x}_{t-1}, \mathbf{y}_{t-1}; \zeta_{t-\frac12})
\|^2
}
\\&=
\Exs\|
\nabla_{\mathbf{x}} H(\mathbf{x}_{t-\frac12}, \mathbf{y}_{t-\frac12})
-
\nabla_{\mathbf{x}} H(\mathbf{x}_{t-1}, \mathbf{y}_{t-1})
-
\noiseBLin^{1,t-\frac12}
\|^2
+
\Exs\|\noiseBLin^{1,t}\|^2
\\&\le
(1+\beta)\smoothBLin^2\Exs\|\mathbf{y}_{t-\frac12}	-	\mathbf{y}_{t-1}\|^2
+
(1+\tfrac{1}{\beta})\Exs\|\noiseBLin^{1,t-\frac12}\|^2
+
\Exs\|\noiseBLin^{1,t}\|^2
.
\end{aligned}$$
Combining the above two displays with expectation taken gives
$$\begin{aligned}
\lefteqn{
\stepsize_t\Exs\langle
\nabla f(\mathbf{x}\Xtrap{t-1}; \xi_{t-\frac12})
+
\nabla_{\mathbf{x}} h(\mathbf{x}_{t-\frac12}, \mathbf{y}_{t-\frac12}; \zeta_t)
,
\mathbf{x}_{t-\frac12} - \tilde{\mathbf{x}}
\rangle
}
\\&\le
\frac12\left(
\Exs\|\mathbf{x}_{t-1} - \tilde{\mathbf{x}}\|^2
-
\Exs\|\mathbf{x}_t - \tilde{\mathbf{x}}\|^2
-
\Exs\|\mathbf{x}_{t-\frac12} - \mathbf{x}_{t-1}\|^2
\right)
\\&\quad\,
+
\frac{\stepsize_t^2}{2}\left(
(1+\beta)\smoothBLin^2\Exs\|\mathbf{y}_{t-\frac12}	-	\mathbf{y}_{t-1}\|^2
+
(1+\tfrac{1}{\beta})\Exs\|\noiseBLin^{1,t-\frac12}\|^2
+
\Exs\|\noiseBLin^{1,t}\|^2
\right)
.
\end{aligned}$$
Analogously by setting the appropriate parameters, we have
$$\begin{aligned}
\lefteqn{
\stepsize_t\Exs\langle
-\nabla_{\mathbf{y}} h(\mathbf{x}_{t-\frac12}, \mathbf{y}_{t-\frac12}; \zeta_t)
+
\nabla g(\mathbf{y}\Xtrap{t-1}; \xi_{t-\frac12})
,
\mathbf{y}_{t-\frac12} - \tilde{\mathbf{y}}
\rangle
}
\\&\le
\frac12\left(
\Exs\|\mathbf{y}_{t-1} - \tilde{\mathbf{y}} \|^2 
-
\Exs\|\mathbf{y}_t - \tilde{\mathbf{y}} \|^2
-
\Exs\|\mathbf{y}_{t-\frac12} - \mathbf{y}_{t-1} \|^2
\right)
\\&\quad\,
+
\frac{\stepsize_t^2}{2} \left(
(1+\beta)\smoothBLin^2\Exs\|\mathbf{x}_{t-\frac12}	-	\mathbf{x}_{t-1}\|^2
+
(1+\tfrac{1}{\beta})\Exs\|\noiseBLin^{2,t-\frac12}\|^2
+
\Exs\|\noiseBLin^{2,t}\|^2
\right)
.
\end{aligned}$$
Combining the last two displays gives
\begin{align}
\lefteqn{
\stepsize_t\Exs\langle
\nabla f(\mathbf{x}\Xtrap{t-1}; \xi_{t-\frac12})
+
\nabla_{\mathbf{x}} h(\mathbf{x}_{t-\frac12}, \mathbf{y}_{t-\frac12}; \zeta_t)
,
\mathbf{x}_{t-\frac12} - \tilde{\mathbf{x}}
\rangle
}
\nonumber\\&\hspace{1in}\,
+
\stepsize_t\Exs\langle
-\nabla_{\mathbf{y}} h(\mathbf{x}_{t-\frac12}, \mathbf{y}_{t-\frac12}; \zeta_t)
+
\nabla g(\mathbf{y}\Xtrap{t-1}; \xi_{t-\frac12})
,
\mathbf{y}_{t-\frac12} - \tilde{\mathbf{y}}
\rangle
\nonumber\\&\le
\frac12\left(
\Exs[\distancemetric(\mathbf{x}_{t-1}, \mathbf{y}_{t-1}; \tilde{\mathbf{x}}, \tilde{\mathbf{y}})]
-
\Exs[\distancemetric(\mathbf{x}_t, \mathbf{y}_t; \tilde{\mathbf{x}}, \tilde{\mathbf{y}})]
\right)
-
\frac{1 - (1+\beta)\smoothBLin^2 \stepsize_t^2}{2}
\Exs[\distancemetric(\mathbf{x}_{t-\frac12}, \mathbf{y}_{t-\frac12}; \mathbf{x}_{t-1}, \mathbf{y}_{t-1})]
\nonumber\\&\quad\,
+
\frac{\stepsize_t^2}{2}\left(
(1+\tfrac{1}{\beta})\Exs[
\|\noiseBLin^{1,t-\frac12}\|^2	+	\|\noiseBLin^{2,t-\frac12}\|^2
]
+
\Exs[
\|\noiseBLin^{1,t}\|^2			+	\|\noiseBLin^{2,t}\|^2
]
\right)
.
\label{keyeq}
\end{align}
Therefore plugging the above \eqref{keyeq} into \eqref{aggregated} of Lemma \ref{lemm_aggregated} with $\alpha_t = \frac{2}{t+1}$ and expectation taken, we have from \eqref{Vquantity_def2} that
$$\begin{aligned}
\lefteqn{
\Exs[\Vquantity(\mathbf{x}\avg{t-\frac12}, \mathbf{y}\avg{t-\frac12} \mid \tilde{\mathbf{x}}, \tilde{\mathbf{y}})]
-
\frac{t-1}{t+1}\Exs[\Vquantity(\mathbf{x}\avg{t-\frac32}, \mathbf{y}\avg{t-\frac32} \mid \tilde{\mathbf{x}}, \tilde{\mathbf{y}})]
}
\\&\le
\frac{2}{t+1}\Exs\langle
\nabla F(\mathbf{x}\Xtrap{t-1})
+\nabla_{\mathbf{x}} H(\mathbf{x}_{t-\frac12}, \mathbf{y}_{t-\frac12})
,
\mathbf{x}_{t-\frac12} - \tilde{\mathbf{x}}
\rangle
\\&\quad\,
+
\frac{2}{t+1}\Exs\langle
-\nabla_{\mathbf{y}} H(\mathbf{x}_{t-\frac12}, \mathbf{y}_{t-\frac12})
+\nabla G(\mathbf{y}\Xtrap{t-1})
,
\mathbf{y}_{t-\frac12} - \tilde{\mathbf{y}}
\rangle
+ 
\frac{2\smoothFG}{(t+1)^2}\Exs[\distancemetric(\mathbf{x}_{t-\frac12}, \mathbf{y}_{t-\frac12}; \mathbf{x}_{t-1}, \mathbf{y}_{t-1})]
\\&=
\frac{2}{t+1}\Exs\langle
\nabla f(\mathbf{x}\Xtrap{t-1}; \xi_{t-\frac12})
+\nabla_{\mathbf{x}} h(\mathbf{x}_{t-\frac12}, \mathbf{y}_{t-\frac12}; \zeta_t)
,
\mathbf{x}_{t-\frac12} - \tilde{\mathbf{x}}
\rangle
\\&\quad\,
+
\frac{2}{t+1}\Exs\langle
-\nabla_{\mathbf{y}} h(\mathbf{x}_{t-\frac12}, \mathbf{y}_{t-\frac12}; \zeta_t)
+\nabla g(\mathbf{y}\Xtrap{t-1}; \xi_{t-\frac12})
,
\mathbf{y}_{t-\frac12} - \tilde{\mathbf{y}}
\rangle
+
\frac{2\smoothFG}{(t+1)^2}\Exs[\distancemetric(\mathbf{x}_{t-\frac12}, \mathbf{y}_{t-\frac12}; \mathbf{x}_{t-1}, \mathbf{y}_{t-1})]
\\&\quad\,
-
\frac{2}{t+1}\Exs\langle
\noiseFG^{1,t-\frac12} + \noiseBLin^{1,t}
,
\mathbf{x}_{t-\frac12} - \tilde{\mathbf{x}}
\rangle
-
\frac{2}{t+1}\Exs\langle
\noiseFG^{2,t-\frac12} + \noiseBLin^{2,t}
,
\mathbf{y}_{t-\frac12} - \tilde{\mathbf{y}}
\rangle
\\&\le
\frac{1}{(t+1)\stepsize_t}\left(
\Exs[\distancemetric(\mathbf{x}_{t-1}, \mathbf{y}_{t-1}; \tilde{\mathbf{x}}, \tilde{\mathbf{y}})] 
- 
\Exs[\distancemetric(\mathbf{x}_t, \mathbf{y}_t; \tilde{\mathbf{x}}, \tilde{\mathbf{y}})] 
\right)
\\&\quad\,
-
\frac{1}{(t+1)\stepsize_t}\left(1 - \frac{2\smoothFG}{t+1} \stepsize_t - (1+\beta)\smoothBLin^2 \stepsize_t^2\right)\Exs[\distancemetric(\mathbf{x}_{t-\frac12}, \mathbf{y}_{t-\frac12}; \mathbf{x}_{t-1}, \mathbf{y}_{t-1})]
\\&\quad\,
+
\frac{\stepsize_t}{t+1}\left(
(1+\tfrac{1}{\beta})\Exs\|\noiseBLin^{1,t-\frac12}\|^2
+
\Exs\|\noiseBLin^{1,t}\|^2
\right)
+
\frac{\stepsize_t}{t+1}\left(
(1+\tfrac{1}{\beta})\Exs\|\noiseBLin^{2,t-\frac12}\|^2
+
\Exs\|\noiseBLin^{2,t}\|^2
\right)
\\&\quad\,
-
\frac{2}{t+1}\Exs\langle
\noiseFG^{1,t-\frac12} + \noiseBLin^{1,t}
,
\mathbf{x}_{t-\frac12} - \tilde{\mathbf{x}}
\rangle
-
\frac{2}{t+1}\Exs\langle
\noiseFG^{2,t-\frac12} + \noiseBLin^{2,t}
,
\mathbf{y}_{t-\frac12} - \tilde{\mathbf{y}}
\rangle
.
\end{aligned}$$
With some manipulations we obtain
\allowdisplaybreaks
\begin{align}
\lefteqn{
\Exs[\Vquantity(\mathbf{x}\avg{t-\frac12}, \mathbf{y}\avg{t-\frac12} \mid \tilde{\mathbf{x}}, \tilde{\mathbf{y}})] 
- 
\frac{t-1}{t+1}\Exs[\Vquantity(\mathbf{x}\avg{t-\frac32}, \mathbf{y}\avg{t-\frac32} \mid \tilde{\mathbf{x}}, \tilde{\mathbf{y}})]
}
\nonumber\\&\le
\frac{1}{(t+1)\stepsize_t}\left(
\Exs[ \distancemetric(\mathbf{x}_{t-1}, \mathbf{y}_{t-1}; \tilde{\mathbf{x}}, \tilde{\mathbf{y}})] 
- 
\Exs[ \distancemetric(\mathbf{x}_t, \mathbf{y}_t; \tilde{\mathbf{x}}, \tilde{\mathbf{y}})] 
\right)
\nonumber\\&\quad\,
-
\frac{1}{(t+1)\stepsize_t}\left(
\quantile - \frac{2\smoothFG}{t+1} \stepsize_t - (1+\beta)\smoothBLin^2 \stepsize_t^2
\right)\Exs[\distancemetric(\mathbf{x}_{t-\frac12}, \mathbf{y}_{t-\frac12}; \mathbf{x}_{t-1}, \mathbf{y}_{t-1})]
\nonumber\\&\quad\,
+
\frac{\stepsize_t}{t+1}\left(
(1+\tfrac{1}{\beta})\Exs\|\noiseBLin^{1,t-\frac12}\|^2
+
\Exs\|\noiseBLin^{1,t}\|^2
\right)
+
\frac{\stepsize_t}{t+1}\left(
(1+\tfrac{1}{\beta})\Exs\|\noiseBLin^{2,t-\frac12}\|^2
+
\Exs\|\noiseBLin^{2,t}\|^2
\right)
\nonumber\\&\quad\,
\blue{-
\frac{1-\quantile}{(t+1)\stepsize_t}\Exs[\distancemetric(\mathbf{x}_{t-\frac12}, \mathbf{y}_{t-\frac12}; \mathbf{x}_{t-1}, \mathbf{y}_{t-1})]
-
\frac{2}{t+1}\Exs\langle
\noiseFG^{1,t-\frac12}
,
\mathbf{x}_{t-\frac12} - \mathbf{x}_{t-1}
\rangle
-
\frac{2}{t+1}\Exs\langle
\noiseFG^{2,t-\frac12}
,
\mathbf{y}_{t-\frac12} - \mathbf{y}_{t-1}
\rangle
}
\nonumber\\&\quad\,
-
\frac{2}{t+1}\Exs\langle
\noiseFG^{1,t-\frac12}
,
\mathbf{x}_{t-1} - \tilde{\mathbf{x}}
\rangle
-
\frac{2}{t+1}\Exs\langle
\noiseFG^{2,t-\frac12}
,
\mathbf{y}_{t-1} - \tilde{\mathbf{y}}
\rangle
\nonumber\\&\quad\,
-
\frac{2}{t+1}\Exs\langle
\noiseBLin^{1,t}
,
\mathbf{x}_{t-\frac12} - \tilde{\mathbf{x}}
\rangle
-
\frac{2}{t+1}\Exs\langle
\noiseBLin^{2,t}
,
\mathbf{y}_{t-\frac12} - \tilde{\mathbf{y}}
\rangle
\nonumber\\&\equiv
\mbox{I}_1
+
\mbox{I}_2
+
\mbox{II}_1
+
\mbox{II}_2
+
\mbox{III}_1
+
\mbox{III}_2
,
\label{eq_recursionmain}
\end{align}
where for each line
$$\begin{aligned}
\mbox{I}_1 + \mbox{I}_2
&\le
\frac{1}{(t+1)\stepsize_t}\left(
\Exs[\distancemetric(\mathbf{x}_{t-1}, \mathbf{y}_{t-1}; \tilde{\mathbf{x}}, \tilde{\mathbf{y}})] 
- 
\Exs[\distancemetric(\mathbf{x}_t, \mathbf{y}_t; \tilde{\mathbf{x}}, \tilde{\mathbf{y}})]
\right)
,
\end{aligned}$$
due to the stepsize condition \eqref{stepsize_cond} which in turn gives the factor in bracket $
\quantile	-	\frac{2\smoothFG}{t+1} \stepsize_t	-	(1+\beta)\smoothBLin^2\stepsize_t^2
$ is nonnegative, and
$$\begin{aligned}
\mbox{II}_1 + \mbox{II}_2
&\le
\frac{\stepsize_t}{t+1}\left(
(1+\tfrac{1}{\beta})\Exs\|\noiseBLin^{t-\frac12}\|^2
+
\Exs\|\noiseBLin^t\|^2
\right)
\blue{+
\frac{\stepsize_t}{(1-\quantile)(t+1)}\Exs\|\noiseFG^{t-\frac12}\|^2
}
,
\end{aligned}$$
due to the basic quadratic inequalities that $
-
\tfrac{1-\quantile}{\stepsize_t}\|\mathbf{x}_{t-1}	-	\mathbf{x}_{t-\frac12}\|^2
-
2\langle
\noiseFG^{1,t-\frac12},	\mathbf{x}_{t-\frac12} - \mathbf{x}_{t-1}
\rangle
\le
\frac{\stepsize_t}{1-\quantile}\|\noiseFG^{1,t-\frac12}\|^2
$ and $
-
\frac{1-\quantile}{\stepsize_t}\|\mathbf{y}_{t-1}	-	\mathbf{y}_{t-\frac12}\|^2
-
2\langle
\noiseFG^{2,t-\frac12},	\mathbf{y}_{t-\frac12} - \mathbf{y}_{t-1}
\rangle
\le
\tfrac{\stepsize_t}{1-\quantile}\|\noiseFG^{2,t-\frac12}\|^2
$, and finally
$$\begin{aligned}
\mbox{III}_1
&=
-
\frac{2}{t+1}\Exs\langle
\noiseFG^{1,t-\frac12},	\mathbf{x}_{t-1} - \tilde{\mathbf{x}}
\rangle
-
\frac{2}{t+1}\Exs\langle
\noiseBLin^{1,t},		\mathbf{x}_{t-\frac12} - \tilde{\mathbf{x}}
\rangle
=
0
,
\end{aligned}$$
and analogously $\mbox{III}_2	=	0$, since each term in above is zero due to the law of iterated expectation applied to martingale difference conditions $
\Exs[\noiseFG^{i,t-\frac12}\mid \mathcal{F}_{t-1}]	=	\mathbf{0}
$ and $
\Exs[\noiseBLin^{i,t}\mid \mathcal{F}_{t-\frac12}]	=	\mathbf{0}
$, $i=1,2$.

Multiplying both sides of \eqref{eq_recursionmain} by $t(t+1)$ combined with the last three estimation bounds, and observing \eqref{noisevarone} and \eqref{noisevartwo}, we obtain for all $t=1,\dots,\Tholder$
$$\begin{aligned}
\lefteqn{
t(t+1)\Exs[\Vquantity(\mathbf{x}\avg{t-\frac12}, \mathbf{y}\avg{t-\frac12} \mid \tilde{\mathbf{x}}, \tilde{\mathbf{y}})] 
- 
(t-1)t\Exs[\Vquantity(\mathbf{x}\avg{t-\frac32}, \mathbf{y}\avg{t-\frac32} \mid \tilde{\mathbf{x}}, \tilde{\mathbf{y}})]
\le
t(t+1)\left(
\mbox{I}_1+\mbox{I}_2+\mbox{II}_1+\mbox{II}_1+\mbox{III}_1+\mbox{III}_2
\right)
}
\\&\le
\frac{t}{\stepsize_t}\left(
\Exs[\distancemetric(\mathbf{x}_{t-1}, \mathbf{y}_{t-1}; \tilde{\mathbf{x}}, \tilde{\mathbf{y}})] 
- 
\Exs[\distancemetric(\mathbf{x}_t, \mathbf{y}_t; \tilde{\mathbf{x}}, \tilde{\mathbf{y}})\right)]
+
t\stepsize_t\left(
\tfrac{1}{1-\quantile}\Exs\|\noiseFG^{t-\frac12}\|^2
+
(1+\tfrac{1}{\beta})\Exs\|\noiseBLin^{t-\frac12}\|^2
+
\Exs\|\noiseBLin^t\|^2
\right)
\\&\le
\frac{t}{\stepsize_t}\left(
\Exs[\distancemetric(\mathbf{x}_{t-1}, \mathbf{y}_{t-1}; \tilde{\mathbf{x}}, \tilde{\mathbf{y}})]
- 
\Exs[\distancemetric(\mathbf{x}_t, \mathbf{y}_t; \tilde{\mathbf{x}}, \tilde{\mathbf{y}})\right)]
+
\left(
\tfrac{1}{1-\quantile}\stdFG^2
+
(2+\tfrac{1}{\beta})\stdBLin^2
\right)t\stepsize_t
,
\end{aligned}$$
where in the last line above we applied \eqref{noisevarone} and \eqref{noisevartwo} in Assumption \ref{assu_noise}, so by law of iterated expectations
\begin{align}\label{noisevar_cont}
\begin{aligned}
&
\Exs\|\noiseFG^{t-\frac12}\|^2
=
\Exs\left[
\|\nabla f(\mathbf{x}\Xtrap{t-1}; \xi_{t-\frac12})	-	\nabla F(\mathbf{x}\Xtrap{t-1})\|^2
+
\|\nabla g(\mathbf{y}\Xtrap{t-1}; \xi_{t-\frac12})	-	\nabla G(\mathbf{y}\Xtrap{t-1})\|^2
\right]
\le
\stdFG^2
,
\\
&\Exs\|\noiseBLin^{t-\frac12}\|^2
=
\Exs\left[
\|\nabla_{\mathbf{x}} h(\mathbf{x}_{t-1}, \mathbf{y}_{t-1}; \zeta_{t-\frac12})	-	\nabla_{\mathbf{x}} H(\mathbf{x}_{t-1}, \mathbf{y}_{t-1})\|^2
\right.
\\&\hspace{1.5in}
+
\left.
\|-\nabla_{\mathbf{y}} h(\mathbf{x}_{t-1}, \mathbf{y}_{t-1}; \zeta_{t-\frac12})	+	\nabla_{\mathbf{y}} H(\mathbf{x}_{t-1}, \mathbf{y}_{t-1})\|^2
\right]
\le
\stdBLin^2
,
\\
&\Exs\|\noiseBLin^t\|^2
=
\Exs\left[
\|\nabla_{\mathbf{x}} h(\mathbf{x}_{t-\frac12}, \mathbf{y}_{t-\frac12}; \zeta_t)		-	\nabla_{\mathbf{x}} H(\mathbf{x}_{t-\frac12}, \mathbf{y}_{t-\frac12})\|^2
\right.
\\&\hspace{1.5in}
+
\left.
\|-\nabla_{\mathbf{y}} h(\mathbf{x}_{t-\frac12}, \mathbf{y}_{t-\frac12}; \zeta_t)	+	\nabla_{\mathbf{y}} H(\mathbf{x}_{t-\frac12}, \mathbf{y}_{t-\frac12})\|^2
\right]
\le
\stdBLin^2
.
\end{aligned}\end{align}
Now for a given $1\le \calT \le \Tholder$, we finish the proof by telescope the above recursion for $t=1,\dots,\calT$.
We conclude from our choice of stepsize as in \eqref{eq_stepsize} that satisfies \eqref{stepsize_cond} so by denoting $
\sigma
\equiv
\frac{1}{\sqrt{3}}\sqrt{
\tfrac{1}{1-\quantile}\stdFG^2
+
(2+\tfrac{1}{\beta})\stdBLin^2
}
$, we have by Lemma \ref{lemm_ssproperties}(i)
$$\begin{aligned}
\lefteqn{
\left(
\tfrac{1}{1-\quantile}\stdFG^2
+
(2+\tfrac{1}{\beta})\stdBLin^2
\right)\sum_{t=1}^{\calT} t\stepsize_t
=
3\sigma^2\sum_{t=1}^{\calT} t\stepsize_t
\le
3\sigma^2\cdot\frac{1}{\square}\sum_{t=1}^{\calT} t^2
}
\\&=
3\sigma^2\cdot\frac{\Cholder\Exs[\distancemetric^{\frac12}(\mathbf{x}_0, \mathbf{y}_0; \tilde{\mathbf{x}}, \tilde{\mathbf{y}})]}{\sigma[\Tholder(\Tholder+1)^2]^{1/2}}
\cdot
\frac{\calT(\calT+\frac12)(\calT+1)}{3}
=
\frac{\calT(\calT+\frac12)(\calT+1)}{[\Tholder(\Tholder+1)^2]^{1/2}}\cdot\sigma \Cholder\Exs[\distancemetric^{\frac12}(\mathbf{x}_0, \mathbf{y}_0; \tilde{\mathbf{x}}, \tilde{\mathbf{y}})]
,
\end{aligned}$$
where we recall in Lemma \ref{lemm_ssproperties} that $
\square
\equiv
\frac{\tilde{\sigma} [\Tholder(\Tholder+1)^2]^{1/2}}{\Cholder\sqrt{\Exs[
\|\mathbf{x}_0 - \oholder_{\mathbf{x}}^\star\|^2
+
\|\mathbf{y}_0 - \oholder_{\mathbf{y}}^\star\|^2
]}}
$.
Finally
$$\begin{aligned}
\lefteqn{
\calT(\calT+1) \Exs[\Vquantity(\mathbf{x}\avg{\calT-\frac12}, \mathbf{y}\avg{\calT-\frac12} \mid \tilde{\mathbf{x}}, \tilde{\mathbf{y}})]
}
\\&\le
\sum_{t=1}^{\calT} \frac{t}{\stepsize_t}\left(
\Exs[\distancemetric(\mathbf{x}_{t-1}, \mathbf{y}_{t-1}; \tilde{\mathbf{x}}, \tilde{\mathbf{y}})] -  \Exs[\distancemetric(\mathbf{x}_t, \mathbf{y}_t; \tilde{\mathbf{x}}, \tilde{\mathbf{y}})]
\right)
+
\left(
\tfrac{1}{1-\quantile}\stdFG^2
+
(2+\tfrac{1}{\beta})\stdBLin^2
\right)\sum_{t=1}^{\calT} t\stepsize_t
\\&=
\frac{1}{\stepsize_1} \Exs[\distancemetric(\mathbf{x}_0, \mathbf{y}_0; \tilde{\mathbf{x}}, \tilde{\mathbf{y}})]
+
\sum_{t=2}^{\calT} \underbrace{
\left(\frac{t}{\stepsize_t} - \frac{t-1}{\stepsize_{t-1}}\right)
}_{
= \sqrt{\tfrac{1+\beta}{\quantile}}\smoothBLin
}  \Exs[\distancemetric(\mathbf{x}_{t-1}, \mathbf{y}_{t-1}; \tilde{\mathbf{x}}, \tilde{\mathbf{y}})]
-
\frac{\calT}{\stepsize_{\calT}}  \Exs[\distancemetric(\mathbf{x}_{\calT}, \mathbf{y}_{\calT}; \tilde{\mathbf{x}}, \tilde{\mathbf{y}})]
\\&\quad\,
+
\frac{\calT(\calT+\frac12)(\calT+1)}{[\Tholder(\Tholder+1)^2]^{1/2}}\cdot\Cholder\sigma \Exs[\distancemetric^{\frac12}(\mathbf{x}_0, \mathbf{y}_0; \tilde{\mathbf{x}}, \tilde{\mathbf{y}})]
.
\end{aligned}$$
Note in above derivations we applied Lemma \ref{lemm_ssproperties}(ii).
Rearranging the terms along with Jensen’s inequality proves \eqref{recursion_bdd_key}.

\paragraph{Step 3.}
We conduct the following ``bootstrapping'' argument to arrive at our final theorem.
Starting from the recursion \eqref{recursion_bdd_key} we have by setting $
\tilde{\mathbf{x}} = \oholder_{\mathbf{x}}^\star
$, $
\tilde{\mathbf{y}} = \oholder_{\mathbf{y}}^\star
$, Lemma \ref{lemm_QuanBdd} implies that its first summand on the left hand $
\calT(\calT + 1) \Exs[\Vquantity(\mathbf{x}\avg{\calT-\frac12}, \mathbf{y}\avg{\calT-\frac12} \mid \oholder_{\mathbf{x}}^\star, \oholder_{\mathbf{y}}^\star)]
$ is nonnegative, and hence we can drop it and have for any $\calT = 1,\dots,\Tholder$
\allowdisplaybreaks
\begin{align}
&
\frac{\calT}{\stepsize_{\calT}} \Exs[ 
\distancemetric(\mathbf{x}_{\calT}, \mathbf{y}_{\calT}; \oholder_{\mathbf{x}}^\star, \oholder_{\mathbf{y}}^\star)
]
\le
\frac{1}{\stepsize_1} \Exs[ 
\distancemetric(\mathbf{x}_0, \mathbf{y}_0; \oholder_{\mathbf{x}}^\star, \oholder_{\mathbf{y}}^\star)
]
\nonumber\\&\quad\,
+
\sqrt{\tfrac{1+\beta}{\quantile}}\smoothBLin\sum_{t=2}^{\calT} \Exs[\distancemetric(\mathbf{x}_{t-1}, \mathbf{y}_{t-1}; \oholder_{\mathbf{x}}^\star, \oholder_{\mathbf{y}}^\star)]
+
\frac{\calT(\calT+\frac12)(\calT+1)}{[\Tholder(\Tholder+1)^2]^{1/2}}\cdot\Cholder\sigma \sqrt{\Exs[\distancemetric(\mathbf{x}_0, \mathbf{y}_0; \oholder_{\mathbf{x}}^\star, \oholder_{\mathbf{y}}^\star)]}
\nonumber\\&=
(\tfrac{2}{\quantile}\smoothFG+\square)\Exs[ 
\distancemetric(\mathbf{x}_0, \mathbf{y}_0; \oholder_{\mathbf{x}}^\star, \oholder_{\mathbf{y}}^\star)
]
\nonumber\\&\quad\,
+
\sqrt{\tfrac{1+\beta}{\quantile}}\smoothBLin\underbrace{
\sum_{t=1}^{\calT}  \Exs[\distancemetric(\mathbf{x}_{t-1}, \mathbf{y}_{t-1}; \oholder_{\mathbf{x}}^\star, \oholder_{\mathbf{y}}^\star)]
}_{
\equiv \mathcal{Q}_{\calT-1}
}
+
\frac{\calT(\calT+\frac12)(\calT+1)}{[\Tholder(\Tholder+1)^2]^{1/2}}\cdot\Cholder\sigma \sqrt{\Exs[\distancemetric(\mathbf{x}_0, \mathbf{y}_0; \oholder_{\mathbf{x}}^\star, \oholder_{\mathbf{y}}^\star)]}
.
\label{eq_recursion}
\end{align}
Converting \eqref{eq_recursion} to a version of partial sum $
\mathcal{Q}_{\calT-1}\equiv \sum_{t=1}^{\calT} \Exs[\distancemetric(\mathbf{x}_{t-1}, \mathbf{y}_{t-1}; \oholder_{\mathbf{x}}^\star, \oholder_{\mathbf{y}}^\star)]
$ that for all $\calT=1,\dots,\Tholder$
\begin{align}\label{eq_recursion_pri}
\begin{aligned}
\lefteqn{
\frac{\calT}{\stepsize_{\calT}} \Exs[ 
\distancemetric(\mathbf{x}_{\calT}, \mathbf{y}_{\calT}; \oholder_{\mathbf{x}}^\star, \oholder_{\mathbf{y}}^\star)
]
=
\frac{\calT}{\stepsize_{\calT}} (\mathcal{Q}_{\calT} - \mathcal{Q}_{\calT-1})
}
\\&\le
\sqrt{\tfrac{1+\beta}{\quantile}}\smoothBLin \mathcal{Q}_{\calT-1}
+
\frac{\calT(\calT+\frac12)(\calT+1)}{[\Tholder(\Tholder+1)^2]^{1/2}}\cdot\Cholder\sigma\sqrt{\mathcal{Q}_0}
+
(\tfrac{2}{\quantile}\smoothFG+\square) \mathcal{Q}_0
.
\end{aligned}\end{align}
\eqref{eq_recursion_pri} is equivalently written as
$$
\frac{\calT}{\stepsize_{\calT}} \mathcal{Q}_{\calT} 
\le
\frac{\calT+1}{\stepsize_{\calT+1}} \mathcal{Q}_{\calT-1}
+
\frac{\calT(\calT+\frac12)(\calT+1)}{[\Tholder(\Tholder+1)^2]^{1/2}}\cdot\Cholder\sigma\sqrt{\mathcal{Q}_0}
+
(\tfrac{2}{\quantile}\smoothFG+\square) \mathcal{Q}_0
.
$$
From here and onwards, we denote $
\kappa_t \equiv \frac{t}{\stepsize_t}
=
\tfrac{2}{\quantile}\smoothFG + \square + \sqrt{\tfrac{1+\beta}{\quantile}}\smoothBLin t
$ for each $t=1,\dots,\Tholder$.
Dividing both sides of the above display by $
\kappa_{\calT}\kappa_{\calT+1}	=	\frac{\calT}{\stepsize_{\calT}}\cdot\frac{\calT+1}{\stepsize_{\calT+1}}
$ gives
$$
\frac{\mathcal{Q}_{\calT}}{\kappa_{\calT+1}}
\le
\frac{\mathcal{Q}_{\calT-1}}{\kappa_{\calT}}
+
\frac{
\frac{\calT(\calT+\frac12)(\calT+1)}{[\Tholder(\Tholder+1)^2]^{1/2}}\cdot\Cholder\sigma\sqrt{\mathcal{Q}_0}
+
(\tfrac{2}{\quantile}\smoothFG+\square) \mathcal{Q}_0
}{\kappa_{\calT}\cdot\kappa_{\calT+1}}
.
$$
Telescoping up from $1,\dots,\calT-1$ for $1\le \calT\le \Tholder$ yields
$$\begin{aligned}
\lefteqn{
\frac{\mathcal{Q}_{\calT-1}}{\kappa_{\calT}}
\le
\frac{\mathcal{Q}_0}{\kappa_1}
+
\sum_{T=1}^{\calT-1}
\frac{
\frac{T(T+\frac12)(T+1)}{[\Tholder(\Tholder+1)^2]^{1/2}}\cdot\Cholder\sigma\sqrt{\mathcal{Q}_0}
+
(\tfrac{2}{\quantile}\smoothFG+\square) \mathcal{Q}_0
}{\kappa_T\cdot\kappa_{T+1}}
}
\\&\le
\frac{\mathcal{Q}_0}{\kappa_1}
+
\left[
\frac{\Tholder(\Tholder+\frac12)(\Tholder+1)}{[\Tholder(\Tholder+1)^2]^{1/2}}\cdot\Cholder\sigma\sqrt{\mathcal{Q}_0}
+
(\tfrac{2}{\quantile}\smoothFG+\square) \mathcal{Q}_0
\right]
\sum_{T=1}^{\calT-1}
\frac{1}{\kappa_T\cdot\kappa_{T+1}}
,
\end{aligned}$$
where we applied Lemma \ref{lemm_ssproperties}(ii) that for all $T=1,\dots,\calT-1$ we have $
\kappa_{T+1} - \kappa_T
= 
\sqrt{\tfrac{1+\beta}{\quantile}}\smoothBLin
$.
This yields
$$\begin{aligned}
\sqrt{\tfrac{1+\beta}{\quantile}}\smoothBLin\sum_{T=1}^{\calT-1}
\frac{1}{\kappa_T\cdot\kappa_{T+1}}
=
\sum_{T=1}^{\calT-1}
\left[
\frac{1}{\kappa_T} - \frac{1}{\kappa_{T+1}}
\right]
=
\frac{1}{\kappa_1} - \frac{1}{\kappa_{\calT}}
,
\end{aligned}$$
and hence
\allowdisplaybreaks
\begin{align*}
\lefteqn{
\sqrt{\tfrac{1+\beta}{\quantile}}\smoothBLin\frac{\mathcal{Q}_{\calT-1}}{\kappa_{\calT}}
}
\\&\le
\sqrt{\tfrac{1+\beta}{\quantile}}\smoothBLin\frac{\mathcal{Q}_0}{\kappa_1}
+
\left[
\frac{\Tholder(\Tholder+\frac12)(\Tholder+1)}{[\Tholder(\Tholder+1)^2]^{1/2}}\cdot\Cholder\sigma\sqrt{\mathcal{Q}_0}
+
(\tfrac{2}{\quantile}\smoothFG+\square) \mathcal{Q}_0
\right]
\sqrt{\tfrac{1+\beta}{\quantile}}\smoothBLin\sum_{T=1}^{\calT-1}
\frac{1}{\kappa_T\cdot\kappa_{T+1}}
\\&=
\sqrt{\tfrac{1+\beta}{\quantile}}\smoothBLin\frac{\mathcal{Q}_0}{\kappa_1}
+
\left[
\frac{\Tholder(\Tholder+\frac12)(\Tholder+1)}{[\Tholder(\Tholder+1)^2]^{1/2}}\cdot\Cholder\sigma\sqrt{\mathcal{Q}_0}
+
(\tfrac{2}{\quantile}\smoothFG+\square) \mathcal{Q}_0
\right]
\left(\frac{1}{
\kappa_1
}
-
\frac{1}{
\kappa_{\calT}
}
\right)
\\&\le
\frac{
\frac{\Tholder(\Tholder+\frac12)(\Tholder+1)}{[\Tholder(\Tholder+1)^2]^{1/2}}\cdot\Cholder\sigma\sqrt{\mathcal{Q}_0}
+
(\tfrac{2}{\quantile}\smoothFG+\square+\sqrt{\tfrac{1+\beta}{\quantile}}\smoothBLin)\mathcal{Q}_0
}{
\tfrac{2}{\quantile}\smoothFG + \square + \sqrt{\tfrac{1+\beta}{\quantile}}\smoothBLin
}
-
\frac{
\frac{\Tholder(\Tholder+\frac12)(\Tholder+1)}{[\Tholder(\Tholder+1)^2]^{1/2}}\cdot\Cholder\sigma\sqrt{\mathcal{Q}_0}
+
(\tfrac{2}{\quantile}\smoothFG+\square) \mathcal{Q}_0
}{
\kappa_{\calT}
}
\\&=
\mathcal{Q}_0 + \frac{
\frac{\Tholder(\Tholder+\frac12)(\Tholder+1)}{[\Tholder(\Tholder+1)^2]^{1/2}}\cdot\Cholder\sigma\sqrt{\mathcal{Q}_0}
}{
\kappa_1
}
-
\frac{
\frac{\Tholder(\Tholder+\frac12)(\Tholder+1)}{[\Tholder(\Tholder+1)^2]^{1/2}}\cdot\Cholder\sigma\sqrt{\mathcal{Q}_0}
+
(\tfrac{2}{\quantile}\smoothFG+\square) \mathcal{Q}_0
}{
\kappa_{\calT}
}
.
\end{align*}
Plugging this into \eqref{eq_recursion_pri} we have for all iterates $1\le \calT\le \Tholder$
\beq\label{Snoexpansion}\begin{aligned}
\lefteqn{
\Exs[\distancemetric(\mathbf{x}_{\calT}, \mathbf{y}_{\calT}; \oholder_{\mathbf{x}}^\star, \oholder_{\mathbf{y}}^\star)]
\le
\sqrt{\tfrac{1+\beta}{\quantile}}\smoothBLin\frac{\mathcal{Q}_{\calT-1}}{\kappa_{\calT}}
+
\frac{
\frac{\Tholder(\Tholder+\frac12)(\Tholder+1)}{[\Tholder(\Tholder+1)^2]^{1/2}}\cdot\Cholder\sigma\sqrt{\mathcal{Q}_0}
+
(\tfrac{2}{\quantile}\smoothFG+\square) \mathcal{Q}_0
}{
\kappa_{\calT}
}
}
\\&\le
\mathcal{Q}_0 + \frac{
\frac{\Tholder(\Tholder+\frac12)(\Tholder+1)}{[\Tholder(\Tholder+1)^2]^{1/2}}\cdot\Cholder\sigma\sqrt{\mathcal{Q}_0}
}{
\kappa_1
}
\le
\left(1 + \frac{
\Cholder\sigma[\Tholder(\Tholder+1)^2]^{1/2}
}{
\kappa_1\sqrt{\mathcal{Q}_0}
}\right) \mathcal{Q}_0
=
\mathcal{A}(\sigma;\Tholder,\Cholder,\quantile,\beta) \Exs[ 
\distancemetric(\mathbf{x}_0, \mathbf{y}_0; \oholder_{\mathbf{x}}^\star, \oholder_{\mathbf{y}}^\star)
]
,
\end{aligned}\eeq
where the prefactor $\mathcal{A}(\sigma;\Tholder,\Cholder,\quantile,\beta)$ lies in $[1,1+\Cholder^2]$ and reduces to 1 when the argument is set as 0.%
\footnote{Indeed, we have from the definition \eqref{prefactorA} of the prefactor $
\mathcal{A}(\tilde{\sigma};\Tholder,\Cholder,\quantile,\beta)
=
1 + \frac{\Cholder\tilde{\sigma}[\Tholder(\Tholder+1)^2]^{1/2}}{\kappa_1\sqrt{\Exs[
\|\mathbf{x}_0 - \oholder_{\mathbf{x}}^\star\|^2
+
\|\mathbf{y}_0 - \oholder_{\mathbf{y}}^\star\|^2
]}}
\ge 
1$ and also by Lemma \ref{lemm_ssproperties}(i) we have $\kappa_1\ge \frac{\tilde{\sigma} [\Tholder(\Tholder+1)^2]^{1/2}}{\Cholder\sqrt{\Exs[
\|\mathbf{x}_0 - \oholder_{\mathbf{x}}^\star\|^2
+
\|\mathbf{y}_0 - \oholder_{\mathbf{y}}^\star\|^2
]}}$ and hence it satisfies $
\mathcal{A}(\tilde{\sigma};\Tholder,\Cholder,\quantile,\beta)
\le
1+\Cholder^2
$.}

Now we drop the second summand on the left hand of \eqref{recursion_bdd_key} with $\tilde{\mathbf{x}} = \oholder_{\mathbf{x}}^\star$, $\tilde{\mathbf{y}} =  \oholder_{\mathbf{y}}^\star$, $\calT = \Tholder$.
Combining with \eqref{Snoexpansion} ($\calT = \Tholder$) gives
\allowdisplaybreaks
\begin{align*}
\lefteqn{
\Tholder(\Tholder+1) \Exs[\Vquantity(\mathbf{x}\avg{\Tholder-\frac12}, \mathbf{y}\avg{\Tholder-\frac12} \mid \oholder_{\mathbf{x}}^\star, \oholder_{\mathbf{y}}^\star)]
}
\\&\le
\kappa_1 \Exs[ 
\distancemetric(\mathbf{x}_0, \mathbf{y}_0; \oholder_{\mathbf{x}}^\star, \oholder_{\mathbf{y}}^\star)
]
+
\sqrt{\tfrac{1+\beta}{\quantile}}\smoothBLin\sum_{t=2}^{\Tholder} \Exs[\distancemetric(\mathbf{x}_{t-1}, \mathbf{y}_{t-1}; \oholder_{\mathbf{x}}^\star, \oholder_{\mathbf{y}}^\star)]
\\&\quad\,
+
\frac{\Tholder(\Tholder+\frac12)(\Tholder+1)}{[\Tholder(\Tholder+1)^2]^{1/2}}\cdot\Cholder\sigma \sqrt{\Exs[\distancemetric(\mathbf{x}_0, \mathbf{y}_0; \oholder_{\mathbf{x}}^\star, \oholder_{\mathbf{y}}^\star)]}
\\&\le
\left(
\tfrac{2}{\quantile}\smoothFG 
+ 
\frac{\sigma [\Tholder(\Tholder+1)^2]^{1/2}}{\Cholder\sqrt{\Exs[\distancemetric(\mathbf{x}_0, \mathbf{y}_0; \oholder_{\mathbf{x}}^\star, \oholder_{\mathbf{y}}^\star)]}}
+ 
\sqrt{\tfrac{1+\beta}{\quantile}}\smoothBLin
\right) \Exs[ 
\distancemetric(\mathbf{x}_0, \mathbf{y}_0; \oholder_{\mathbf{x}}^\star, \oholder_{\mathbf{y}}^\star)
]
\\&\quad\,
+
\sqrt{\tfrac{1+\beta}{\quantile}}\smoothBLin(\Tholder-1) \cdot \mathcal{A}(\sigma;\Tholder,\Cholder,\quantile,\beta) \Exs[ 
\distancemetric(\mathbf{x}_0, \mathbf{y}_0; \oholder_{\mathbf{x}}^\star, \oholder_{\mathbf{y}}^\star)
]
+
\Cholder\sigma[\Tholder(\Tholder+1)^2]^{1/2} \sqrt{\Exs[\distancemetric(\mathbf{x}_0, \mathbf{y}_0; \oholder_{\mathbf{x}}^\star, \oholder_{\mathbf{y}}^\star)]}
\\&\le
\left(\tfrac{2}{\quantile}\smoothFG + \mathcal{A}(\sigma;\Tholder,\Cholder,\quantile,\beta)\sqrt{\tfrac{1+\beta}{\quantile}}\smoothBLin \Tholder\right) \Exs[ 
\distancemetric(\mathbf{x}_0, \mathbf{y}_0; \oholder_{\mathbf{x}}^\star, \oholder_{\mathbf{y}}^\star)
]
\\&\quad\,
+
(\tfrac{1}{\Cholder}+\Cholder)\sigma [\Tholder(\Tholder+1)^2]^{1/2} \sqrt{\Exs[\distancemetric(\mathbf{x}_0, \mathbf{y}_0; \oholder_{\mathbf{x}}^\star, \oholder_{\mathbf{y}}^\star)]}
.
\end{align*}
Using \eqref{QuanBddx} and \eqref{QuanBddy} in Lemma \ref{lemm_QuanBdd} again lower bounds the left hand in the last display as
$$
\Tholder(\Tholder+1) \Exs[\Vquantity(\mathbf{x}\avg{\Tholder-\frac12}, \mathbf{y}\avg{\Tholder-\frac12} \mid \oholder_{\mathbf{x}}^\star, \oholder_{\mathbf{y}}^\star)]
\ge
\frac{\strcvx}{2} \Tholder(\Tholder+1) \Exs[ 
\distancemetric(\mathbf{x}\avg{\Tholder-\frac12}, \mathbf{y}\avg{\Tholder-\frac12}; \oholder_{\mathbf{x}}^\star, \oholder_{\mathbf{y}}^\star)
]
\ge
0
.
$$
Dividing both sides by $\frac{\strcvx}{2}\Tholder(\Tholder+1)$ concludes 
$$\begin{aligned}
\lefteqn{
\Exs[
\distancemetric(\mathbf{x}\avg{\Tholder-\frac12}, \mathbf{y}\avg{\Tholder-\frac12}; \oholder_{\mathbf{x}}^\star, \oholder_{\mathbf{y}}^\star)
]
}
\\&\le
\frac{2\left(
\tfrac{2}{\quantile}\smoothFG + \mathcal{A}(\sigma;\Tholder,\Cholder,\quantile,\beta) \sqrt{\tfrac{1+\beta}{\quantile}}\smoothBLin\Tholder
\right)}{\strcvx\Tholder(\Tholder+1)}
\Exs[ 
\distancemetric(\mathbf{x}_0, \mathbf{y}_0; \oholder_{\mathbf{x}}^\star, \oholder_{\mathbf{y}}^\star)
]
+
\frac{2(\tfrac{1}{\Cholder}+\Cholder)\sigma}{\strcvx\Tholder^{1/2}} \sqrt{\Exs[\distancemetric(\mathbf{x}_0, \mathbf{y}_0; \oholder_{\mathbf{x}}^\star, \oholder_{\mathbf{y}}^\star)]}
,
\end{aligned}$$
and hence concludes \eqref{S_AG_EG} and the whole proof of Theorem \ref{theo_S_AG_EG}.

\pb\subsection{Proof of Theorem~\ref{theo_S_AG_EG_Str}}\label{sec_proof,theo_S_AG_EG_Str}
Using a scaling reduction argument analogous to the one in \S\ref{sec_scalingreduction} we only need to prove the case of $\ratio=1$.
We overload function notations $F, H, G$ to the new group accordingly where $
F
\leftarrow
F(\mathbf{x}) - \tfrac{\mu_\star}{2}\|\mathbf{x}-\mathbf{x}_0\|^2 
$ and $
G
\leftarrow
G(\mathbf{y}) - \tfrac{\mu_\star}{2}\|\mathbf{y}-\mathbf{y}_0\|^2 
$ are nonstrongly convex and $
H
\leftarrow
\frac{\mu_\star}{2}\|\mathbf{x}-\mathbf{x}_0\|^2 
+ 
\mathbf{x}^\top \mathbf{B}\mathbf{y}
-
\mathbf{x}^\top \mathbf{u}_{\mathbf{x}} 
+
\mathbf{u}_{\mathbf{y}}^\top \mathbf{y}
- 
\frac{\mu_\star}{2}\|\mathbf{y}-\mathbf{y}_0\|^2
$ is an isotropic quadratic.
\blue{For convenience we repeat the iterates of Algorithm~\ref{algo_AG_EG_Str} with $\ratio=1$ as
\allowdisplaybreaks
\begin{align*}
\mathbf{x}_{t-\frac12}	&=	\mathbf{x}_{t-1}	-	\stepsize_t\left(\nabla f(\mathbf{x}\Xtrap{t-1}; \xi_{t-\frac12}) + \nabla_{\mathbf{x}} h(\mathbf{x}_{t-1}, \mathbf{y}_{t-1}; \zeta_{t-\frac12})\right)
,\\
\mathbf{y}_{t-\frac12}	&=	\mathbf{y}_{t-1}	-	\stepsize_t\left(-\nabla_{\mathbf{y}} h(\mathbf{x}_{t-1}, \mathbf{y}_{t-1}; \zeta_{t-\frac12}) + \nabla g(\mathbf{y}\Xtrap{t-1}; \xi_{t-\frac12})\right)
,\\
\mathbf{x}\avg{t-\frac12}	&=	(1-\alpha_t)\mathbf{x}\avg{t-\frac32} + \alpha_t\mathbf{x}_{t-\frac12}
,\\
\mathbf{y}\avg{t-\frac12}	&=	(1-\alpha_t)\mathbf{y}\avg{t-\frac32} + \alpha_t\mathbf{y}_{t-\frac12}
,\\
\mathbf{x}_t			&=	\mathbf{x}_{t-1}	-	\stepsize_t\left(\nabla f(\mathbf{x}\Xtrap{t-1}; \xi_{t-\frac12}) + \nabla_{\mathbf{x}} h(\mathbf{x}_{t-\frac12}, \mathbf{y}_{t-\frac12}; \zeta_t)\right)
,\\
\mathbf{y}_t			&=	\mathbf{y}_{t-1}	-	\stepsize_t\left(-\nabla_{\mathbf{y}} h(\mathbf{x}_{t-\frac12}, \mathbf{y}_{t-\frac12}; \zeta_t) + \nabla g(\mathbf{y}\Xtrap{t-1}; \xi_{t-\frac12})\right)
,\\
\mathbf{x}\Xtrap{t}		&=	(1-\alpha_{t+1})\mathbf{x}\avg{t-\frac12} + \alpha_{t+1}\mathbf{x}_{t}
,\\
\mathbf{y}\Xtrap{t}		&=	(1-\alpha_{t+1})\mathbf{y}\avg{t-\frac12} + \alpha_{t+1}\mathbf{y}_{t}
,
\end{align*}
with the initialization $
\mathbf{x}_0	=	\mathbf{x}\Xtrap{0}	=	\mathbf{x}\avg{-\frac12}	\in	\reals^n
$, $
\mathbf{y}_0	=	\mathbf{y}\Xtrap{0}	=	\mathbf{y}\avg{-\frac12}	\in	\reals^m
$.
}%
We continue to assume the noise-related setting as in \eqref{noisedef}, and continue to denote $
\distancemetric(\mathbf{x}, \mathbf{y}; \tilde{\mathbf{x}}, \tilde{\mathbf{y}})
\equiv
\left\|\mathbf{x} - \tilde{\mathbf{x}}\right\|^2
+
\left\|\mathbf{y} - \tilde{\mathbf{y}}\right\|^2
$.
Our proof proceeds in the following steps:

\paragraph{Step 1.}
We prove the following generalization of Lemma \ref{lemm_aggregated}:
\begin{lemma}\label{lemm_aggregatedprime}
For arbitrary $\tilde{\mathbf{x}}\in \reals^n, \tilde{\mathbf{y}}\in \reals^m$ and $\alpha_t\in (0,1]$ the iterates of Algorithm~\ref{algo_AG_EG_Str} satisfy almost surely
\begin{equation}\label{aggregatedprime}
\begin{aligned}
\lefteqn{
\Vquantity(\mathbf{x}\avg{t-\frac12}, \mathbf{y}\avg{t-\frac12} \mid \tilde{\mathbf{x}}, \tilde{\mathbf{y}})
-
(1-\alpha_t)\Vquantity(\mathbf{x}\avg{t-\frac32}, \mathbf{y}\avg{t-\frac32} \mid \tilde{\mathbf{x}}, \tilde{\mathbf{y}})
}
\\&\le
\alpha_t\langle\nabla F(\mathbf{x}\Xtrap{t-1}) + \nabla_{\mathbf{x}} H(\mathbf{x}_{t-\frac12},\mathbf{y}_{t-\frac12}),	\mathbf{x}_{t-\frac12} - \tilde{\mathbf{x}} \rangle
+
\alpha_t\langle -\nabla_{\mathbf{y}} H(\mathbf{x}_{t-\frac12},\mathbf{y}_{t-\frac12}) + \nabla G(\mathbf{y}\Xtrap{t-1}),	\mathbf{y}_{t-\frac12} - \tilde{\mathbf{y}} \rangle
\\&\quad\,
+
\tfrac{\alpha_t^2\smoothFG}{2}\distancemetric(\mathbf{x}_{t-\frac12}, \mathbf{y}_{t-\frac12}; \mathbf{x}_{t-1}, \mathbf{y}_{t-1})
- 
\alpha_t\mu_\star\distancemetric(\mathbf{x}_{t-\frac12}, \mathbf{y}_{t-\frac12}; \tilde{\mathbf{x}}, \tilde{\mathbf{y}})
.
\end{aligned}\end{equation}
\end{lemma}

The proof goes in an analogous fashion as the proof of Lemma \ref{lemm_aggregated}, except that the display above \eqref{eq_acc_3} is replaced by
$$\begin{aligned}
\lefteqn{
\langle \nabla_{\mathbf{x}} H(\tilde{\mathbf{x}},\tilde{\mathbf{y}}), \mathbf{x}_{t-\frac12} - \tilde{\mathbf{x}} \rangle
+
\langle -\nabla_{\mathbf{y}} H(\tilde{\mathbf{x}},\tilde{\mathbf{y}}), \mathbf{y}_{t-\frac12} - \tilde{\mathbf{y}} \rangle
}
\\&\le
\langle \nabla_{\mathbf{x}} H(\mathbf{x}_{t-\frac12},\mathbf{y}_{t-\frac12}), \mathbf{x}_{t-\frac12} - \tilde{\mathbf{x}} \rangle
+
\langle -\nabla_{\mathbf{y}} H(\mathbf{x}_{t-\frac12},\mathbf{y}_{t-\frac12}), \mathbf{y}_{t-\frac12} - \tilde{\mathbf{y}} \rangle
- 
\mu_\star\distancemetric(\mathbf{x}_{t-\frac12}, \mathbf{y}_{t-\frac12}; \tilde{\mathbf{x}}, \tilde{\mathbf{y}})
,
\end{aligned}$$
due to our $H$ being a $\mu_\star$-strongly-convex-$\mu_\star$-strongly-concave isotropic quadratic function after scaling reduction.
Hence \eqref{eq_acc_3} becomes
\begin{align}
&
\langle \nabla_{\mathbf{x}} H(\tilde{\mathbf{x}},\tilde{\mathbf{y}}), \mathbf{x}\avg{t-\frac12} - \tilde{\mathbf{x}} \rangle 
- (1-\alpha_t)\langle \nabla_{\mathbf{x}} H(\tilde{\mathbf{x}},\tilde{\mathbf{y}}), \mathbf{x}\avg{t-\frac32} - \tilde{\mathbf{x}} \rangle
\nonumber\\&\hspace{1in}
+
\langle -\nabla_{\mathbf{y}} H(\tilde{\mathbf{x}},\tilde{\mathbf{y}}), \mathbf{y}\avg{t-\frac12} - \tilde{\mathbf{y}} \rangle - (1-\alpha_t)\langle -\nabla_{\mathbf{y}} H(\tilde{\mathbf{x}},\tilde{\mathbf{y}}), \mathbf{y}\avg{t-\frac32} - \tilde{\mathbf{y}} \rangle 
\nonumber\\&\le
\alpha_t \left[ 
\langle \nabla_{\mathbf{x}} H(\mathbf{x}_{t-\frac12},\mathbf{y}_{t-\frac12}), \mathbf{x}_{t-\frac12} - \tilde{\mathbf{x}} \rangle
+
\langle -\nabla_{\mathbf{y}} H(\mathbf{x}_{t-\frac12},\mathbf{y}_{t-\frac12}), \mathbf{y}_{t-\frac12} - \tilde{\mathbf{y}} \rangle
- 
\mu_\star\distancemetric(\mathbf{x}_{t-\frac12}, \mathbf{y}_{t-\frac12}; \tilde{\mathbf{x}}, \tilde{\mathbf{y}})
\right]
.
\label{eq_acc_3_prime}
\end{align}
This concludes \eqref{aggregatedprime} and the whole lemma.

\paragraph{Step 2.}
Analogous to \eqref{keyeq} in Step 2 in the proof of Theorem \ref{theo_S_AG_EG} in \S\ref{sec_proof,theo_S_AG_EG} we conclude for all $\mathbf{x} \in \reals^n$, $\mathbf{y} \in \reals^m$,
$$\begin{aligned}
&\quad\,
\stepsize_t\Exs\langle\nabla f(\mathbf{x}\Xtrap{t-1}; \xi_{t-\frac12}) + \nabla_{\mathbf{x}} h(\mathbf{x}_{t-\frac12}, \mathbf{y}_{t-\frac12}; \zeta_t),	\mathbf{x}_{t-\frac12} - \tilde{\mathbf{x}} \rangle
\\&\hspace{1in}\,
+
\stepsize_t\Exs\langle -\nabla_{\mathbf{y}} h(\mathbf{x}_{t-\frac12}, \mathbf{y}_{t-\frac12}; \zeta_t) + \nabla g(\mathbf{y}\Xtrap{t-1}; \xi_{t-\frac12}),	\mathbf{y}_{t-\frac12} - \tilde{\mathbf{y}} \rangle
\\&\le 
\frac12\left(\Exs[\distancemetric(\mathbf{x}_{t-1}, \mathbf{y}_{t-1}; \tilde{\mathbf{x}}, \tilde{\mathbf{y}})] - \Exs[\distancemetric(\mathbf{x}_t, \mathbf{y}_t; \tilde{\mathbf{x}}, \tilde{\mathbf{y}})]\right)
\\&\quad\,
-
\frac{1 - (1+\beta)\smoothBLin^2\stepsize_t^2}{2}\Exs[\distancemetric(\mathbf{x}_{t-\frac12}, \mathbf{y}_{t-\frac12}; \mathbf{x}_{t-1}, \mathbf{y}_{t-1})]
+
\frac{\stepsize_t^2}{2}(2+\tfrac{1}{\beta})\stdBLin^2
.
\end{aligned}$$
To show this, note that
$$\begin{aligned}
\lefteqn{
\stepsize_t\langle
\nabla f(\mathbf{x}\Xtrap{t-1}; \xi_{t-\frac12})
+
\nabla_{\mathbf{x}} h(\mathbf{x}_{t-\frac12}, \mathbf{y}_{t-\frac12}; \zeta_t)
,
\mathbf{x}_{t-\frac12} - \tilde{\mathbf{x}}
\rangle
}
\\&\le
\frac12\left(
\|\mathbf{x}_{t-1} - \tilde{\mathbf{x}}\|^2
-
\|\mathbf{x}_t - \tilde{\mathbf{x}}\|^2
-
\|\mathbf{x}_{t-\frac12} - \mathbf{x}_{t-1}\|^2
\right)
+
\frac{\stepsize_t^2}{2}\|
\nabla_{\mathbf{x}} h(\mathbf{x}_{t-\frac12}, \mathbf{y}_{t-\frac12}; \zeta_t)
-
\nabla_{\mathbf{x}} h(\mathbf{x}_{t-1}, \mathbf{y}_{t-1}; \zeta_{t-\frac12})
\|^2
,
\end{aligned}$$
and analogously
$$\begin{aligned}
\lefteqn{
\stepsize_t\langle
-\nabla_{\mathbf{y}} h(\mathbf{x}_{t-\frac12}, \mathbf{y}_{t-\frac12}; \zeta_t)
+
\nabla g(\mathbf{y}\Xtrap{t-1}; \xi_{t-\frac12})
,
\mathbf{y}_{t-\frac12} - \tilde{\mathbf{y}}
\rangle
}
\\&\le
\frac12\left(
\|\mathbf{y}_{t-1} - \tilde{\mathbf{y}} \|^2 
-
\|\mathbf{y}_t - \tilde{\mathbf{y}} \|^2
-
\|\mathbf{y}_{t-\frac12} - \mathbf{y}_{t-1} \|^2
\right)
+
\frac{\stepsize_t^2}{2}\|
\nabla_{\mathbf{y}} h(\mathbf{x}_{t-\frac12}, \mathbf{y}_{t-\frac12}; \zeta_t)
-
\nabla_{\mathbf{y}} h(\mathbf{x}_{t-1}, \mathbf{y}_{t-1}; \zeta_{t-\frac12})
\|^2
.
\end{aligned}$$
To handle the stochastic terms, Young’s inequality combined with the martingale structure, along with the definition of $\smoothBLin$, indicates
$$\begin{aligned}
\lefteqn{
\Exs\left\|\begin{bmatrix}
\nabla_{\mathbf{x}} h(\mathbf{x}_{t-\frac12}, \mathbf{y}_{t-\frac12}; \zeta_t)
-
\nabla_{\mathbf{x}} h(\mathbf{x}_{t-1}, \mathbf{y}_{t-1}; \zeta_{t-\frac12})
\\
\nabla_{\mathbf{y}} h(\mathbf{x}_{t-\frac12}, \mathbf{y}_{t-\frac12}; \zeta_t)
-
\nabla_{\mathbf{y}} h(\mathbf{x}_{t-1}, \mathbf{y}_{t-1}; \zeta_{t-\frac12})
\end{bmatrix}\right\|^2
}
\\&=
\Exs\left\|\begin{bmatrix}
\nabla_{\mathbf{x}} H(\mathbf{x}_{t-\frac12}, \mathbf{y}_{t-\frac12})
-
\nabla_{\mathbf{x}} H(\mathbf{x}_{t-1}, \mathbf{y}_{t-1})
-
\noiseBLin^{1,t-\frac12}
\\
\nabla_{\mathbf{y}} H(\mathbf{x}_{t-\frac12}, \mathbf{y}_{t-\frac12})
-
\nabla_{\mathbf{y}} H(\mathbf{x}_{t-1}, \mathbf{y}_{t-1})
-
\noiseBLin^{2,t-\frac12}
\end{bmatrix}\right\|^2
+
\Exs\left\|\begin{bmatrix}
\noiseBLin^{1,t}
\\
\noiseBLin^{2,t}
\end{bmatrix}\right\|^2
\\&\le
(1+\beta)\Exs\left\|\begin{bmatrix}
\mu_\star\mathbf{I}	&	\mathbf{B}
\\
\mathbf{B}^\top		&	\mu_\star\mathbf{I}
\end{bmatrix}\begin{bmatrix}
\mathbf{x}_{t-\frac12}	-	\mathbf{x}_{t-1}
\\
\mathbf{y}_{t-\frac12}	-	\mathbf{y}_{t-1}
\end{bmatrix}\right\|^2
+
(1+\tfrac{1}{\beta})\Exs\left\|\begin{bmatrix}
\noiseBLin^{1,t-\frac12}
\\
\noiseBLin^{2,t-\frac12}
\end{bmatrix}\right\|^2
+
\Exs\left\|\begin{bmatrix}
\noiseBLin^{1,t}
\\
\noiseBLin^{2,t}
\end{bmatrix}\right\|^2
\\&\le
(1+\beta)\smoothBLin^2\left(
\Exs\|\mathbf{x}_{t-\frac12}	-	\mathbf{x}_{t-1}\|^2
+
\Exs\|\mathbf{y}_{t-\frac12}	-	\mathbf{y}_{t-1}\|^2
\right)
+
(1+\tfrac{1}{\beta})\Exs\|\noiseBLin^{t-\frac12}\|^2
+
\Exs\|\noiseBLin^t\|^2
.
\end{aligned}$$
Combining the last three displays gives
\begin{align}
\lefteqn{
\stepsize_t\Exs\langle
\nabla f(\mathbf{x}\Xtrap{t-1}; \xi_{t-\frac12})
+
\nabla_{\mathbf{x}} h(\mathbf{x}_{t-\frac12}, \mathbf{y}_{t-\frac12}; \zeta_t)
,
\mathbf{x}_{t-\frac12} - \tilde{\mathbf{x}}
\rangle
}
\nonumber\\&\hspace{1in}\,
+
\stepsize_t\Exs\langle
-\nabla_{\mathbf{y}} h(\mathbf{x}_{t-\frac12}, \mathbf{y}_{t-\frac12}; \zeta_t)
+
\nabla g(\mathbf{y}\Xtrap{t-1}; \xi_{t-\frac12})
,
\mathbf{y}_{t-\frac12} - \tilde{\mathbf{y}}
\rangle
\nonumber\\&\le
\frac12\left(
\Exs[\distancemetric(\mathbf{x}_{t-1}, \mathbf{y}_{t-1}; \tilde{\mathbf{x}}, \tilde{\mathbf{y}})]
-
\Exs[\distancemetric(\mathbf{x}_t, \mathbf{y}_t; \tilde{\mathbf{x}}, \tilde{\mathbf{y}})]
\right)
-
\frac{1 - (1+\beta)\smoothBLin^2 \stepsize_t^2}{2}
\Exs[\distancemetric(\mathbf{x}_{t-\frac12}, \mathbf{y}_{t-\frac12}; \mathbf{x}_{t-1}, \mathbf{y}_{t-1})]
\nonumber\\&\quad\,
+
\frac{\stepsize_t^2}{2}\left(
(1+\tfrac{1}{\beta})\Exs\|\noiseBLin^{t-\frac12}\|^2
+
\Exs\|\noiseBLin^t\|^2
\right)
.
\tag{\ref{keyeq}}
\end{align}
Combining this with Lemma~\ref{lemm_aggregatedprime}, we have
$$\begin{aligned}
\lefteqn{
\Exs[\Vquantity(\mathbf{x}\avg{t-\frac12}, \mathbf{y}\avg{t-\frac12} \mid \tilde{\mathbf{x}}, \tilde{\mathbf{y}})]	-	(1-\alpha_t)\Exs[\Vquantity(\mathbf{x}\avg{t-\frac32}, \mathbf{y}\avg{t-\frac32} \mid \tilde{\mathbf{x}}, \tilde{\mathbf{y}})]
}
\\&\le
\blue{
\alpha_t\Exs\langle\nabla F(\mathbf{x}\Xtrap{t-1}) + \nabla_{\mathbf{x}} H(\mathbf{x}_{t-\frac12},\mathbf{y}_{t-\frac12}),	\mathbf{x}_{t-\frac12} - \tilde{\mathbf{x}} \rangle
+
\alpha_t\Exs\langle - \nabla_{\mathbf{y}} H(\mathbf{x}_{t-\frac12},\mathbf{y}_{t-\frac12}) + \nabla G(\mathbf{y}\Xtrap{t-1}),	\mathbf{y}_{t-\frac12} - \tilde{\mathbf{y}} \rangle
}
\\&\quad\,
\blue{
+
\tfrac{\alpha_t^2\smoothFG}{2}\Exs[\distancemetric(\mathbf{x}_{t-\frac12}, \mathbf{y}_{t-\frac12}; \mathbf{x}_{t-1}, \mathbf{y}_{t-1})]
-
\alpha_t\mu_\star\Exs[\distancemetric(\mathbf{x}_{t-\frac12}, \mathbf{y}_{t-\frac12}; \tilde{\mathbf{x}}, \tilde{\mathbf{y}})]
}
\\&=
\blue{
\alpha_t\Exs\langle\nabla f(\mathbf{x}\Xtrap{t-1}; \xi_{t-\frac12}) + \nabla_{\mathbf{x}} h(\mathbf{x}_{t-\frac12}, \mathbf{y}_{t-\frac12}; \zeta_t),	\mathbf{x}_{t-\frac12} - \tilde{\mathbf{x}} \rangle
}
\\&\quad\,
\blue{
+
\alpha_t\Exs\langle -\nabla_{\mathbf{y}} h(\mathbf{x}_{t-\frac12}, \mathbf{y}_{t-\frac12}; \zeta_t) + \nabla g(\mathbf{y}\Xtrap{t-1}; \xi_{t-\frac12}),	\mathbf{y}_{t-\frac12} - \tilde{\mathbf{y}} \rangle
}
\\&\quad\,
\blue{
-
\alpha_t\Exs\langle \noiseFG^{1,t-\frac12} + \noiseBLin^{1,t},	\mathbf{x}_{t-\frac12} - \tilde{\mathbf{x}}\rangle
-
\alpha_t\Exs\langle \noiseFG^{2,t-\frac12} + \noiseBLin^{2,t},	\mathbf{y}_{t-\frac12} - \tilde{\mathbf{y}} \rangle
}
\\&\quad\,
\blue{
+
\tfrac{\alpha_t^2\smoothFG}{2}\Exs[\distancemetric(\mathbf{x}_{t-\frac12}, \mathbf{y}_{t-\frac12}; \mathbf{x}_{t-1}, \mathbf{y}_{t-1})]
-
\alpha_t\mu_\star\Exs[\distancemetric(\mathbf{x}_{t-\frac12}, \mathbf{y}_{t-\frac12}; \tilde{\mathbf{x}}, \tilde{\mathbf{y}})]
}
\\&\le
\frac{\alpha_t}{\stepsize_t}\left(
\frac12\left(\Exs[\distancemetric(\mathbf{x}_{t-1}, \mathbf{y}_{t-1}; \tilde{\mathbf{x}}, \tilde{\mathbf{y}})] - \Exs[\distancemetric(\mathbf{x}_t, \mathbf{y}_t; \tilde{\mathbf{x}}, \tilde{\mathbf{y}})]\right)
\right.
\\&\hspace{1in}\,
\left.
-
\frac{1 - (1+\beta)\smoothBLin^2\stepsize_t^2}{2}\Exs[\distancemetric(\mathbf{x}_{t-\frac12}, \mathbf{y}_{t-\frac12}; \mathbf{x}_{t-1}, \mathbf{y}_{t-1})]
+
\frac{\stepsize_t^2}{2}(2+\tfrac{1}{\beta})\stdBLin^2
\right)
\\&\quad\,
\blue{
-
\alpha_t\Exs\langle \noiseFG^{1,t-\frac12} + \noiseBLin^{1,t},	\mathbf{x}_{t-\frac12} - \tilde{\mathbf{x}} \rangle
-
\alpha_t\Exs\langle \noiseFG^{2,t-\frac12} + \noiseBLin^{2,t},	\mathbf{y}_{t-\frac12} - \tilde{\mathbf{y}} \rangle
}
\\&\quad\,
+
\tfrac{\alpha_t^2\smoothFG}{2}\Exs[\distancemetric(\mathbf{x}_{t-\frac12}, \mathbf{y}_{t-\frac12}; \mathbf{x}_{t-1}, \mathbf{y}_{t-1})]
-
\alpha_t\mu_\star\Exs[\distancemetric(\mathbf{x}_{t-\frac12}, \mathbf{y}_{t-\frac12}; \tilde{\mathbf{x}}, \tilde{\mathbf{y}})]
.
\end{aligned}$$
Continuing this estimation gives
\allowdisplaybreaks
\begin{align*}
\lefteqn{
\Exs[\Vquantity(\mathbf{x}\avg{t-\frac12}, \mathbf{y}\avg{t-\frac12} \mid \tilde{\mathbf{x}}, \tilde{\mathbf{y}})]	-	(1-\alpha_t)\Exs[\Vquantity(\mathbf{x}\avg{t-\frac32}, \mathbf{y}\avg{t-\frac32} \mid \tilde{\mathbf{x}}, \tilde{\mathbf{y}})]
}
\\&\le
\frac{\alpha_t}{2\stepsize_t}\left(\Exs[\distancemetric(\mathbf{x}_{t-1}, \mathbf{y}_{t-1}; \tilde{\mathbf{x}}, \tilde{\mathbf{y}})] - \Exs[\distancemetric(\mathbf{x}_t, \mathbf{y}_t; \tilde{\mathbf{x}}, \tilde{\mathbf{y}})]\right)
\\&\quad\,
-
\frac{\alpha_t}{2\stepsize_t}\left(\quantile - \alpha_t\smoothFG\stepsize_t - (1+\beta)\smoothBLin^2\stepsize_t^2\right)\Exs[\distancemetric(\mathbf{x}_{t-\frac12}, \mathbf{y}_{t-\frac12}; \mathbf{x}_{t-1}, \mathbf{y}_{t-1})]
\\&\quad\,
+
\frac{\alpha_t\stepsize_t}{2}(2+\tfrac{1}{\beta})\stdBLin^2
-
\frac{\alpha_t(1 - \quantile)}{2\stepsize_t}\Exs[\distancemetric(\mathbf{x}_{t-\frac12}, \mathbf{y}_{t-\frac12}; \mathbf{x}_{t-1}, \mathbf{y}_{t-1})]
\\&\quad\,
\blue{
-
\alpha_t\Exs\langle \noiseFG^{1,t-\frac12}, \mathbf{x}_{t-\frac12} - \mathbf{x}_{t-1} \rangle
-
\alpha_t\Exs\langle \noiseFG^{2,t-\frac12}, \mathbf{y}_{t-\frac12} - \mathbf{y}_{t-1} \rangle
}
-
\alpha_t\mu_\star\Exs[\distancemetric(\mathbf{x}_{t-\frac12}, \mathbf{y}_{t-\frac12}; \tilde{\mathbf{x}}, \tilde{\mathbf{y}})]
\\&\le
\frac{\alpha_t}{2\stepsize_t}\left(\Exs[\distancemetric(\mathbf{x}_{t-1}, \mathbf{y}_{t-1}; \tilde{\mathbf{x}}, \tilde{\mathbf{y}})] - \Exs[\distancemetric(\mathbf{x}_t, \mathbf{y}_t; \tilde{\mathbf{x}}, \tilde{\mathbf{y}})]\right)
\\&\quad\,
-
\frac{\alpha_t}{2\stepsize_t}\left( \quantile - \alpha_t\smoothFG\stepsize_t - (1+\beta)\smoothBLin^2\stepsize_t^2\right)\Exs[\distancemetric(\mathbf{x}_{t-\frac12}, \mathbf{y}_{t-\frac12}; \mathbf{x}_{t-1}, \mathbf{y}_{t-1})]
\\&\quad\,
+
\frac{\alpha_t\stepsize_t}{2}(2+\tfrac{1}{\beta})\stdBLin^2
\blue{
+
\frac{\alpha_t\stepsize_t}{2(1 - \quantile)}\Exs[
\|\noiseFG^{1,t-\frac12}\|^2	+	\|\noiseFG^{2,t-\frac12}\|^2
]
}
-
\alpha_t\mu_\star\Exs[\distancemetric(\mathbf{x}_{t-\frac12}, \mathbf{y}_{t-\frac12}; \tilde{\mathbf{x}}, \tilde{\mathbf{y}})]
\\&\le
\frac{\alpha_t}{2\stepsize_t}\left(\Exs[\distancemetric(\mathbf{x}_{t-1}, \mathbf{y}_{t-1}; \tilde{\mathbf{x}}, \tilde{\mathbf{y}})]		-	\Exs[\distancemetric(\mathbf{x}_t, \mathbf{y}_t; \tilde{\mathbf{x}}, \tilde{\mathbf{y}})] \right)
\\&\quad\,
-
\frac{\alpha_t}{2\stepsize_t} \left( \quantile - \alpha_t\smoothFG\stepsize_t - (1+\beta)\smoothBLin^2\stepsize_t^2 \right)\Exs[\distancemetric(\mathbf{x}_{t-\frac12}, \mathbf{y}_{t-\frac12}; \mathbf{x}_{t-1}, \mathbf{y}_{t-1})]
\\&\quad\,
-
\alpha_t\mu_\star\Exs[\distancemetric(\mathbf{x}_{t-\frac12}, \mathbf{y}_{t-\frac12}; \tilde{\mathbf{x}}, \tilde{\mathbf{y}})]
+
\frac{\alpha_t\stepsize_t}{2} \left(
\tfrac{1}{1-\quantile}\stdFG^2	+	(2+\tfrac{1}{\beta})\stdBLin^2
\right)
.
\end{align*}
This yields, applying Young's inequality,
\allowdisplaybreaks
\begin{align*}
\lefteqn{
\Exs[\Vquantity(\mathbf{x}\avg{t-\frac12}, \mathbf{y}\avg{t-\frac12} \mid \tilde{\mathbf{x}}, \tilde{\mathbf{y}})]	-	(1-\alpha_t)\Exs[\Vquantity(\mathbf{x}\avg{t-\frac32}, \mathbf{y}\avg{t-\frac32} \mid \tilde{\mathbf{x}}, \tilde{\mathbf{y}})]
}
\\& \le
\frac{\alpha_t}{2\stepsize_t}\left(\Exs[\distancemetric(\mathbf{x}_{t-1}, \mathbf{y}_{t-1}; \tilde{\mathbf{x}}, \tilde{\mathbf{y}})]		-	\Exs[\distancemetric(\mathbf{x}_t, \mathbf{y}_t; \tilde{\mathbf{x}}, \tilde{\mathbf{y}})] \right)
\\&\quad\,
-
\frac{\alpha_t}{2\stepsize_t} \left( \quantile - \alpha_t\smoothFG\stepsize_t - (1+\beta)\smoothBLin^2\stepsize_t^2 \right)\Exs[\distancemetric(\mathbf{x}_{t-\frac12}, \mathbf{y}_{t-\frac12}; \mathbf{x}_{t-1}, \mathbf{y}_{t-1})]
\\&\quad\,
-
\alpha_t\mu_\star\Exs[\distancemetric(\mathbf{x}_{t-\frac12}, \mathbf{y}_{t-\frac12}; \tilde{\mathbf{x}}, \tilde{\mathbf{y}})]
+
\frac{\alpha_t\stepsize_t}{2} \left(
\tfrac{1}{1-\quantile}\stdFG^2	+	(2+\tfrac{1}{\beta})\stdBLin^2
\right)
\\& \le
\frac{\alpha_t}{2\stepsize_t}\left((1-\alpha_t)\Exs[\distancemetric(\mathbf{x}_{t-1}, \mathbf{y}_{t-1}; \tilde{\mathbf{x}}, \tilde{\mathbf{y}})]		-	\Exs[\distancemetric(\mathbf{x}_t, \mathbf{y}_t; \tilde{\mathbf{x}}, \tilde{\mathbf{y}})] \right)
\\&\quad\,
-
\frac{\alpha_t}{2\stepsize_t} \left( \quantile - \alpha_t\smoothFG\stepsize_t - (1+\beta)\smoothBLin^2\stepsize_t^2 \right)\Exs[\distancemetric(\mathbf{x}_{t-\frac12}, \mathbf{y}_{t-\frac12}; \mathbf{x}_{t-1}, \mathbf{y}_{t-1})]
\\&\quad\,
+
\frac{\alpha_t^2}{2\stepsize_t}\Exs[\distancemetric(\mathbf{x}_{t-1}, \mathbf{y}_{t-1}; \tilde{\mathbf{x}}, \tilde{\mathbf{y}})]
-
\alpha_t\mu_\star\Exs[\distancemetric(\mathbf{x}_{t-\frac12}, \mathbf{y}_{t-\frac12}; \tilde{\mathbf{x}}, \tilde{\mathbf{y}})]
+
\frac{\alpha_t\stepsize_t}{2} \left(
\tfrac{1}{1-\quantile}\stdFG^2	+	(2+\tfrac{1}{\beta})\stdBLin^2
\right)
\\& \le
\frac{\alpha_t}{2\stepsize_t}\left((1-\alpha_t)\Exs[\distancemetric(\mathbf{x}_{t-1}, \mathbf{y}_{t-1}; \tilde{\mathbf{x}}, \tilde{\mathbf{y}})]		-	\Exs[\distancemetric(\mathbf{x}_t, \mathbf{y}_t; \tilde{\mathbf{x}}, \tilde{\mathbf{y}})] \right)
\\&\quad\,
-
\frac{\alpha_t}{2\stepsize_t} \left( \quantile - \alpha_t\smoothFG\stepsize_t - (1+\beta)\smoothBLin^2\stepsize_t^2 \right)\Exs[\distancemetric(\mathbf{x}_{t-\frac12}, \mathbf{y}_{t-\frac12}; \mathbf{x}_{t-1}, \mathbf{y}_{t-1})]
\\&\quad\,
+
\stepsize_t\mu_\star^2\Exs[\distancemetric(\mathbf{x}_{t-\frac12}, \mathbf{y}_{t-\frac12}; \mathbf{x}_{t-1}, \mathbf{y}_{t-1})]
+
\frac{\alpha_t\stepsize_t}{2} \left(
\tfrac{1}{1-\quantile}\stdFG^2	+	(2+\tfrac{1}{\beta})\stdBLin^2
\right)
.
\end{align*}
Setting $\stepsize_t = \frac{\alpha_t}{\mu_\star}$ we have
\allowdisplaybreaks
\begin{align*}
\lefteqn{
\Exs[\Vquantity(\mathbf{x}\avg{t-\frac12}, \mathbf{y}\avg{t-\frac12} \mid \tilde{\mathbf{x}}, \tilde{\mathbf{y}})]	+	\frac{\mu_\star}{2}\Exs[\distancemetric(\mathbf{x}_t, \mathbf{y}_t; \tilde{\mathbf{x}}, \tilde{\mathbf{y}})]
}
\\&\quad\,
-
(1-\alpha_t)\left(\Exs[\Vquantity(\mathbf{x}\avg{t-\frac32}, \mathbf{y}\avg{t-\frac32} \mid \tilde{\mathbf{x}}, \tilde{\mathbf{y}})]	+	\frac{\mu_\star}{2}\Exs[\distancemetric(\mathbf{x}_{t-1}, \mathbf{y}_{t-1}; \tilde{\mathbf{x}}, \tilde{\mathbf{y}})]\right)
\\&\le
-
\frac{\mu_\star}{2}\left(
\quantile - 2\alpha_t - \left(\tfrac{\smoothFG}{\mu_\star} + \tfrac{(1+\beta)\smoothBLin^2}{\mu_\star^2}\right)\alpha_t^2
\right)\Exs[\distancemetric(\mathbf{x}_{t-\frac12}, \mathbf{y}_{t-\frac12}; \mathbf{x}_{t-1}, \mathbf{y}_{t-1})]
\\&\quad\,
+
\frac{\alpha_t^2}{2\mu_\star} \left(
\tfrac{1}{1-\quantile}\stdFG^2	+	(2+\tfrac{1}{\beta})\stdBLin^2
\right)
.
\end{align*}

\paragraph{Step 3.}
\blue{
By the definition $\alpha_t$ we have
$
\quantile - 2\alpha_t - \left(\tfrac{\smoothFG}{\mu_\star} + \tfrac{(1+\beta)\smoothBLin^2}{\mu_\star^2}\right)\alpha_t^2
\ge
0
$, so we obtain regularity condition $
\alpha_t
\le
\bar{\alpha}
=
\frac{\quantile}{1+\sqrt{1 + \quantile\left(\tfrac{\smoothFG}{\mu_\star} + \tfrac{(1+\beta)\smoothBLin^2}{\mu_\star^2}\right)}}
$ of Theorem \ref{theo_S_AG_EG_Str}.
}%
Since we assumed both $F$ and $G$ are nonstrongly convex and $H$ is a $\mu_\star$-strongly-convex-$\mu_\star$-strongly-concave isotropic quadratic, this implies
$$\begin{aligned}
\lefteqn{
\Exs[\Vquantity(\mathbf{x}\avg{t-\frac12}, \mathbf{y}\avg{t-\frac12} \mid \tilde{\mathbf{x}}, \tilde{\mathbf{y}})]		+	\frac{\mu_\star}{2}\Exs[\distancemetric(\mathbf{x}_t, \mathbf{y}_t; \tilde{\mathbf{x}}, \tilde{\mathbf{y}})]
}
\\&\le
(1-\alpha_t)\left(
\Exs[\Vquantity(\mathbf{x}\avg{t-\frac32}, \mathbf{y}\avg{t-\frac32} \mid \tilde{\mathbf{x}}, \tilde{\mathbf{y}})]	+	\frac{\mu_\star}{2}\Exs[\distancemetric(\mathbf{x}_{t-1}, \mathbf{y}_{t-1}; \tilde{\mathbf{x}}, \tilde{\mathbf{y}})]
\right)
+
\frac{3\alpha_t^2}{2\mu_\star} \sigma^2
.
\end{aligned}$$
Plugging in $\tilde{\mathbf{x}} = \oholder_{\mathbf{x}}^\star, \tilde{\mathbf{y}} = \oholder_{\mathbf{y}}^\star$ gives
$$\begin{aligned}
&
\Exs[\Vquantity(\tilde{\mathbf{x}}, \tilde{\mathbf{y}} \mid \oholder_{\mathbf{x}}^\star, \oholder_{\mathbf{y}}^\star)]
\\&=
F(\tilde{\mathbf{x}}) + G(\tilde{\mathbf{y}}) - F(\oholder_{\mathbf{x}}^\star) - G(\oholder_{\mathbf{y}}^\star)
+
\langle \nabla_{\mathbf{x}} H(\oholder_{\mathbf{x}}^\star, \oholder_{\mathbf{y}}^\star),		\tilde{\mathbf{x}} - \oholder_{\mathbf{x}}^\star \rangle
+
\langle -\nabla_{\mathbf{y}} H(\oholder_{\mathbf{x}}^\star, \oholder_{\mathbf{y}}^\star),	\tilde{\mathbf{y}} - \oholder_{\mathbf{y}}^\star \rangle
\\&\ge
\langle\nabla F(\oholder_{\mathbf{x}}^\star) + \nabla_{\mathbf{x}} H(\oholder_{\mathbf{x}}^\star, \oholder_{\mathbf{y}}^\star),		\tilde{\mathbf{x}} - \oholder_{\mathbf{x}}^\star \rangle
+
\langle\nabla G(\oholder_{\mathbf{y}}^\star) - \nabla_{\mathbf{y}} H(\oholder_{\mathbf{x}}^\star, \oholder_{\mathbf{y}}^\star),	\tilde{\mathbf{y}} - \oholder_{\mathbf{y}}^\star \rangle
=
0
,
\end{aligned}$$
and also
$$\begin{aligned}
&
\Exs[\Vquantity(\tilde{\mathbf{x}}, \tilde{\mathbf{y}} \mid \oholder_{\mathbf{x}}^\star, \oholder_{\mathbf{y}}^\star)]
\\&\le
\langle\nabla F(\oholder_{\mathbf{x}}^\star) + \nabla_{\mathbf{x}} H(\oholder_{\mathbf{x}}^\star, \oholder_{\mathbf{y}}^\star),		\tilde{\mathbf{x}} - \oholder_{\mathbf{x}}^\star \rangle
+
\langle\nabla G(\oholder_{\mathbf{y}}^\star) - \nabla_{\mathbf{y}} H(\oholder_{\mathbf{x}}^\star, \oholder_{\mathbf{y}}^\star),	\tilde{\mathbf{y}} - \oholder_{\mathbf{y}}^\star \rangle
+
\tfrac{\smoothFG}{2} \distancemetric(\tilde{\mathbf{x}}, \tilde{\mathbf{y}}; \oholder_{\mathbf{x}}^\star, \oholder_{\mathbf{y}}^\star)
\\&=
\tfrac{\smoothFG}{2} \distancemetric(\tilde{\mathbf{x}}, \tilde{\mathbf{y}}; \oholder_{\mathbf{x}}^\star, \oholder_{\mathbf{y}}^\star)
,
\end{aligned}$$
so (by the fact that $\mathbf{x}\avg{-\frac12} = \mathbf{x}_0$ and $\mathbf{y}\avg{-\frac12} = \mathbf{y}_0$)
$$\begin{aligned}
\lefteqn{
\frac{\mu_\star}{2}\Exs[\distancemetric(\mathbf{x}_t, \mathbf{y}_t; \oholder_{\mathbf{x}}^\star, \oholder_{\mathbf{y}}^\star)]
\le
\Exs[\Vquantity(\mathbf{x}\avg{t-\frac12}, \mathbf{y}\avg{t-\frac12} \mid \oholder_{\mathbf{x}}^\star, \oholder_{\mathbf{y}}^\star)	+	\frac{\mu_\star}{2}\Exs[\distancemetric(\mathbf{x}_t, \mathbf{y}_t; \oholder_{\mathbf{x}}^\star, \oholder_{\mathbf{y}}^\star)]
}
\\&\le
\left(
\Vquantity(\mathbf{x}\avg{-\frac12}, \mathbf{y}\avg{-\frac12} \mid \oholder_{\mathbf{x}}^\star, \oholder_{\mathbf{y}}^\star) 	+	\frac{\mu_\star}{2} \distancemetric(\mathbf{x}_0, \mathbf{y}_0; \oholder_{\mathbf{x}}^\star, \oholder_{\mathbf{y}}^\star)
\right)\prod_{\tau=1}^t (1-\alpha_\tau)
+
\sum_{\tau=1}^t \frac{3\alpha_\tau^2}{2\mu_\star}\left[\prod_{\tau’=\tau+1}^t (1-\alpha_{\tau’})\right] \sigma^2
\\&\le
\distancemetric(\mathbf{x}_0, \mathbf{y}_0; \oholder_{\mathbf{x}}^\star, \oholder_{\mathbf{y}}^\star)\frac{\smoothFG+\mu_\star}{2}\prod_{\tau=1}^t (1-\alpha_\tau)
+
\frac{3\sigma^2}{2\mu_\star}\sum_{\tau=1}^t \alpha_\tau^2\prod_{\tau’=\tau+1}^t (1-\alpha_{\tau’})
.
\end{aligned}$$
Dividing both sides by $\frac{\mu_\star}{2}$ gives \eqref{S_AG_EG_Str} and our theorem.

\pb\subsection{Proof of Theorem \ref{theo_S_AG_EG_Bilinear}}\label{sec_proof,theo_S_AG_EG_Bilinear}
Before the proof we first adopt the scaling reduction argument as in \S\ref{sec_scalingreduction}, to argue that we only need to prove the result for the case of bilinear games centered at zero, i.e.~$
F(\mathbf{x}) = 0 = G(\mathbf{y})
$ where from \eqref{params} we have $
\smoothFG = \strcvx = \mu_F = 0
$.
We set the iteration symbol $\mathbf{z} \equiv
\begin{bmatrix}
\hat{\mathbf{x}}
\\
\hat{\mathbf{y}}
\end{bmatrix}
= 
\begin{bmatrix}
\mathbf{x} - \oholder_{\mathbf{x}}^\star
\\ 
\mathbf{y} - \oholder_{\mathbf{y}}^\star
\end{bmatrix}
$ and also $
\hat{\mathscr{F}}(\hat{\mathbf{x}},\hat{\mathbf{y}})
=
\hat{\mathbf{x}}^\top \mathbf{B}\hat{\mathbf{y}}
$, with $
\hat{\mathscr{F}}(\hat{\mathbf{x}},\hat{\mathbf{y}})
$ being equal to $\mathscr{F}(\mathbf{x},\mathbf{y})$ defined as in \eqref{BilinearF_stochastic} up to an additive constant.
Our scaling-reduction argument hence applies.

\begin{proof}[Proof of Theorem \ref{theo_S_AG_EG_Bilinear}]
From the update rule we have
\begin{subequations}\begin{align}
\mathbf{z}_{t-\frac12}
&=
\mathbf{z}_{t-1} - \stepsize\mathbf{J}\mathbf{z}_{t-1}
+ \eta \boldsymbol{\varepsilon}_{t-\frac12}
,
\label{eq_update_ibbi_a}
\\
\mathbf{z}\avg{t-\frac12}
&= 
\tfrac{t-1}{t+1}\mathbf{z}\avg{t-\frac32} + \tfrac{2}{t+1}\mathbf{z}_{t-\frac12}
,
\label{eq_update_ibbi_b}
\\
\mathbf{z}_t
&=
\mathbf{z}_{t-1} - \stepsize\mathbf{J}\mathbf{z}_{t-\frac12}
+ \eta \boldsymbol{\varepsilon}_t
.
\label{eq_update_ibbi_c}
\end{align}
\end{subequations}
Note the $[\mathbf{x}\Xtrap{t}; \mathbf{y}\Xtrap{t}]$ sequence becomes irrelevant in this update
\,\,
   skew-symmetric with $\mathbf{J}^\top = -\mathbf{J}$, so $\mathbf{J}^2 = -\mathbf{J}^\top\mathbf{J}$ is symmetric and negative semidefinite.
We proceed with the proof in steps:

\paragraph{Step 1.}
We target to show the last-iterate bound
\begin{equation}\label{last_iterate}
\Exs\|\mathbf{z}_t\|^2 
\le
\Exs\|\mathbf{z}_0\|^2 
+ 
2t\eta^2\stdBLin^2
\end{equation}
Note \eqref{eq_update_ibbi_a} and~\eqref{eq_update_ibbi_c} together gives
\begin{align}\label{eq_update_ibbi_combined}
\mathbf{z}_t
&=
\left(
\mathbf{I} - \stepsize \mathbf{J}
+ 
\stepsize^2 \mathbf{J}^2
\right) \mathbf{z}_{t-1}
- \eta^2 \mathbf{J} \boldsymbol{\varepsilon}_{t-\frac12}
+ \eta \boldsymbol{\varepsilon}_t
\end{align}
Taking squared norm on both sides of \eqref{eq_update_ibbi_combined}, we have when $
\stepsize	\le	\frac{1}{\sqrt{\lambda_{\max}(\mathbf{B}^\top \mathbf{B})}}
$, $\mathbf{z}_t$ does not expand in Euclidean norm (noiseless), so
\beq\label{ztnoexpand}
\begin{aligned}
&
\Exs\|\mathbf{z}_t\|^2 
=
\Exs\left[
(\mathbf{z}_{t-1})^\top\left(
\mathbf{I} + \stepsize^2 \mathbf{J}^2
+
\stepsize^4 \mathbf{J}^4
\right) \mathbf{z}_{t-1}
\right]
+
\Exs\left\|
- \eta^2 \mathbf{J} \boldsymbol{\varepsilon}_{t-\frac12}
+ \eta \boldsymbol{\varepsilon}_t
\right\|^2
\\&\le
\Exs\|\mathbf{z}_{t-1}\|^2 
+
\Exs\left\|\eta^2 \mathbf{J} \boldsymbol{\varepsilon}_{t-\frac12}\right\|^2
+ 
\Exs\left\|\eta \boldsymbol{\varepsilon}_t\right\|^2
\le
\Exs\|\mathbf{z}_{t-1}\|^2 
+ 
\eta^2\left(1+\eta^2\lambda_{\max}(\mathbf{B}^\top\mathbf{B})\right)\stdBLin^2
\le
\Exs\|\mathbf{z}_{t-1}\|^2 
+ 
2\eta^2\stdBLin^2
.
\end{aligned}\eeq
Recursively applying the above concludes \eqref{last_iterate}.

\paragraph{Step 2.}
We start from the update rule \eqref{eq_update_ibbi_b} which implies $
(t+1)t\mathbf{z}\avg{t-\frac12}
= 
t(t-1)\mathbf{z}\avg{t-\frac32} + 2t\mathbf{z}_{t-\frac12}
$ holds for $t=1,\dots,\Tholder$, so
$$
(\Tholder+1)\Tholder\mathbf{z}\avg{\Tholder-\frac12} = 2\sum_{t=1}^\Tholder t\mathbf{z}_{t-\frac12}
\quad\Rightarrow\quad
\mathbf{z}\avg{\Tholder-\frac12} = \frac{2}{(\Tholder+1)\Tholder}\sum_{t=1}^\Tholder t\mathbf{z}_{t-\frac12}
.
$$
Using this to analyze our algorithm:
$$
t \mathbf{z}_t - (t-1) \mathbf{z}_{t-1} - \mathbf{z}_{t-1}	
=	
t (\mathbf{z}_t - \mathbf{z}_{t-1})	
=	
-\eta \mathbf{J}\left[ t \mathbf{z}_{t-\frac12} \right]
+ \eta t\boldsymbol{\varepsilon}_t
,
$$
so telescoping gives
$$
\Tholder\mathbf{z}_\Tholder - \sum_{t=1}^\Tholder \mathbf{z}_{t-1}	
=	
-\eta \mathbf{J}\sum_{t=1}^\Tholder t \mathbf{z}_{t-\frac12}
+ \eta \sum_{t=1}^\Tholder t\boldsymbol{\varepsilon}_t
,
$$
which yields
\beq\label{last_avg_relation}
\mathbf{z}\avg{\Tholder-\frac12}
=
\frac{2}{(\Tholder+1)\Tholder}\sum_{t=1}^\Tholder t \mathbf{z}_{t-\frac12}
=
\frac{2}{-\eta (\Tholder+1)\Tholder} \mathbf{J}^{-1}\left(
\Tholder\mathbf{z}_\Tholder - \sum_{t=1}^\Tholder \mathbf{z}_{t-1}
- 
\eta \sum_{t=1}^\Tholder t\boldsymbol{\varepsilon}_t
\right)
.
\eeq
Obviously the least singular value of the matrix $
\mathbf{J}
$ can be lower-bounded as $
\sigma_{\min}(\mathbf{J})
\ge
\sqrt{\lambda_{\min}(\mathbf{B}\mathbf{B}^\top)}
$.
We conclude from \eqref{last_avg_relation} along with Young’s inequality that
$$\begin{aligned}
&
\lambda_{\min}(\mathbf{B}\mathbf{B}^\top)\Exs\left\|\mathbf{z}\avg{\Tholder-\frac12}\right\|^2
\le
\Exs\left\|\mathbf{J}\mathbf{z}\avg{\Tholder-\frac12}\right\|^2
\\&=
(1+\gamma)\frac{4}{\eta^2 (\Tholder+1)^2\Tholder^2}
\Exs\left\|\sum_{t=1}^\Tholder \left(\mathbf{z}_\Tholder - \mathbf{z}_{t-1}\right)\right\|^2
+
(1+\tfrac{1}{\gamma})\frac{4}{\eta^2 (\Tholder+1)^2\Tholder^2}\Exs\left\|\eta \sum_{t=1}^\Tholder t\boldsymbol{\varepsilon}_t\right\|^2
\\&\equiv
(1+\gamma)\mbox{I} 
+ 
(1+\tfrac{1}{\gamma})\mbox{II}
,
\end{aligned}$$
where applying the last-iterate bound \eqref{last_iterate} together with some elementary estimates leads to
$$\begin{aligned}
\mbox{I}
&\le
\frac{4}{\eta^2 (\Tholder+1)^2\Tholder^2}
\cdot
\Tholder\sum_{t=1}^\Tholder \left[
2\Exs\left\|\mathbf{z}_\Tholder\right\|^2 + 2\Exs\left\|\mathbf{z}_{t-1}\right\|^2
\right]
\\&\le
\frac{4}{\eta^2 (\Tholder+1)^2\Tholder^2}
\cdot
\Tholder\sum_{t=1}^\Tholder \left[
4\Exs\|\mathbf{z}_0\|^2 + 4(\Tholder+t-1)\eta^2\stdBLin^2
\right]
\\&\le
\frac{16\Exs\|\mathbf{z}_0\|^2 + 24\eta^2\stdBLin^2\Tholder}{\eta^2 (\Tholder+1)^2}
\le
\frac{16\lambda_{\max}(\mathbf{B}^\top \mathbf{B})\Exs\|\mathbf{z}_0\|^2}{(\Tholder+1)^2}
+
\frac{24\stdBLin^2}{\Tholder+1}
,
\end{aligned}$$
and, using the property of square-integrable martingales,
$$\begin{aligned}
\mbox{II}
&\le
\frac{4}{\eta^2 (\Tholder+1)^2\Tholder^2}\Exs\left\|\eta \sum_{t=1}^\Tholder t\boldsymbol{\varepsilon}_t\right\|^2
=
\frac{4}{\eta^2 (\Tholder+1)^2\Tholder^2}
\cdot
\eta^2 \sum_{t=1}^\Tholder t^2\Exs\left\|\boldsymbol{\varepsilon}_t\right\|^2
\\&\le
\frac{4\stdBLin^2}{\eta^2 (\Tholder+1)^2\Tholder^2}
\cdot
\eta^2 \frac{\Tholder(\Tholder+\tfrac12)(\Tholder+1)}{3}
\le
\frac{4\stdBLin^2}{3\Tholder}
.
\end{aligned}$$
To summarize we have for arbitrary $\gamma\in (0,\infty)$
$$\begin{aligned}
\lambda_{\min}(\mathbf{B}\mathbf{B}^\top)\Exs\left\|\mathbf{z}\avg{\Tholder-\frac12}\right\|^2
\le
(1+\gamma)\left(
\frac{16\lambda_{\max}(\mathbf{B}^\top \mathbf{B})\Exs\|\mathbf{z}_0\|^2}{(\Tholder+1)^2}
+
\frac{24\stdBLin^2}{\Tholder+1}
\right)
+ 
(1+\tfrac{1}{\gamma})\frac{4\stdBLin^2}{3\Tholder}
.
\end{aligned}$$
Optimizing $\gamma$ gives along with $\sqrt{a+b}\le \sqrt{a}+\sqrt{b}$ for nonnegatives $a$ and $b$:
$$\begin{aligned}
&
\sqrt{\lambda_{\min}(\mathbf{B}\mathbf{B}^\top)}
\sqrt{\Exs\left\|\mathbf{z}\avg{\Tholder-\frac12}\right\|^2}
\le
\sqrt{
\frac{16\lambda_{\max}(\mathbf{B}^\top \mathbf{B})\Exs\|\mathbf{z}_0\|^2}{(\Tholder+1)^2}
+
\frac{24\stdBLin^2}{\Tholder+1}
}
+ 
\sqrt{\frac{4\stdBLin^2}{3\Tholder}}
\\&\le
\sqrt{
\frac{16\lambda_{\max}(\mathbf{B}^\top \mathbf{B})\Exs\|\mathbf{z}_0\|^2}{(\Tholder+1)^2}
}
+
\sqrt{\frac{24\stdBLin^2}{\Tholder+1}}
+ 
\sqrt{\frac{4\stdBLin^2}{3\Tholder}}
.
\end{aligned}$$
Dividing both sides by $
\sqrt{\lambda_{\min}(\mathbf{B}\mathbf{B}^\top)}
$ and taking squares conclude \eqref{AG_EG_Bilinear} and hence the theorem.
\end{proof}

\pb\section{Proof of auxiliary lemmas}\label{sec_proofaux}
\subsection{Proof of Lemma \ref{lemm_PRecursion}}\label{sec_proof,lemm_PRecursion}
The analysis in this subsection is partially motivated by Lemma 2 of \citet{chen2017accelerated}.
\begin{proof}[Proof of Lemma \ref{lemm_PRecursion}]
By definition of $\boneholder, \btwoholder$, we have for any $\arbholder\in \reals^d$
\begin{align}
\label{eqnMP2u}
\langle \boneholder,	\foneholder - \arbholder\rangle 
&= 
\frac{1}{2} \left[
\|\thetaholder - \arbholder\|^2 - \|\thetaholder - \foneholder\|^2 - \|\foneholder - \arbholder\|^2
\right]
,
\\
\label{eqnMP1}
\langle \btwoholder,	\ftwoholder - \arbholder \rangle 
&=
\frac{1}{2} \left[
\|\thetaholder - \arbholder\|^2 - \|\thetaholder - \ftwoholder\|^2 - \|\ftwoholder - \arbholder\|^2
\right]
.
\end{align}
Specifically, letting $\arbholder = \ftwoholder$ in \eqref{eqnMP2u} we have
\begin{align}\label{eqnMP2}
\langle \boneholder,	\foneholder - \ftwoholder \rangle 
&= 
\frac{1}{2} \left[
\|\thetaholder - \ftwoholder\|^2 - \|\thetaholder - \foneholder\|^2 - \|\foneholder - \ftwoholder\|^2
\right]
.
\end{align}
Now, combining inequalities (\ref{eqnMP1}) and (\ref{eqnMP2}) we have
$$
\langle \btwoholder,	\ftwoholder - \arbholder \rangle 
+ 
\langle \boneholder,	\foneholder - \ftwoholder \rangle 
\le
\frac{1}{2} \left[
\|\thetaholder - \arbholder\|^2 - \|\ftwoholder - \arbholder\|^2- \|\thetaholder - \foneholder\|^2 - \|\foneholder - \ftwoholder\|^2
\right]
,
$$
which in turn gives 
$$\begin{aligned}
\langle \btwoholder,	\foneholder - \arbholder \rangle 
\le
\langle \btwoholder - \boneholder,	\foneholder - \ftwoholder \rangle 
+ 
\frac{1}{2} \left[
\|\thetaholder - \arbholder\|^2 - \|\ftwoholder - \arbholder\|^2 - \|\thetaholder - \foneholder\|^2 - \|\foneholder - \ftwoholder\|^2
\right]
.
\end{aligned}$$
An application of the Young and Cauchy-Schwartz inequalities gives
\beq\label{tmp1}
\begin{aligned}
\langle\btwoholder,	\foneholder - \arbholder \rangle
&\le
\|\btwoholder - \boneholder\| \|\foneholder - \ftwoholder\| 
+ 
\frac{1}{2} \left[
\|\thetaholder - \arbholder\|^2 - \|\ftwoholder - \arbholder\|^2- \|\thetaholder - \foneholder\|^2 - \|\foneholder - \ftwoholder\|^2
\right]
\\&\le
\frac{1}{2}\|\btwoholder - \boneholder\|^2 + \frac{1}{2}\|\foneholder - \ftwoholder\|^2 + \frac{1}{2} \left[
\|\thetaholder - \arbholder\|^2 - \|\ftwoholder - \arbholder\|^2- \|\thetaholder - \foneholder\|^2 - \|\foneholder - \ftwoholder\|^2
\right]
\\&=
\frac{1}{2}\|\btwoholder - \boneholder\|^2 + \frac{1}{2} \left[
\|\thetaholder - \arbholder\|^2 - \|\ftwoholder - \arbholder\|^2 - \|\thetaholder - \foneholder\|^2
\right]
.
\end{aligned}\eeq
This establishes \eqref{eqnPRecursion} and hence Lemma \ref{lemm_PRecursion}.
\end{proof}

\pb\subsection{Proof of Lemma \ref{lemm_QuanBdd}}\label{sec_proof,lemm_QuanBdd}
\begin{proof}[Proof of Lemma \ref{lemm_QuanBdd}]
It is straightforward to verify that $F(\mathbf{x})$ and $G(\mathbf{y})$ are $\smoothFG$-smooth and $\strcvx$-strongly convex.
For the rest of this proof, we observe that the saddle definition of $\oholder_{\mathbf{x}}^\star, \oholder_{\mathbf{y}}^\star$ satisfies the first-order stationary condition for problem \eqref{problemopt_stochastic}:
\beq\label{FOSC}
\nabla_{\mathbf{x}} \mathscr{F}(\oholder_{\mathbf{x}}^\star, \oholder_{\mathbf{y}}^\star)
=
\nabla F(\oholder_{\mathbf{x}}^\star) 
+ \nabla_{\mathbf{x}} H(\oholder_{\mathbf{x}}^\star, \oholder_{\mathbf{y}}^\star)
=
0
,\qquad
\nabla_{\mathbf{y}} \mathscr{F}(\oholder_{\mathbf{x}}^\star, \oholder_{\mathbf{y}}^\star)
=
\nabla_{\mathbf{y}} H(\oholder_{\mathbf{x}}^\star, \oholder_{\mathbf{y}}^\star)
- \nabla G(\oholder_{\mathbf{y}}^\star)
=
0
.
\eeq
Since both $f(\mathbf{x})$ and $g(\mathbf{y})$ are $\strcvx$-strongly convex, we have
$$\begin{aligned}
&
F(\mathbf{x}) - F(\oholder_{\mathbf{x}}^\star) 
+ \left\langle \nabla_{\mathbf{x}} H(\oholder_{\mathbf{x}}^\star, \oholder_{\mathbf{y}}^\star), \mathbf{x} - \oholder_{\mathbf{x}}^\star \right\rangle 
\ge
\left\langle
\nabla F(\oholder_{\mathbf{x}}^\star), \mathbf{x} - \oholder_{\mathbf{x}}^\star
\right\rangle
+
\frac{\strcvx}{2}\left\| \mathbf{x} - \oholder_{\mathbf{x}}^\star\right\|^2
+
\left\langle
\nabla_{\mathbf{x}} H(\oholder_{\mathbf{x}}^\star, \oholder_{\mathbf{y}}^\star), \mathbf{x} - \oholder_{\mathbf{x}}^\star
\right\rangle
\\&=
\left\langle
\nabla F(\oholder_{\mathbf{x}}^\star) + \nabla_{\mathbf{x}} H(\oholder_{\mathbf{x}}^\star, \oholder_{\mathbf{y}}^\star), \mathbf{x} - \oholder_{\mathbf{x}}^\star
\right\rangle
+
\frac{\strcvx}{2}\left\| \mathbf{x} - \oholder_{\mathbf{x}}^\star\right\|^2
=
\frac{\strcvx}{2}\left\| \mathbf{x} - \oholder_{\mathbf{x}}^\star\right\|^2
,
\end{aligned}$$
and
$$\begin{aligned}
&
G(\mathbf{y}) - G(\oholder_{\mathbf{y}}^\star) 
-
\left\langle \nabla_{\mathbf{y}} H(\oholder_{\mathbf{x}}^\star, \oholder_{\mathbf{y}}^\star), \mathbf{y} - \oholder_{\mathbf{y}}^\star \right\rangle
\ge
\left\langle
\nabla G(\oholder_{\mathbf{y}}^\star), \mathbf{y} - \oholder_{\mathbf{y}}^\star
\right\rangle
+
\frac{\strcvx}{2} \left\|\mathbf{y} - \oholder_{\mathbf{y}}^\star\right\|^2
-
\left\langle
\nabla_{\mathbf{y}} H(\oholder_{\mathbf{x}}^\star, \oholder_{\mathbf{y}}^\star), \mathbf{y} - \oholder_{\mathbf{y}}^\star
\right\rangle
\\&=
-\left\langle
\nabla_{\mathbf{y}} H(\oholder_{\mathbf{x}}^\star, \oholder_{\mathbf{y}}^\star) - \nabla G(\oholder_{\mathbf{y}}^\star), \mathbf{y} - \oholder_{\mathbf{y}}^\star
\right\rangle
+
\frac{\strcvx}{2} \left\|\mathbf{y} - \oholder_{\mathbf{y}}^\star\right\|^2
=
\frac{\strcvx}{2} \left\|\mathbf{y} - \oholder_{\mathbf{y}}^\star\right\|^2
,
\end{aligned}$$
where in both of the two displays, the inequality holds due to the $\strcvx$-strong convexity of $F$ and $G$, and the equality holds due to the first-order stationary condition \eqref{FOSC}.
This completes the proof.
\end{proof}

\pb\subsection{Proof of Lemma \ref{lemm_ssproperties}}\label{sec_proof,lemm_ssproperties}
\begin{proof}[Proof of Lemma \ref{lemm_ssproperties}]
Items (i)---(iii) are straightforward.
For the proof of \eqref{stepsize_cond} in item (iv), we note that $
\stepsize_t
=
\bar{\stepsize}_t(\sigma;\Tholder,\Cholder,\quantile,\beta)
\le
\frac{t}{\frac{2}{\quantile}\smoothFG + \sqrt{\tfrac{1+\beta}{\quantile}}\smoothBLin t}
\le
\frac{1}{\sqrt{\tfrac{1+\beta}{\quantile}}\smoothBLin}
$ which gives
$$
\quantile - \frac{2\smoothFG}{t+1} \stepsize_t - (1+\beta)\smoothBLin^2 \stepsize_t^2
\ge
\frac{\quantile}{t}\left(t - \left(\frac{2}{\quantile}\smoothFG + \sqrt{\tfrac{1+\beta}{\quantile}}\smoothBLin t\right)\stepsize_t\right)
\ge
0
,
$$
and hence completes the proof.
\end{proof}

\end{document}